\documentclass[letter,11pt]{amsart}
\usepackage[toc,page]{appendix}
\usepackage{amssymb} 
\usepackage{bbm}
\usepackage{graphicx}
\usepackage{color}
\usepackage[font=footnotesize]{caption}
\usepackage{mathtools}
\usepackage{pinlabel}
\usepackage[shortlabels]{enumitem}
\usepackage{graphicx}
\graphicspath{ {images/} }
\usepackage{comment}
\usepackage{tikz-cd}
\usepackage{subfigure}

\usepackage{hyperref}

\usepackage{float}
\usepackage{calc}
\def\thetitle{{ . }}
\hypersetup{
	pdftitle=  \thetitle,
	pdfauthor=  {}  
}

\newtheorem{THM}{Theorem}[section]

\newtheorem*{THM*}{Theorem}
\newtheorem{LEM}[THM]{Lemma}

\newtheorem{prop}[THM]{Proposition}
\newtheorem{que}{Question}

\newtheorem{defn}[THM]{Definition}
\newtheorem{conv}[THM]{Convention}

\newtheorem{nota}[THM]{Notation}
\newtheorem{cons}[THM]{Construction}
\newtheorem{assume}[THM]{Assumption}
\newtheorem{fact}[THM]{Fact}
\newtheorem{case}{Case}
\newtheorem{case1}{Case}
\newtheorem*{orientable double cover}{\textbf{The orientable double cover}}

\theoremstyle{remark}
\newtheorem{exmp}[THM]{Example}
\newtheorem{rmk}[THM]{\textbf{Remark}}

\usepackage{tikz}
\usetikzlibrary{calc,decorations.markings}
\usetikzlibrary{shapes,snakes}

\theoremstyle{definition}
\newtheorem*{defn*}{Definition}

\newcommand\N{\mathbb{N}}
\newcommand\Q{\mathbb{Q}}
\newcommand\R{\mathbb{R}}

\newcommand\Z{\mathbb{Z}}
\newcommand\B{\mathcal{B}}
\newcommand\PP{\mathcal{P}}
\newcommand{\F}{\mathcal{F}}
\newcommand{\Fa}{\mathcal{F}^{+}}
\newcommand{\Fb}{\mathcal{F}^{-}}
\newcommand{\Fs}{\mathcal{F}^{s}}
\newcommand{\Fu}{\mathcal{F}^{u}}
\newcommand{\Ea}{\mathcal{E}^{+}}
\newcommand{\Eb}{\mathcal{E}^{-}}
\newcommand{\Ors}{\mathcal{O}^{s}}
\newcommand{\Oru}{\mathcal{O}^{u}}
\newcommand{\La}{\Lambda^{+}}
\newcommand{\Lb}{\Lambda^{-}}

\newcommand{\Ia}{I^{+}}
\newcommand{\Ib}{I^{-}}
\newcommand{\Ir}{I^{r}}

\newcommand{\lame}{Ends}
\newcommand{\DD}{\mathcal{D}}
\newcommand{\RR}{\mathcal{R}}

\newcommand{\Or}{\mathcal{O}}
\newcommand{\w}{\widetilde}

\begin{document}
\title{Reconstruction of Anosov flows from infinity }


\author{Hyungryul Baik}
\address{Department of Mathematical Sciences, KAIST,  
	291 Daehak-ro, Yuseong-gu, Daejeon 34141, South Korea }
\email{hrbaik@kaist.ac.kr}
\author{Chenxi Wu}
\address{Department of Mathematics, University of Wisconsin at Madison, Madison WI 53703 }
\email{cwu367@wisc.edu}
\author{Bojun Zhao}
\address{D\'epartement de math\'ematiques, Universit\'e du Qu\'ebec \`a Montr\'eal,
201 President Kennedy Avenue, Montr\'eal, QC, Canada H2X 3Y7}
\email{bojunzha@buffalo.edu}

\maketitle
\begin{abstract}

    
    Every pseudo-Anosov flow $\phi$ in a closed $3$-manifold $M$ gives rise to an action of $\pi_1(M)$ on a circle $S^{1}_{\infty}(\phi)$ from infinity \cite{Fen12},
    with a pair of invariant \emph{almost} laminations.
    From certain actions on $S^{1}$ with invariant almost laminations,
    we reconstruct flows and manifolds realizing these actions,
    including all orientable transitive pseudo-Anosov flows in closed $3$-manifolds.
    Our construction provides a geometry model for such flows and manifolds 
    induced from $\DD \times \DD$,
    where $\DD$ is the Poincar\'e disk with $\partial \DD$ identified with $S^{1}_{\infty}(\phi)$.
    In addition, our result applies to Cannon’s conjecture under the assumption that certain group-equivariant sphere-filling Peano curve exists, which offers a description of orientable quasigeodesic pseudo-Anosov flows in hyperbolic $3$-manifolds in terms of group actions on $\partial \mathbb{H}^{3} \times \partial \mathbb{H}^{3} \times \partial \mathbb{H}^{3}$.
    \end{abstract}

\section{Introduction}
In the fields of geometric topology and low-dimensional topology, the construction of the space at infinity and the analysis of the extended group action at infinity have proven to be highly fruitful when a group action is present on a given space. One of the most notable examples of this approach was Thurston's compactification of Teichm\"uller space via the space of projective measured foliations on the surface. The detail can be found in \cite{fathi2021thurston}. 

By pursuing this line of thought further, one can ask how much information about the space is retained in the space at infinity and whether the space can be reconstructed from infinity. For instance, a seminal work of Tukia \cite{tukia1988homeomorphic}, Gabai \cite{gabai1992convergence}, Casson-Jungreis \cite{casson1994convergence} tells us that if one has a convergence group action on the circle, this group is a Fuchsian group, and the circle is group-equivariantly identified with the ideal boundary of the hyperbolic plane. \cite{baik2015fuchsian}, \cite{baik2021characterization} are other examples of characterization of the Fuchsian group actions on the ideal boundary (via invariant circle laminations). More recently, \cite{baik2022groups} provided an example of this in dimension 3. Namely, if one has a group action on the circle with a veering pair of laminations with no polygonal gaps, then the group must be the fundamental group of a 3-manifold with veering triangulation where the circle is identified with the space of ideal vertices of the triangulation in the universal cover. In this paper, we achieve this goal for 3-manifolds with Anosov or pseudo-Anosov flows. 

Here are some conventions that we use throughout the paper. Firstly, all $3$-manifolds are assumed to be orientable,
and all pseudo-Anosov flows are assumed to be smooth pseudo-Anosov flows, unless explicitly stated otherwise (see Subsection \ref{subsec: settings} for more information).
The stable and unstable foliations of a (pseudo-)Anosov flow in a $3$-manifold always refer to the weak stable and unstable foliations.
A (pseudo-)Anosov flow in a closed orientable $3$-manifold is called \emph{orientable} if its stable foliation is orientable.

Let $\phi$ be a pseudo-Anosov flow in a closed $3$-manifold $M$. Then $\phi$ lifted to the universal cover has a well-defined orbit space $\mathcal{O}(\phi) \cong \mathbb{R}^{2}$ with a pair of singular $1$-dimensional foliations $\Ors(\phi), \Oru(\phi)$ \cite{Fen94}, \cite{FM01}. Moreover, 
the orbit space $\mathcal{O}(\phi)$ has a canonical ideal boundary $S^{1}_{\infty}(\phi) \cong S^{1}$ \cite{Fen12}.
$\Ors(\phi), \Oru(\phi)$ can be canonically identified with
a pair of laminations on $S^{1}_{\infty}(\phi)$ with certain leaves removed, which we call \emph{almost} laminations (see Section \ref{sec:BBM} for relevant definitions and background).
We denote by $\Lambda^{s}_{\infty}(\phi), \Lambda^{u}_{\infty}(\phi)$ this pair of
almost laminations on $S^{1}_{\infty}(\phi)$.

The deck transformations of the universal cover of $M$ induce
an action of $\pi_1(M)$ on $\mathcal{O}(\phi)$ that preserves $\Ors(\phi), \Oru(\phi)$,
and they also induce
an action of $\pi_1(M)$ on $S^{1}_{\infty}(\phi)$ that preserves
$\Lambda^{s}_{\infty}(\phi), \Lambda^{u}_{\infty}(\phi)$.
These actions are referred to as the $\pi_1$-actions on
$\mathcal{O}(\phi)$ and $S^{1}_{\infty}(\phi)$ respectively.

Let $G$ be a finitely generated group, and let $\rho: G \to \text{Homeo}(S^{1})$ be a faithful action of $G$ on $S^{1}$ 
via homeomorphisms that preserves a pair of almost laminations $\Lambda^{+}, \Lambda^{-}$.
Suppose that $M$ is a closed orientable $3$-manifold with $\pi_1(M) = G$ and
$M$ admits a pseudo-Anosov flow $\phi$.
We say $(M,\phi)$ \emph{realizes} $\rho$ if $\pi_1(M)$ is isomorphic to $G$, and
there is a homeomorphism $h: S^{1} \to S^{1}_{\infty}(\phi)$ such that
the $\pi_1$-action on $S^{1}_{\infty}(\phi)$ is conjugate to $\rho$ by $h$ and
$h(\Lambda^{+}) = \Lambda^{s}_{\infty}(\phi)$,
$h(\Lambda^{-}) = \Lambda^{u}_{\infty}(\phi)$.

In \cite{BFM22},
Barthelm\'e, Frankel and Mann introduced the notion of a \emph{bifoliated plane}
$(\PP, \Fa, \Fb)$ to
describe the structure of $(\mathcal{O}(\phi), \Ors(\phi), \Oru(\phi))$,
and asked when
an Anosov-like action (Definition \ref{Anosov-like}) on $(\PP, \Fa, \Fb)$ is conjugate to the $\pi_1$-action on $(\mathcal{O}(\phi), \Ors(\phi), \Oru(\phi))$ for
certain pseudo-Anosov flow $\phi$ in a closed $3$-manifold $M$
\cite[Question 1]{BFM22}.
It is natural to ask 

\begin{que}\label{question 1}
Given an action of a group on $S^{1}$ with a pair of invariant almost laminations,
when can this action be realized by a closed $3$-manifold with a pseudo-Anosov flow, and how can we reconstruct this manifold and flow from the action on $S^{1}$?
\end{que}

Pseudo-Anosov flows have been extensively studied as dynamic objects, but a purely hyperbolic geometric description remains elusive. 
For a quasigeodesic pseudo-Anosov flow in a hyperbolic $3$-manifold, $S^{1}_{\infty}$ can be canonically identified with a sphere-filling Peano curve on $\partial \mathbb{H}^{3}$. 
Hence a straightforward reconstruction from $S^{1}_{\infty}$ could provide a geometric description by expressing the flow in terms of $\partial \mathbb{H}^{3}$.


In this paper,
we focus on Question \ref{question 1} for orientable pseudo-Anosov flows.
An almost lamination on $S^{1}$ is said to be orientable if 
its geometric realization has no ideal polygon with odd number of sides.
We offer a precise description for orientations on 
(the geometric realization of) an almost lamination below.

Let $\Lambda^{+}, \Lambda^{-}$ be a pair of almost laminations on $S^{1}$.
The pair $(\La, \Lb)$ is called \emph{bifoliar} if
there is a sufficiently nice bifoliated plane $(\PP,\Fa,\Fb)$ so that
$(S^{1},\La,\Lb)$ is the ideal boundary of $(\PP,\Fa,\Fb)$.
\cite{BBM} gave a characterization of bifoliar pairs of almost laminations that is suitable for the study of (pseudo-)Anosov flows in 3-manifolds. More precisely, following \cite{BBM}, we define a bifoliar pair of almost laminations as a pair of almost laminations satisfying all conditions of Theorem \ref{thm:BBM}. 

In Section \ref{sec:circle-lamination} we also include a method to characterize bifoliar pair of almost laminations by generalizing the method in \cite{baik2022groups}. 
We note that if $\La,\Lb$ are bifoliar, then any group acting on $S^{1}$ preserving $\La, \Lb$
can extend to a group action on the bifoliated plane completed from $(S^{1},\La,\Lb)$. 

For torsion-free groups acting on $S^{1}$ that preserve
a bifoliar pair of almost laminations on $S^{1}$,
we introduce a condition called \emph{flowable} 
so that all orientable transitive pseudo-Anosov flows in closed $3$-manifolds are realized from flowable group action on $S^{1}$ preserving 
a bifoliar pair of almost laminations on $S^{1}$.
Roughly speaking, an action of a group $G$ on $S^{1}$ with 
an invariant bifoliar pair of almost laminations is called \emph{flowable} if 
\begin{itemize}
\item if $g \in G - \{1\}$ fixes a leaf of $\La$ and a leaf of $\Lb$, then $g$ topologically stretches one of them and compresses along the other. The same behavior occurs if $g$ fixes a pair of ideal polygon gaps of $\La, \Lb$,
\item the stabilizer of an intersection between leaves of the laminations is either trivial or infinite cyclic, and 
\item for any two points on a leaf of one lamination, there does not exist an element $g \in G$ that moves both points by an arbitrarily small amount.
\end{itemize}
For a precise definition of a flowable action in terms of almost laminations, see Section \ref{sec:circle-lamination} (also Section \ref{sec:realizingtopflow} for a definition in terms of bifoliations).

For concreteness, we would like to describe our construction in terms of 
reduced pseudo-Anosov flows,
where a flow $\phi$ in 
a (possibly non-compact) $3$-manifold $M$ is a \emph{reduced pseudo-Anosov flow} if
$\phi$ satisfies all properties of smooth pseudo-Anosov flows except the restrictions on the metric of $M$ (see Definition \ref{reduced pseudo-Anosov flow}).
Similar to smooth pseudo-Anosov flows,
the reduced pseudo-Anosov flow $\phi$ also has
a well-defined orbit space $\Or(\phi)$ with invariant singular foliations 
$\Ors(\phi), \Oru(\phi)$,
as well as a well-defined ideal boundary $S^{1}_{\infty}(\phi)$ of $\Or(\phi)$ with
a pair of invariant almost laminations 
$\Lambda^{s}_{\infty}(\phi), \Lambda^{u}_{\infty}(\phi)$
induced from $\Ors(\phi),\Oru(\phi)$.

\begin{THM}\label{main theorem}
Let $\La, \Lb$ be a bifoliar pair of orientable almost laminations. 
Let $\rho: G \to \text{Homeo}(S^{1})$ be a faithful action of 
a torsion-free group $G$ on $S^{1}$ that
preserves $\La, \Lb$ and their orientations.

\begin{enumerate}[(a)]
\item Assume $\rho$ is flowable. Now we can construct a (possibly non-compact) orientable $3$-manifold $M(\rho)$ with 
a reduced (orientable) pseudo-Anosov flow $\phi(\rho)$ so that
$(M(\rho), \phi(\rho))$ realizes $\rho$.
\item If $\rho$ is conjugate to the $\pi_1$-action on $S^{1}_{\infty}(\phi_0)$ for some orientable transitive pseudo-Anosov flow $\phi_0$ in a closed $3$-manifold $M_0$,
then $M(\rho)$ is homeomorphic to $M_0$ and $\phi(\rho)$ is orbitally equivalent to $\phi_0$.
\end{enumerate}
\end{THM}
 
We note that for an orientable pseudo-Anosov flow $\phi_0$,
the flowable property holds for the $\pi_1$-action on $S^{1}_{\infty}(\phi_0)$ without the requirement of the transitivity of $\phi_0$.

\begin{rmk}
(a)
Let $\phi_0$ be any orientable smooth pseudo-Anosov flow or
an orientable topological Anosov flow in a closed orientable $3$-manifold.
If $\rho$ is conjugate to the $\pi_1$-action on $S^{1}_{\infty}(\phi_0)$,
then $\rho$ is flowable.

(b)
Under the assumption of Theorem \ref{main theorem},
we have a necessary and sufficient condition for 
$\rho$ being conjugate to the $\pi_1$-action of some orientable transitive pseudo-Anosov flow. See Remark \ref{sufficient and necessary condition} for details.
\end{rmk}

As the paper was being written, the authors were told that Barthelm\'e-Fenley-Mann had similar results in a forthcoming paper. 

\subsection{Construction of $M(\rho),\phi(\rho)$} \label{subsec:fromcircle}

In this section, we describe a construction of a 3-manifold with a flow from a flowable action on the circle. The construction can be obtained by combining Section \ref{sec:frombifoliatedplane} and Section \ref{sec:circle-lamination}, which we will outline below for the convenience of the reader.

Let $\Lambda$ be an almost lamination on $S^{1}$.
Recall that $\Lambda$ has a geometric realization in the Poincar\'e disk 
$\DD \cong \mathbb{H}^{2} \cup \partial \mathbb{H}^{2}$,
as a collection of disjoint geodesics of $\mathbb{H}^{2}$ 
(which is not necessarily a closed subset) 
whose restriction to  $\partial \mathbb{H}^{2}$ is exactly $\Lambda$. 
Let $|\Lambda|$ be the geometric realization of $\Lambda$. 
For a leaf $\lambda$ of $\Lambda$,
its realization, denoted as $|\lambda|$, 
is the leaf of $|\Lambda|$ whose endpoints are exactly $\lambda$. 
Now assume that $\Lambda$ is bifoliar.

Recall that $\Lambda$ is orientable if 
the leaves of $|\Lambda|$ have continuously varying orientations,
or equivalently, if
all ideal polygons of $|\Lambda|$ have even number of sides.
An orientation on $\Lambda$ induces continuously varying orientations on 
all leaves of $|\Lambda|$.

Let $\La, \Lb$ be a bifoliated pair of almost laminations on $S^{1}$.
Then under both of the approaches of \cite{BBM} and ours,
there is a canonical quotient map $\pi: \DD \to \PP$,
where $\PP$ is the bifoliated plane completed from $(S^{1},\La,\Lb)$.
We fix an orientation on $\La$.
For each $x \in |\lambda|$,
we define the \emph{positive half-thread} of $|\lambda|$ with respect to $x$
(denoted as $|\lambda_+(x)|$) as follows.
Let $|\lambda^{'}_{+}(x)|$ be the component of $|\lambda| - \{x\}$ in the positive side of $x$,
and let $|\lambda_+(x)| = 
|\lambda^{'}_{+}(x)| - \pi^{-1}(\pi(x))$.

Let $\rho: G \to \text{Homeo}(S^{1})$ be a faithful action of 
a torsion-free group $G$ on $S^{1}$ that 
preserves $\La, \Lb$ and their orientations,
and we further assume that $\rho$ is flowable.
Let $\mathcal{G}(|\La|)$ be the set of ideal polygons of $|\La|$.
Let $P \in \mathcal{G}(|\La|)$ consists of leaves $l_1,\ldots,l_{2n}$.
We note that the flowable condition on $\rho$ implies that
the stabilizer of $P$ is an infinite cyclic group which preserves the orientations of leaves of $P$.
We can assign to the boundary leaves of $P$ a monotone map $h_P: \bigcup_{i=1}^{2n} l_i \to \R$,
equivariant under the stabilizer of $P$ \cite{TZZ},
such that for all $i \in \{1,\ldots,2n\}$ and $x,y \in l_i$,
$h_P(x) \ne h_P(y)$ if and only if $\pi(x) \ne \pi(y)$.
In addition,
we can assign to each $P \in \mathcal{G}(|\La|)$ such a monotone map $h_P$ so that
$h_{f(P)}(f(x)) = h_{f(P)}(f(y))$ if and only if 
$h_P(x) = h_P(y)$ for all $P \in \mathcal{G}(|\La|)$, $x, y \in P$ and $f \in G$.

Let 
\[N_1(\rho) = \{(x,t) \mid \exists \lambda \in |\La| \text{ s.t. }
x \in \lambda, t \in |\lambda_+(x)|\}\]
endowed with the subspace topology of $\DD \times \DD$,
let
$N_2(\rho)$ be the image of $N_1(\rho)$ under the quotient map
$\pi \cdot \pi: \DD \times \DD \to \PP \times \PP$.
Let $\sim$ be the equivalence relation on $N_2(\rho)$ such that,
for any $(x_1,t_1), (x_2,t_2) \in N_2(\rho)$,
$(x_1,t_1) \sim (x_2,t_2)$ if and only if:
\begin{enumerate}[(1)]
\item $x_1 = x_2$,
\item if $x_1$ is not contained in a singular leaf of $\PP$,
then $t_1 = t_2$,
\item if $x_1$ is contained in a singular leaf $\lambda$ of $\PP$,
let $P \in \mathcal{G}(|\La|)$ so that $\pi(P) = \lambda$,
then there are $q_1, q_2 \in P$ with
$t_1 = \pi(q_1)$, $t_2 = \pi(q_2)$ and $h_P(q_1) = h_P(q_2)$.
\end{enumerate}

Let $N(\rho)$ be the quotient space $N_2(\rho) / \sim$,
endowed with the quotient topology.
We note that $N(\rho)$ is still an $\R$-bundle over $\PP$,
and we denote by $\w{\phi}(\rho)$ the space of these $\R$-fibers.
We note that $N_1(\rho), N_2(\rho), N(\rho),\w{\phi}(\rho)$ depend on
the choice of maps $h_P: P \to \R$ as above.
$M(\rho)$ and $\phi(\rho)$ are constructed via the following proposition:

\begin{prop}\label{pseudo-Anosov construction}
$\rho$ induces a free and discrete action of $G$ on $N(\rho)$,
and $\w{\phi}(\rho)$ is equivariant under this action.
\end{prop}

\begin{cons}
    $M(\rho) = N(\rho)/ G$,
$\phi(\rho) = \w{\phi}(\rho)/ G$.
\end{cons}

\begin{rmk}
\begin{enumerate}[(a)]
\item Suppose $\rho$ is conjugate to the $\pi_1$-action on $S^{1}_{\infty}(\phi)$
for some transitive pseudo-Anosov flow $\phi$.
Combining Theorem \ref{main theorem} (b) with \cite[Theorem 3.4]{Bar95},
$M(\rho)$ is independent of the choice of $\{h_P: P \to \R \mid P \in \mathcal{G}(|\La|)\}$ 
up to homeomorphisms,
and $\phi(\rho)$ is independent of the choice of $\{h_P: P \to \R \mid P \in \mathcal{G}(|\La|)\}$
up to orbit equivalence.
\item If $\rho$ is conjugate to the $\pi_1$-action on $S^{1}_{\infty}(\phi)$
for some Anosov flow $\phi$,
we have $\mathcal{G}(\La) = \emptyset$,
and therefore
$N(\rho) = N_2(\rho)$ and $M(\rho) = N_2(\rho) / G$.
\end{enumerate}
\end{rmk}



\subsection{Taut foliations transverse to pseudo-Anosov flows}\label{subsec 1.2}
The construction of $(M(\rho),\phi(\rho))$ (from $\DD \times \DD$)
reveals some new geometric and topological structures.
Here we provide such an example:
there is a canonical foliation of $M(\rho)$ transverse to $\phi(\rho)$ that
is taut when $M(\rho)$ is a closed $3$-manifold.
We summarize the properties of this foliation in the case where $\phi(\rho)$ is
a transitive pseudo-Anosov flow in the following theorem.

\begin{THM}\label{transverse foliation}
Let $\phi$ be an orientable transitive pseudo-Anosov flow in
a closed orientable $3$-manifold $M$.
Then there exists a co-orientable taut foliation $\F$ in $M$ transverse to $\phi$ such that
all orbits of $\w{\phi}$ are asymptotic to some leaf of $\w{\F}$,
where $\w{\phi}, \w{\F}$ are the pull-backs of $\phi, \F$ in the universal cover of $\w{M}$.
\end{THM}

See Remark \ref{comparison} for a comparison between $\F$ and
the foliation constructed in \cite{Gab92b}.
When $\phi$ is non-orientable,
it's possible that $M$ admits no co-orientable taut foliation 
(see Example \ref{pretzel example}).
Hence Theorem \ref{transverse foliation} may not hold when $\phi$ is non-orientable.

\subsection{Reconstruction from convergence group actions on $S^{2}$}

For a quasigeodesic topological pseudo-Anosov flow $\phi$ in a hyperbolic $3$-manifold $M$,
$S^{1}_{\infty}(\phi)$ can be canonically identified with a sphere-filling Peano curve in the ideal boundary of $\mathbb{H}^{3}$.
This sphere-filling curve is studied by Cannon-Thurston \cite{CT07},
Fenley \cite{Fen12, Fen16},
and is further generalized to general quasigeodesic flows by Frankel \cite{Fra15}.
In this case,
our reconstruction of $(M,\phi)$ from $S^{1}_{\infty}(\phi)$ is directly achieved from $\partial \mathbb{H}^{3}$.
A conjecture of Cannon \cite{CS98} says that,
all uniform convergence group actions on $S^{2}$ are 
conjugate to cocompact Klein groups.
We focus on the case where this $S^{2}$ is an ideal boundary coming from
certain almost lamination or bifoliated plane in the sense of Fenley \cite{Fen12, Fen16},
and from this perspective we will provide some information on
``how a hyperbolic $3$-manifold comes from a convergence group action''.

Fenley \cite{Fen12, Fen16} defined a flow ideal boundary for 
any quasigeodesic topological pseudo-Anosov flow in a closed $3$-manifold with Gromov hyperbolic group,
which is homeomorphic to a $2$-sphere, and the $\pi_1$-action on its orbit space induces
a uniform convergence group action.
The flow ideal boundary is completely determined by the orbit space and thus
can be generalized to the \emph{abstract flow ideal boundary} for
an almost lamination or bifoliated plane with certain properties.

Let $\La, \Lb$ be a bifoliar pair of almost laminations on $S^{1}$.
Under certain \emph{basic properties} satisfied by
the orbit space of all quasigeodesic topological pseudo-Anosov flows (see Assumption \ref{assume: no infinite chain}),
there always exists a well-defined $2$-sphere abstract flow ideal boundary, 
denoted $\RR(\La, \Lb)$.

Let $\rho: G \to \text{Homeo}(S^{1})$ be a faithful action of 
a torsion-free group $G$ on $S^{1}$ that preserves $\La, \Lb$.
We prove that

\begin{THM}\label{thm: convergence}
    Suppose that $\La, \Lb$ are orientable and $G$ preserves the orientations on $\La, \Lb$,
    and that the basic properties for $\RR(\La, \Lb)$ hold.
\begin{enumerate}[(a)]
 \item If the induced action of $G$ on $\RR(\La, \Lb)$ is a convergence group action,
then $\rho$ is realized by a $3$-manifold $M$ with a reduced pseudo-Anosov flow $\phi$.
\item Suppose, furthermore, that $G$ acts on $\RR(\La, \Lb)$ as a uniform convergence group action. Then one of the following two possibilities hold: 
\begin{enumerate}[(1)]
          \item
$M$ is a closed hyperbolic $3$-manifold and $\phi$ is a quasigeodesic topological pseudo-Anosov flow.
\item $M$ is a non-compact $3$-manifold with Gromov hyperbolic group and has no $T^{2}$ end.
\end{enumerate}
\end{enumerate}
\end{THM}

The scenario (2) can be ruled out under the additional assumption that the $\rho$ satisfies the \emph{compactness condition}, a condition met by $\pi_1$-actions of all smooth pseudo-Anosov flows in closed $3$-manifolds. 
See Section \ref{sec: convergence group action} for definitions and more detailed discussions. 

Note that all non-$\R$-covered Anosov flows in 
closed hyperbolic $3$-manifolds are quasigeodesic \cite{Fen22} and thus 
can be reconstructed from the above theorem up to double covers.

Our construction provides a description for the resulting hyperbolic $3$-manifolds
arising from uniform convergence group actions.
Theorem \ref{thm: convergence} implicitly offers a description of
$(M, \phi)$ in terms of $(\RR(\La, \Lb))^{3}$.
If we further assume that $\La, \Lb$ make no perfect fit,
then $M$ is a quotient space of a subspace of $\Theta(\RR(\La, \Lb)) / G$,
where $\Theta(\RR(\La, \Lb))$ is the distinct (ordered) triple of $\RR(\La, \Lb)$.
See Remark \ref{rem: unordered triple} for an explanation.


\subsection{The reconstruction of non-orientable pseudo-Anosov flows}

Here we briefly explain the reconstruction in the non-orientable case: it is reduced to the problem of reconstructing a $3$-manifold (possibly with cusps) with a flow from a regular double cover.

Let $\rho: G \to \text{Homeo}(S^{1})$ be a faithful action of a torsion-free group $G$ on
$S^{1}$ that preserves a bifoliar pair of almost laminations.
We aim to construct a $3$-manifold $M$ with reduced pseudo-Anosov flow $\phi$ such that
$(M,\phi)$ realizes $\rho$.
$\rho$ has an \emph{orientable doubling} $\rho_*: G_* \to \text{Homeo}(S^{1})$ satisfying the assumptions in Theorem \ref{main theorem},
such that the orientable double cover of $(M,\phi)$ realizes $\rho_*$.
See the end of Subsection \ref{sec:orbitspace} and Remark \ref{rmk: orientable cover on bifoliated plane} for explanations.

By theorem \ref{main theorem}, if $\rho_*$ is flowable, then $\rho_*$ is realized by $(M_*, \phi_*)$, as constructed in Subsection \ref{subsec 1.2}.
In addition, by theorem \ref{main theorem} (b), $(M_*, \phi_*)$ is exactly the orientable double cover of $(M,\phi)$ when $\rho$ is conjugate to the $\pi_1$-action of certain transitive pseudo-Anosov flow in closed $3$-manifold.
In this case, the reconstruction of $(M,\phi)$ is reduced to reconstructing $(M,\phi)$ from
$(M_*,\phi_*)$.
Let $M^{\text{c}}$ (resp. $M^{\text{c}}_{*}$) denote the complements of the odd-indexed singular orbits of $\phi$ (resp. $\phi_*$) in $M$ (resp. $M_*$). Then $(M^{\text{c}}_{*}, \phi_* \mid_{M^{\text{c}}_{*}})$ is exactly a regular double cover of $(M^{\text{c}}, \phi \mid_{M^{\text{c}}})$. Hence, the problem of reconstructing $(M,\phi)$ is reduced to the reconstruction of  $(M^{\text{c}}, \phi \mid_{M^{\text{c}}})$ from the double cover $(M^{\text{c}}_{*}, \phi_* \mid_{M^{\text{c}}_{*}})$. 

It is conceivable that $(M^{\text{c}}, \phi \mid_{M^{\text{c}}})$ can be further reconstructed from its double cover, under additional processes. In the case that the invariant almost laminations of $\rho$ have no ideal polygon with at least three sides, it is confirmed that $(M,\phi)$ can be reconstructed based on our work, in a personal communication between the last named author and Jonathan Zung \cite{ZZ}. See Remark \ref{personal communication}.

\subsection{Organization of the paper}
The paper is organized as follow. In Section \ref{sec:prelim}, we collect all the preliminaries for the paper. In particular, we provide the definition of topological/smooth Anosov and pseudo-Anosov flows, explain the orbit space of pseudo-Anosov flow and its compactification, introduce the notion of bifoliated plane and Anosov-like action, and finally almost laminations on the circle. 
In Section \ref{sec:realizingtopflow}, we first introduce a new property which is not listed in the notion of an Anosov-like action but is always satisfied by the $\pi_1$-action on the orbit space of a transitive pseudo-Anosov flow. Then, when we have a flowable group action (weaker assumption than the Anosov-like action but satisfying the new property we introduce) on a bifoliated plane without singularities, we construct a 3-manifold which has the given group as the fundamental group and also construct a flow in the manifold.  
In Section \ref{sec:char_anosov}, the Compactness, Convergence, and Divergence properties of the flows are introduced, and using these properties, we give a necessary and sufficient condition for the flow constructed in Section \ref{sec:realizingtopflow} is a topological Anosov flow. 
In Section \ref{sec:frombifoliatedplane}, we strongly generalize the construction in Section \ref{sec:realizingtopflow} and construct a 3-manifold with a reduced pseudo-Anosov flow from a flowable action on the bifoliated plane with even-prong singularities. 
In Section \ref{sec:circle-lamination}, we show how an Anosov-like/flowable action on bifoliated planes can be interpreted as an action on the circle with a pair of invariant almost laminations. Then we explain how this perspective can be combined with the results in the previous sections to give a proof of our main theorem. 
In Section \ref{sec: convergence group action}, we investigate the situation when we get an induced action on the 2-sphere from an action on the bifoliated plane. This case can be considered as a special case of the Cannon's conjecture. 
Finally, in Section \ref{sec:questions}, we collect some questions that arise naturally from our work. 

\subsection{Settings of the paper}\label{subsec: settings}
For Anosov and pseudo-Anosov flows, here are two versions commonly studied and used in various works, \emph{smooth} (Definitions \ref{Anosov flow}, \ref{smooth pseudo-Anosov flow}) and \emph{topological} (Definitions \ref{topological Anosov flow}, \ref{topological pseudo-Anosov flow}), with the smooth version always implying the topological version. In this paper, unless explicitly stated otherwise, all pseudo-Anosov flows are assumed to be smooth pseudo-Anosov flows. However, in several instances, we will discuss the setting of topological pseudo-Anosov flows, with special emphasis on the relevant contexts. For the reader's convenience, we summarize this distinction as follows.

In Subsection \ref{subsec: ideal boundary}, we review some basic concepts in the context of topological pseudo-Anosov flows as introduced in \cite{Fen12, Fen16},
and we prove Theorem \ref{thm: convergence} in Section \ref{sec: convergence group action}.
Therefore, in these sections, we focus on topological pseudo-Anosov flows.

\subsection*{Acknowledgement} 
Zhao specially thanks his Ph.D. thesis advisor, Xingru Zhang, for his guidance and helpful conversations.

Lemma \ref{equivalence class} is known from the personal communication between
Chi Cheuk Tsang, Jonathan Zung and Zhao.
Zhao thanks them for their helpful communications prior to this work.

Baik thanks Shanghai Center for Mathematical Sciences for hospitality where the last part of the project was done and he had many fruitful conversations with participants. Baik especially thanks Chris Connell, Koji Fujiwara, Ying Hu, Junzhi Huang, Robert Tang, and Samuel Taylor for inspiring comments.

Baik was supported by the National Research Foundation of Korea (NRF) grant funded by the Korea government (MSIT) (No. 2020R1C1C1A01006912).

\section{Preliminary} \label{sec:prelim}
\subsection{Pseudo-Anosov flows in $3$-manifolds} \label{subsec:pAflows}

For Anosov and pseudo-Anosov flows, at least two different versions exist and have been studied in the literature - smooth and topological. We recall these definitions and collect some relevant facts. The following definition makes sense when $M$ has an auxiliary (Riemannian) metric. 

\begin{defn}[Smooth Anosov flow]\rm\label{Anosov flow}\cite[Definition 1.1]{B17} A $C^1$ flow $\phi^t$ on a closed manifold $M$ is called {\em Anosov}, if there is a splitting $TM=\mathbb{R}X\oplus E^{ss}\oplus E^{uu}$, which is preserved by $\phi^t$, $a, b>0$, and a Riemannian (or Finsler) metric $\|\cdot\|$ on $TM$, such that
\begin{enumerate}
    \item $X$ is the vector field that generates the flow $\phi^t$.
    \item For any $v\in E^{ss}$, $t>0$, 
    \[\|(D\phi^t)(v)\|\leq be^{-at}\|v\|\]
    \item For any $v\in E^{uu}$, $t>0$, 
    \[\|(D\phi^{-t})(v)\|\leq be^{-at}\|v\|\]
\end{enumerate}
\end{defn}

Mosher introduces the concept of topological Anosov flows in
\cite{Mos92a}, \cite{Mos92b}.
The following definition basically follows from \cite[Definition 1.5]{B17}.

\begin{defn}[Topological Anosov flow]\rm\label{topological Anosov flow}
Let $M$ be a closed $3$-manifold with 
a flow $\phi^{t}: M \to M$ ($t \in \R$),
and we choose a metric $d$ on $M$.
$\phi^{t}$ is called 
a \emph{topological Anosov flow} if the following conditions hold:
\begin{enumerate}[(a)]
    \item $\phi$ has a $C^{0}$ tangent (non-singular) vector field,
and each orbit $t \to \phi^{t}(x)$ ($x \in M$) is $C^{1}$.
\item There are a pair of transverse foliations $\Fs$
(called the \emph{stable} foliation of $\phi^{t}$) and 
$\Fu$ (called the \emph{unstable} foliations) of $M$,
such that the orbits of $\phi^{t}$ are exactly the intersections of $\Fs, \Fu$.
\item If $x,y$ are two points in the same leaf of $\Fs$,
then there exists an orientation-preserving homeomorphism 
$h: \R \to \R$ such that
$$\lim_{t \to +\infty}d(\phi^{t}(x),\phi^{h(t)}(y)) = 0.$$
If $x,y$ are two points in the same leaf of $\Fu$,
then there exists an orientation-preserving homeomorphism 
$h: \R \to \R$ such that
$$\lim_{t \to -\infty}d(\phi^{t}(x),\phi^{h(t)}(y)) = 0.$$
\item Let $\lambda$ be a leaf of $\Fs$ and let $\mu$ be a leaf of $\Fu$.
We denote by $d_{\lambda}$ (resp. $d_{\mu}$) 
the induced path metric on $\lambda$ (resp. $\mu$).
For any $x,y \in \lambda$ contained in distinct orbits,
$$\lim_{t \to -\infty}d_\lambda(\phi^{t}(x),\phi^{t}(y)) = +\infty.$$
For any $x,y \in \mu$ contained in distinct orbits,
$$\lim_{t \to +\infty}d_\mu(\phi^{t}(x),\phi^{t}(y)) = +\infty.$$
\end{enumerate}
\end{defn}

\begin{defn}[Orbit equivalence]\rm
For $i = 1,2$,
let $M_i$ be a closed orientable $3$-manifold that admits a flow $\phi_i$.
$(M_1, \phi_1)$ is said to be \emph{orbitally equivalent} to $(M_2, \phi_2)$ if
there exists a homeomorphism $$h: M_1 \to M_2$$ 
such that $h$ sends $\{\text{orbits of } \phi_1\}$ to
$\{\text{orbits of } \phi_2\}$, and preserves the orientation.
\end{defn}

\begin{THM}[Shannon; {\cite{Sha21}}]\label{Shannon}
	Any transitive topological Anosov flow is orbit equivalent to a (transitive) smooth Anosov flow.
\end{THM}

Now we introduce pseudo-Anosov flows. The following two definitions are from \cite[Definition 5.9]{AT22}:

\begin{defn}[Smooth pseudo-Anosov flow]\rm\label{smooth pseudo-Anosov flow}
Let $M$ be a closed smooth $3$-manifold. A continuous flow $\phi^t$ on $M$ is called a {\em smooth pseudo-Anosov flow}, if:
\begin{enumerate}
    \item $M$ has a finite collection of closed orbits $\{\gamma_1, \dots, \gamma_s\}$, called the {\em singular orbits}, and $\phi^t$ is smooth outside the singular orbits.
    \item There is a Riemannian metric on $M\backslash\bigcup_i\gamma_i$, which gives a path metric on $M$, $a, b>0$, and a $\phi^t$-invariant splitting of $T(M\backslash\bigcup_i\gamma_i)$ into $\mathbb{R}X\oplus E^{ss}\oplus E^{uu}$, such that:
    
\begin{enumerate}
    \item $X$ is the vector field that generates the flow $\phi^t$.
    \item For any $v\in E^{ss}$, $t>0$, 
    \[\|(D\phi^t)(v)\|\leq be^{-at}\|v\|\]
    \item For any $v\in E^{uu}$, $t>0$, 
    \[\|(D\phi^{-t})(v)\|\leq be^{-at}\|v\|\]
\end{enumerate}
\item Consider a map $\mathbb{R}^2\rightarrow\mathbb{R}^2$ defined by $(x, y)\mapsto (\lambda x, \lambda^{-1}y)$ for some $\lambda>0$. Lifting it to a branched cover branching at $(0, 0)$, the monodromy flow in the resulting mapping torus, restricted to a neighborhood of the orbit of $(0, 0)$, is called a {\em neighborhood of a pseudo hyperbolic orbit}.  For each singular orbit $\gamma_i$, there is a neighborhood $N_i$ of $\gamma_i$, a bijection $f_i$ from this neighborhood to a neighborhood of a pseudo hyperbolic orbit, such that:
\begin{enumerate}
    \item $f_i$ is bi-Lipschitz on $N_i$, smooth outside $\gamma_i$.
    \item $f_i$ sends orbits to orbits.
    \item $f_i$ sends $E^{uu}$ and $E^{ss}$ to the horizontal and vertical subbundle after composing with the branched cover.
\end{enumerate}
\end{enumerate}

\end{defn}

The following definition generalizes topological Anosov flows to
the pseudo-Anosov setting.

\begin{defn}[Topological pseudo-Anosov flow]\rm\label{topological pseudo-Anosov flow}
Let $M$ be a closed $3$-manifold with a flow $\phi^{t}: M \to M$ ($t \in \R$).
We choose a metric $d$ on $M$.
$\phi^{t}$ is called 
a \emph{topological pseudo-Anosov flow} if:
\begin{enumerate}[(a)]

\item There is a pair of singular foliations $\Fs, \Fu$ of $M$ such that
      \begin{enumerate}
          \item[(a1)] 
each orbit of $\phi^{t}$ is contained in a leaf of $\Fs$ and a leaf of $\Fu$,
\item[(a2)]
there is a finite collection of orbits $\gamma_1,\ldots,\gamma_n$ of $M$ such that
$\Fs, \Fu$ are regular foliations when restricted to $M - \bigcup_{i=1}^{n} \gamma_n$.
      \end{enumerate}
\item 
For any point $x \in M - \bigcup_{i=1}^{n} \gamma_n$,
$x$ has a closed neighborhood homeomorphic to $I_s \times I_u \times I_\phi$
(where $I_s, I_u, I_\phi$ are homeomorphic to a closed interval),
where each $(t_1,t_2,I_\phi)$ ($t_1 \in I_s, t_2 \in I_u$) 
is contained in an orbit of $\phi^{t}$,
each $(t, I_u, I_\phi)$ is contained in a leaf of $\Fu$,
and each $(I_s, t, I_\phi)$ is contained in a leaf of $\Fs$.
We call such a closed neighborhood a \emph{flow box}.
\item If $x,y$ are two points in the same leaf of $\Fs$,
then there exists an orientation-preserving homeomorphism 
$h: \R \to \R$ such that
$$\lim_{t \to +\infty}d(\phi^{t}(x),\phi^{h(t)}(y)) = 0.$$
If $x,y$ are two points in the same leaf of $\Fu$,
then there exists an orientation-preserving homeomorphism 
$h: \R \to \R$ such that
$$\lim_{t \to -\infty}d(\phi^{t}(x),\phi^{h(t)}(y)) = 0.$$
\item 
Let $\lambda$ be a leaf of $\Fs$, let $\mu$ be a leaf of $\Fu$,
and let $d_{\lambda}$ (resp. $d_{\mu}$) refer to
the induced path metric on $\lambda$ (resp. $\mu$).
For any $x,y \in \lambda$ contained in distinct orbits,
$$\lim_{t \to -\infty}d_\lambda(\phi^{t}(x),\phi^{t}(y)) = +\infty.$$
For any $x,y \in \mu$ contained in distinct orbits,
$$\lim_{t \to +\infty}d_\mu(\phi^{t}(x),\phi^{t}(y)) = +\infty.$$
\item Let $i \in \{\gamma_1,\ldots,\gamma_n\}$,
and let $\lambda_i$ denote the leaf of $\Fs$ containing $\gamma_i$.
Then $\lambda_i \cong P \times I / \stackrel{f}{\sim}$,
where $P$ is a $k_i$-prong on the $2$-plane for some $k_i \geqslant 3$,
and $f: P \times \{0\} \to P \times \{1\}$ is a homeomorphism such that,
\begin{enumerate}[(1)]
    \item $f$ preserves the circular order of the prongs of $P$,
\item $f$ either topologically contracts each prong of $P$ or
topologically expands each prong of $P$.
\end{enumerate}
\end{enumerate}
\end{defn}

Each $\gamma_i$ is called a \emph{singular orbit} of $\phi^{t}$.

\begin{rmk}\label{AT}
    Definition \ref{topological pseudo-Anosov flow} is
    an adaptation of Definition \ref{topological Anosov flow} in
    the pseudo-Anosov setting,
    which is also compactified with the setting of Fenley's works in
    \cite{Fen12, Fen16}.
    
    In \cite[Definition 5.9]{AT22},
    Agol and Tsang consider a different version of topological pseudo-Anosov flow and
    prove that any such flow is orbit equivalent to 
    a smooth Anosov flows when it is transitive \cite[Theorem 5.11]{AT22}.
    The only difference between Definition \ref{topological pseudo-Anosov flow} and 
    their definition is that,
    Definition \ref{topological pseudo-Anosov flow} (d) is
    replaced by the existence of a Markov partition.
    As the topological Anosov flows in our setting might not 
    satisfy \cite[Definition 5.9]{AT22} when they are non-transitive, we do not use this definition for the pseudo-Anosov setting.
\end{rmk}

\subsection{Surgeries on pseudo-Anosov flows}\label{subsec: surgeries on pA flows}

In this subsection, we review some background on Dehn surgeries along closed orbits of pseudo-Anosov flows.

Each closed orbit of a pseudo-Anosov flow is equipped with a local system,
called its degeneracy locus:

\begin{conv}[Degeneracy locus]\rm
Let $M$ be a closed orientable $3$-manifold that admits a pseudo-Anosov flow $\phi$.
Let $\Fs$ denote the stable foliation of $\phi$.
Let $\gamma$ be a closed orbit of $\phi$,
and let $N(\gamma)$ be a (closed) regular neighborhood of $\gamma$ in $M$.
Finally, let $\lambda$ be the leaf of $\Fs$ containing $\gamma$. 
Then $\lambda \cap \partial N(\gamma)$ is
a finite union of (parallel) essential simple closed curves
$\{s_1, \ldots, s_k\}$ on $\partial N(\gamma)$.
The slopes of $s_i$ on $\partial N(\gamma)$ depend only on $\gamma$,
which is called the \emph{degeneracy slope} of $\gamma$ and
denoted as $\delta(\gamma)$.
The \emph{degeneracy locus} of $\gamma$ is defined as
a multiple of the degeneracy slope with the coefficient $k$,
denoted $d(\gamma) = k\delta(\gamma)$.
And we call $k$ the \emph{multiplicity of} $d(\gamma)$.
\end{conv} 

\begin{nota}\rm
Let $\alpha, \beta$ be two loops in the complement of a knot $K$ in a closed $3$-manifold.
The \emph{distance} between $\alpha, \beta$,
denoted as $\Delta(\alpha,\beta)$, is defined as the minimal geometric intersection number of $\alpha, \beta$.
To make this notation compatible with the concept of degeneracy locus,
we also define
$\Delta(\alpha,k\beta) = \Delta(k\beta,\alpha) = k \Delta(\alpha,\beta)$ 
for any $k \in \N$.
\end{nota}

For a closed orbit $\gamma$ of a pseudo-Anosov flow $\phi$,
$\gamma$ has a well-defined meridian in the ambient $3$-manifold, denoted $m(\gamma)$.
We always have $\Delta(d(\gamma), m(\gamma)) \geqslant 2$.
In addition,
$\Delta(d(\gamma), m(\gamma)) = 2$ when $\gamma$ is a non-singular orbit and
$\Delta(d(\gamma), m(\gamma)) > 2$ when $\gamma$ is a singular orbit.

Fried's surgery \cite{Fri83} produces new pseudo-Anosov flows from
a suspension pseudo-Anosov flow,
by performing Dehn surgeries along a collection of closed orbits:

\begin{defn}[Fried's surgery, {\cite{Fri83}}]\rm\label{Fried's surgery}
Let $\Sigma$ be a closed orientable surface, and let $h: \Sigma \to \Sigma$
be an orientation-preserving pseudo-Anosov homeomorphism.
Let $M = \Sigma \times I / \stackrel{h}{\sim}$ be
the mapping torus of $h$ and let $\phi$ denote the suspension pseudo-Anosov flow.
Let $\gamma_1, \ldots, \gamma_n$ be a finite collection of closed orbits of $\phi$,
and let $s_i$ be a slope on each $\gamma_i$ such that
$\Delta(s_i, d(\gamma_i)) \geqslant 2$.
We perform the Dehn surgery of $M$ along $\bigcup_{i=1}^{n} \gamma_i$, with the slope $s_i$ on each $\gamma_i$.
Let $N$ denote the resulting $3$-manifold.
Note that $M - \bigcup_{i=1}^{n} \gamma_i$ is canonically identified with a link complement of $N$, and thus $\phi \mid_{M - \bigcup_{i=1}^{n} \gamma_i}$ is canonically identified with
a flow of this link complement of $N$,
which extends to a flow $\varphi$ of $N$.
We call this operation \emph{Fried's surgery} on $(M,\phi)$.
\end{defn}

Let $\Fs, \Fu$ denote the stable and unstable foliations of $\phi$.
Clearly,
Fried's surgery produces a pair of singular foliations
$\mathcal{E}^{+}, \mathcal{E}^{-}$ from $\Fs, \Fu$ respectively.
Note that Fried's surgery takes each $\gamma_i$ to a closed orbit that is contained in a leaf of the stable foliation with
$\Delta(s_i,d(\gamma_i))$ prongs.

\begin{THM}[{\cite{Fri83, brunella1995surfaces}}]\label{Fried's surgery theorem}
All transitive pseudo-Anosov flows in closed $3$-manifolds can be
produced by Fried's surgery from a suspension pseudo-Anosov flow on a hyperbolic fibered $3$-manifold.
\end{THM}

\subsection{Ideal boundary of the orbit space}\label{subsec: ideal boundary}

Let $M$ be a closed orientable $3$-manifold that admits 
a topological pseudo-Anosov flow $\phi$.
We denote by $\Fs, \Fu$ the stable and unstable foliations of $\phi$.

Let $\w{M}$ denote the universal cover of $M$ and
let $\w{\phi}$ denote the pull-back flow of $\phi$ in $\w{M}$.
Let $\w{\Fs}, \w{\Fu}$ denote 
the pull-backs of $\Fs, \Fu$ in $\w{M}$.
The orbit space of $\w{\phi}$,
denoted $\Or(\phi)$,
is the quotient space of $\w{M}$ by collapsing each orbit of $\w{\phi}$ to a single point,
endowed with the quotient topology.
Note that $\Or(\phi)$ is always homeomorphic to $\R^{2}$ 
\cite[Proposition 4.2]{FM01}.

Let $\pi_\phi: \w{\phi} \to \Or(\phi)$ denote the canonical projection that
sends each orbit of $\w{\phi}$ to the point in $\Or(\phi)$ canonically identified with it.
Then $\pi_{\phi}(\w{\Fs}),\pi_{\phi}(\w{\Fu})$ are 
a pair of singular $1$-dimensional foliations in $\Or(\phi)$.
We denote them by $\Ors(\phi), \Oru(\phi)$.
$\pi_\phi$ descends the deck transformations in $\w{M}$ to 
an action on $\Or(\phi)$ that preserves $\Ors(\phi), \Oru(\phi)$,
which is referred to as the \emph{$\pi_1$-action} on $\Or(\phi)$.
Because $M$ is orientable and the orbits of $\w{\phi}$ have continuously varying orientations,
the $\pi_1$-action on $\Or(\phi)$ is an orientation-preserving action on $\Or(\phi)$.

In \cite{Fen12},
Fenley constructs an ideal boundary for $\Or(\phi)$,
denoted $S^{1}_\infty(\phi)$, 
so that the $\pi_1$-action on $\Or(\phi)$ extends to $S^{1}_\infty(\phi)$.

\begin{THM}[Fenley, 2012]
$\Or(\phi)$ has a canonical compactification $\Or(\phi) \cup S^{1}_\infty(\phi)$,
homeomorphic to a $2$-disk,
such that the $\pi_1$-action on $\Or(\phi)$ extends to $\Or(\phi) \cup S^{1}_\infty(\phi)$.
\end{THM}

$\Ors(\phi), \Oru(\phi)$ can be canonically identified with
a pair of laminations on $S^{1}_{\infty}(\phi)$ with certain leaves removed.
We call such pair
\emph{almost} laminations (see the following Subsection \ref{sec:BBM}).
We denote by $\Lambda^{s}_{\infty}(\phi), \Lambda^{u}_{\infty}(\phi)$ this pair of
almost laminations on $S^{1}_{\infty}(\phi)$.
The action of $\pi_1(M)$ on $S^{1}_{\infty}(\phi)$ preserves
$\Lambda^{s}_{\infty}(\phi), \Lambda^{u}_{\infty}(\phi)$.
We also refer to this action as the $\pi_1$-action on
$S^{1}_{\infty}(\phi)$.

For a leaf $\lambda$ of $\Ors(\phi)$ and a leaf $\mu$ of $\Oru(\phi)$,
$\lambda, \mu$ share an ideal endpoint in $S^{1}_{\infty}(\phi)$ exactly
when they make a perfect fit:

\begin{defn}[Perfect fit, {\cite{Fen12}}]\rm\label{perfect fit}
   Two leaves $\lambda$, $\mu$ that share an ideal endpoint are said to make a {\em perfect fit}, if one can pick half leaves from them, which together with some intervals on other leaves bound a closed product rectangle with one corner removed. 
   Here, by a half leaf on $\lambda$, we mean the closure of a component of $\lambda - \{x\}$ for certain $x \in \lambda$.
\end{defn}

Fenley \cite{Fen12, Fen16} defines
the \emph{flow ideal boundary} $\RR(\phi)$ for $\phi$ when
$\phi$ is \emph{bounded} \cite[Definition 1.1]{Fen16},
that is,
$\phi$ is not topologically conjugate to a suspension Anosov flow,
any closed orbit of $\phi$ is not nontrivially freely homotopic to itself,
and the cardinality of free homotopy classes has an upper bound.

\begin{defn}[Flow ideal boundary]\rm\label{flow ideal boundary}
    Suppose that $\phi$ is bounded.
    Let $\sim$ be the closed equivalence relation on $S^{1}_{\infty}(\phi)$ 
    generated by the following relation:
    $x \sim y$ when $(x,y)$ is a leaf of $\Lambda^{s}_{\infty}(\phi)$ or
    $\Lambda^{u}_{\infty}(\phi)$,
    i.e. $x, y$ are the endpoints of some leaf of $\Ors(\phi)$, $\Oru(\phi)$.
    Let $\RR(\phi)$ be the quotient of $S^{1}_{\infty}(\phi)$ by $\sim$.
\end{defn}

The main result of \cite{Fen16} is

\begin{THM}[Fenley, 2016]\label{bounded to convergence}
\hspace{0em}
\begin{enumerate}[(a)]
    \item Suppose that $\phi$ is bounded. 
    Then $\RR(\phi)$ is homeomorphic to a $2$-sphere,
    and the $\pi_1$-action on $\Or(\phi)$ induces a uniform convergence group action
    on $\RR(\phi)$.
    \item Suppose that $\pi_1(M)$ is Gromov hyperbolic.
    Then $\phi$ is quasigeodesic if and only if $\phi$ is bounded.
\end{enumerate}    
\end{THM}

We will give some remarks on the differences between our Definition \ref{flow ideal boundary} and 
Fenley's definition \cite{Fen12, Fen16}.

\begin{rmk}
\begin{enumerate}[(a)]
    \item Note that $\w{\phi}$ is homeomorphic to $\Or(\phi) \times (0,1)$ and 
has a natural compactification 
$X(\phi) = (\Or(\phi) \cup S^{1}_{\infty}(\phi)) \times [0,1]$.
In \cite[Theorem 4.3]{Fen12} and \cite[Definition 6.2]{Fen16},
$\RR(\phi)$ is defined as a quotient of 
\[\partial X(\phi) \cong (\Or(\phi) \cup S^{1}_{\infty}(\phi)) \times \{0,1\} \cup
S^{1}_{\infty}(\phi)) \times [0,1]\]
so that $\w{M} \cup \RR(\phi)$ is a compactification of $\w{M}$.
In particular,
$\pi_1(M)$ is Gromov hyperbolic and the action of $\pi_1(M)$ on $\R(\phi)$ is
topologically conjugate to the action on the Gromov ideal boundary of $\w{M}$.
\item For convenience, 
we define $\RR(\phi)$ as an abstract space homeomorphic to the flow ideal boundary
described above,
and we omit the connection between $\RR(\phi)$ and $\w{M}, \w{\phi}$.
We note from the proof of \cite[Theorem 6.5]{Fen16} that
$\RR(\phi)$ can be defined in terms of the almost laminations $\Lambda^{s}_{\infty}(\phi), \Lambda^{u}_{\infty}(\phi)$ 
or bifoliated planes $\Ors(\phi), \Oru(\phi)$,
independent of further information of the flow $\phi$.
Furthermore,
to make $\RR(\phi)$ homeomorphic to a $2$-sphere,
we only need that $\Or(\phi)$ satisfies certain properties.
See Section \ref{sec: convergence group action} for more details.
\end{enumerate}
\end{rmk}

\subsection{$\pi_1$-actions on orbit spaces} \label{sec:orbitspace}

We begin by introducing the notion of a reduced pseudo-Anosov flow which is a slightly weaker notion than a topological pseudo-Anosov flow. In general, our flowable actions on the bifoliated plane produces a reduced pseudo-Anosov flow in the 3-manifold. 

\begin{defn}\rm\label{reduced pseudo-Anosov flow}
    Let $M$ be a (possibly non-compact) orientable $3$-manifold.
    We call $\phi$ a \emph{reduced pseudo-Anosov flow} of $M$ if 
    the following conditions hold:
\begin{itemize}
    \item $\phi$ satisfies Definition \ref{topological pseudo-Anosov flow} (a), (b), (e).
We still denote by $\Fs, \Fu$ the pair of transverse foliations for $\phi$,
as in Definition \ref{topological pseudo-Anosov flow} (b).
\item Each non-singular leaf of $\Fs$ or $\Fu$ is either homeomorphic to $\R^{2}$ or
homeomorphic to $S^{1} \times \R$;
in the second case,
we still call this leaf a \emph{periodic leaf}.
\item For each periodic leaf $\lambda$ of $\Fs$ (resp. $\Fu$),
there is an orientation-preserving homeomorphism $h: \R \to \R$ such that
\begin{enumerate}[(1)]
\item $h(0) = 0$,
\item $h$ topologically contracts (resp. expands) 
both of the half-lines $[0,+\infty)$ and $(-\infty,0]$.
\item Let $\stackrel{h}{\sim}$ be the equivalence relation on $\R \times I$ such that
$(x,1) \sim (h(x),0)$ for each $x \in \R$.
Let $\mathcal{I}$ be the vertical foliation $\{\{t\} \times I \mid t \in \R\}$ of
$\R \times I$.
Then $(\lambda, \phi \mid_{\lambda})$ is homeomorphic to
$(\R \times I / \stackrel{h}{\sim}, \mathcal{I} / \stackrel{h}{\sim})$.
\end{enumerate}

    \item Let $\w{\phi}$ be the pull-back of $\phi$ in the universal cover of $M$.
Then the orbit space of $\w{\phi}$ is homeomorphic to $\R^{2}$.
\end{itemize}
\end{defn}

For a reduced pseudo-Anosov flow $\phi$,
we will adopt the same conventions as pseudo-Anosov flows:
we still denote by $\Or(\phi)$ the orbit space of $\w{\phi}$ and
denote by $\Ors(\phi), \Oru(\phi)$ the pair of singular $1$-dimensional foliations on
$\Or(\phi)$.

Barbot proves that two Anosov flows in closed $3$-manifolds are orbitally equivalent
if and only if 
the $\pi_1$-actions on their orbit spaces are conjugate \cite[Theorem 3.4]{Bar95}.
The purpose of this section is to prove the following. 

\begin{prop}\label{equivalent}
Let $M_1$ be a closed orientable $3$-manifold that admits
a transitive pseudo-Anosov flow $\phi_1$.
Let $G = \pi_1(M_1)$.
Suppose that $M_2$ is a closed orientable $3$-manifold with
$\pi_1(M_2) = G$, 
and $M_2$ admits a reduced pseudo-Anosov flow $\phi_2$.
Suppose further that the $\pi_1$-action on $\Or(\phi_2)$ is conjugate to
the $\pi_1$-action on $\Or(\phi_1)$ by 
an orientation-preserving homeomorphism between $\Or(\phi_1), \Or(\phi_2)$,
then $M_1 \cong M_2$, and $\phi_1$ is orbitally equivalent to $\phi_2$.
\end{prop}

We denote by $p_i: \w{M_i} \to M_i$ the universal cover of $M_i$ and
denote by $\w{\phi_i}$ the pull-back of $\phi_i$ in $\w{M_i}$.
For convenience,
we denote $\Or(\phi_i), \Ors(\phi), \Oru(\phi)$ by $\Or_i, \Ors_i, \Oru_i$.
Let $\omega_i: \w{M_i} \to \Or_i$ denote the projection maps.

Let $\tau_g: \Or_1 \to \Or_1$ for $g\in G$ be the $\pi_1$-action on $\Or_1$,
and let
$\sigma_g: \Or_2 \to \Or_2$ for $g \in G$ be the $\pi_1$-action on $\Or_2$.
There is a homeomorphism $h: \Or_1 \to \Or_2$ such that
$\tau_g = h^{-1} \circ \sigma_g \circ h$ for all $g\in G$.
We assign orientations on $\Or_1, \Or_2$ so that 
$h$ preserves their orientations.

By Theorem \ref{Fried's surgery theorem},
there is a collection of
closed orbits $\gamma_1, \gamma_2, \ldots, \gamma_n$ of $\phi_1$ such that
$M_1 - \bigcup_{i=1}^{n} \gamma_n$ is 
a fibration of punctured surfaces over $S^{1}$.
Let $\Upsilon_1 = p^{-1}_{1}(\bigcup_{j=1}^{n} \gamma_j)$ and
let $\Omega_1 = \omega_1(\Upsilon_1) \subseteq \Or_1$.
We note that $\Omega_1$ is a discrete subset of $\Or_1$ 
invariant under the $\pi_1$-action on $\Or_1$.

Set $\Omega_2 = h(\Omega_1)$. Then $\Omega_2$ is a $\pi_1$-equivariant discrete subset of $\Or_2$.
Let $\Upsilon_2 = \omega^{-1}_{2}(\Omega_2)$.
Then $p_2(\Upsilon_2)$ is a finite union of closed orbits of $\phi_2$,
and $p_2 \circ \omega^{-1}_{2} \circ  h \circ \omega_1 \circ p^{-1}_{1}$ sends
$\gamma_1, \ldots, \gamma_n$ to these closed orbits of $\phi_2$.
We denote by $\eta_j$ the image of $\gamma_j$ under 
$p_2 \circ \omega^{-1}_{2} \circ  h \circ \omega_1 \circ p^{-1}_{1}$.

Let $M^{'}_{1} = M_1 - \bigcup_{j=1}^{n} \gamma_j$,
and let $M^{'}_{2} = M_2 - \bigcup_{j=1}^{n} \eta_j$.
Let $\phi^{'}_{i}$ be the restriction of $\phi_i$ to $M^{'}_{i}$.
We denote by $\w{M^{'}_{i}}$ the universal cover of $M^{'}_{i}$,
$\w{\phi^{'}_{i}}$ the pull-back of $\phi^{'}_{i}$ in $\w{M^{'}_{i}}$,
$\w{\Or_i}$ the orbit space of $\w{\phi^{'}_{i}}$.
Since $\w{M_i} - \Upsilon_i$ is a cover of $M^{'}_{i}$,
$\w{M^{'}_{i}}$ is also the universal cover of $\w{M_i} - \Upsilon_i$.
Because $\w{M^{'}_{i}} \cong \w{\Or_i} \times \R$ and
$\w{M_i} - \Upsilon_i \cong (\Or_i - \Omega_i) \times \R$,
we can ensure that $\w{\Or_i}$ is the universal cover of $\Or_i - \Omega_i$.
We let $q_i: \w{\Or_i} \to \Or_i - \Omega_i$ be the covering map, let
 $\mathcal{D}(\w{\Or_i})$ be the group of deck transformations for this covering,
let $\mathcal{A}(\w{\Or_i})$ be the $\pi_1(M_i')$-action on $\w{\Or_i}$ 
(induced from the deck transformations of the covering
$\w{M^{'}_{i}} \to M^{'}_{i}$), 
and let $\mathcal{A}(\Or_i - \Omega_i)$ be the $\pi_1(M_i)$-action on $\Or_i - \Omega_i$
(induced from the deck transformations of the covering
$\w{M_i} - \Upsilon_i \to M^{'}_{i}$).
We now have a long exact sequence

\[1 \longrightarrow \mathcal{D}(\w{\Or_i})
\longrightarrow \mathcal{A}(\w{\Or_i})\cong \pi_1(M'_i)
\longrightarrow \mathcal{A}(\Or_i - \Omega_i)\cong \pi_1(M_i)
\longrightarrow 1.\]

Let $h: \Or_1 \to \Or_2$ be the homeomorphism that gives a conjugation between the two $\pi_1(M_i)$ actions. $h$ induces 
a homeomorphism $h^{'}: \Or_1 - \Omega_1 \to \Or_2 - \Omega_2$.
Let $\w{h}: \w{\Or_1} \to \w{\Or_2}$ be a lift of 
$h^{'}: \Or_1 - \Omega_1 \to \Or_2 - \Omega_2$.
For any $f \in \mathcal{A}(\w{\Or_1})$,
$\w{h} \circ f \circ \w{h}^{-1}$ is an element of $\mathcal{A}(\w{\Or_2})$.
We can see that
$$\w{h} \circ fg \circ \w{h}^{-1} =
(\w{h} \circ f \circ \w{h}^{-1}) \circ (\w{h} \circ g \circ \w{h}^{-1})$$
for all $f, g \in \mathcal{A}(\w{\Or_1})$.
Hence, $\w{h}$ induces an isomorphism
$$\w{h}_*: \mathcal{A}(\w{\Or_1}) \to \mathcal{A}(\w{\Or_2}).$$

If $S$ is a fiber of $M^{'}_{1}$ with monodromy $\psi$ whose suspension flow is $\varphi_1$, then $\pi_1(S)$ is a normal subgroup of $\pi_1(M^{'}_{1})$ (or of $\mathcal{A}(\w{\Or_1})$, viewed from the perspective of the action on $\w{\Or_1}$) which acts on $\w{\Or_1}$ as a convergence group. Furthermore, $\psi$ is the same map as the first return map of the flow $\phi^{'}_{1}$. 

This implies that $\w{h}_*(\pi_1(S)) \subset \mathcal{A}(\w{\Or_2})$ acts on $\w{\Or_2}$ as a convergence group as well. Hence, $\w{\Or_2}/\w{h}_*(\pi_1(S))$ is homeomorphic to $S$. This gives us a family of homeomorphic copies of $S$ in $M^{'}_{2}$ transverse to the flow and intersect every flow line. 

Furthermore, $\w{h}_*$ preserves the semi-direct product structure $\pi_1(S) \rtimes \langle \psi \rangle$ and $\w{h}_*(\psi)$ (which equals $\w{h} \circ \psi \circ \w{h}^{-1}$) would be the first return map of the flow $\phi^{'}_{2}$. Hence $M^{'}_{2}$ is also fibered and there exists a homeomorphism from $M^{'}_{1}$ to $M^{'}_{2}$ which maps $\phi^{'}_{1}$ to $\phi^{'}_{2}$. It is clear that we can extend this homeomorphism over the closed orbits $\gamma_1, \ldots, \gamma_n$. Proposition \ref{equivalent} is now proved. 

\begin{orientable double cover}\rm\label{orientable double cover}
The manifold $M_1$ has a well-defined double cover $M_*$,
branched over the set of odd-indexed singular orbits of $\phi_1$,
so that the stable foliation of $\phi_1$ pulls back to an orientable foliation in $M_*$.
Let $\phi_*$ denote the pull-back of $\phi_1$ in $M_*$.
Then Definition \ref{smooth pseudo-Anosov flow} still holds for $\phi_*$ and thus
$\phi_*$ is a pseudo-Anosov flow of $M_*$.
The orientation of $\Ors(\phi_*)$ determines a well-defined transverse orientation on 
$\Oru(\phi_*)$,
hence the unstable foliation of $\phi_*$ is also orientable.
$(M_*,\phi_*)$ is said to be the \emph{orientable double cover} of $(M_1,\phi_1)$.

We note that the $\pi_1$-action on $\Or(\phi_*)$ can be constructed from
the $\pi_1$-action on $\Or(\phi_1)$ directly.
For simplicity,
we use the same notations as in the proof of Proposition \ref{equivalent},
and suppose further that 
\[\Omega_1 = \{\text{odd-indexed singularities of } \Or_1\}.\] 
Then $\Ors_1 \mid_{\Or_1 - \Omega_1}, \Oru_1 \mid_{\Or_1 - \Omega_1}$ are a pair of orientable foliations with possible even-indexed singularities, and they lift to a pair of orientable foliations of $\w{\Or_1}$. We denote by $\w{\Ors_1}, \w{\Oru_1}$ this pair of foliations on $\w{\Or_1}$.
$\mathcal{D}(\w{\Or_1})$ has an index $2$ subgroup $F$ preserving the orientation on $\w{\Ors_1}$, and $F$ induces a covering $\w{\Or_1} \to \Or^{'}_{*}$, where $\Or^{'}_{*}$ is a double cover of $\Or_1 - \Omega_1$. Note that $\Or^{'}_{*}$ is still homeomorphic to a punctured $2$-plane. Let $\Or_*$ denote the $2$-plane obtained from filling the punctures of $\Or^{'}_{*}$.
$\mathcal{A}(\w{\Or_1})$ also has an index $2$ subgroup preserving the orientation on $\w{\Ors_1}$, which descends to an action on $\Or_*$. Let $\mathcal{A}(\Or_*)$ denote this induced action on $\Or_*$. $\mathcal{A}(\Or_*)$ is exactly the $\pi_1$-action of the orientable double cover of $\phi_1$. We call $\mathcal{A}(\Or_*)$ the \emph{orientable doubling} of the $\pi_1$-action on $\Or_1$.

We note that any reduced pseudo-Anosov flow $\varphi$ also has a orientable double cover $\varphi_*$, where the $\pi_1$-action of $\Or(\phi_*)$ is still the orientable doubling of the $\pi_1$-action on $\Or(\phi)$.
\end{orientable double cover}

\subsection{Bifoliated planes and Anosov-like actions}\label{subsec: bifoliated plane}

The concept of bifoliated plane was introduced by
Barthelm\'e, Frankel and Mann in \cite{BFM22},
to describe the orbit spaces of pseudo-Anosov flows.
They also introduced Anosov-like actions to describe the behavior of 
$\pi_1$-actions on the orbit space.

\begin{defn}\rm
$(\mathcal{P}, \mathcal{F}^{+}, \mathcal{F}^{-})$ is called a \emph{bifoliated plane}, if
$\mathcal{P}$ is a topological $2$-plane,
$\mathcal{F}^{+}, \mathcal{F}^{-}$ is a pair of $1$-dimensional (possibly singular) foliations whose leaves 
are properly embedded $n$-prongs ($n \geqslant 2$),
$\mathcal{F}^{+}, \mathcal{F}^{-}$ have the same set of $n$-prong ($n \geqslant 3$) singularities,
and the leaves of $\mathcal{F}^{+}, \mathcal{F}^{-}$ are transverse except at those
$n$-prong ($n \geqslant 3$) singularities.    
\end{defn}

Various properties of orbit spaces of pseudo-Anosov flows can be interpreted 
within the framework of bifoliated planes, see \cite[Section 2]{BFM22}.
In \cite[Section 3]{BFM22},
the ideal boundary for a bifoliated plane is 
constructed by generalizing the approach in \cite{Fen12}.

Motivated by the dynamics of $\pi_1$-actions on the orbit spaces for transitive pseudo-Anosov flows in closed $3$-manifolds,
Barthelm\'e-Frankel-Mann introduce the Anosov-like actions on bifoliated planes 
\cite{BFM22}:

\begin{defn}[Anosov-like action; {\cite[Definition 2.3]{BFM22}}]\rm\label{Anosov-like}
	Let $G$ be a group acting on a bifoliated plane $(\mathcal{P}, \mathcal{F}^{+}, \mathcal{F}^{-})$.
	$G$ is called an \emph{Anosov-like action} if
	\begin{enumerate}
	    \item[(A1)] If $g \in G - \{1\}$ fixes a leaf $\lambda \in \mathcal{F}^{\pm}$,
	then $g$ fixes exactly one point $x \in \lambda$,
	topologically expands one leaf of $\mathcal{F}^{\pm}$ containing $x$ and
	topologically contracts the other.
	\item[(A2)] $G$ has a dense orbit.
	\item[(A3)] $\{t \in \mathcal{P} \mid g(t) = t \text{ for some } g \in G - \{1\}\}$ is dense in $\mathcal{P}$.
	\item[(A4)]
	For any $x \in \mathcal{P}$ fixed by some nontrivial element of $G$,
	$\{g \in G \mid g(x) = x\}$ is an infinite cyclic group. Moreover, every singular point in $\mathcal{P}$ has nontrivial stabilizer. 
	\item[(A5)]
	If $\lambda_1, \lambda_2$ are (regular) leaves in $\mathcal{F}^{+}$ or $\mathcal{F}^{-}$ which
	are non-separated in its leaf space,
	then some $g \in G - \{1\}$ fixes $\lambda_1, \lambda_2$ simultaneously.
	\item[(A6)]
	$\mathcal{P}$ contains no totally ideal quadrilateral.
 \end{enumerate}
\end{defn}

We remark that the nontriviality of stabilzers of singular points in (A4) is not part of the definition of Anosov-like action in \cite{BFM22}. But in the case that $\mathcal{P}$ is the orbit space of a pseudo-Anosov flow in a closed 3-manifold, the condition is always satisfied. However, if the manifold is not compact, this condition may not be satisfied even in the case that $\mathcal{P}$ is the orbit space of a pseudo-Anosov flow in a 3-manifold.

Conditions (A2), (A3) describe the transitivity of pseudo-Anosov flows in closed $3$-manifolds: they ensure that a dense orbit exists and the closed orbits form a dense set.
For the background on condition (A5),
we refer the reader to \cite{Fen98, Fen99}; , and for condition (A6), see
\cite[Section 4]{Fen16}. Conditions (A1) and (A4) are standard; in fact, $\pi_1$-actions for the reduced pseudo-Anosov flows also satisfy (A1), (A4).
Thus, (A1) and (A4) play a crucial role in our future discussions.

\begin{rmk}\label{rmk: orientable cover on bifoliated plane}
Recall from the last subsection,
we define the orientable doubling of the $\pi_1$-action for a pseudo-Anosov flow or a reduced pseudo-Anosov flow in a $3$-manifold, to describe the $\pi_1$-action for the orientable double cover.
This notion generalizes to all actions on bifoliated planes with properties (A1), (A4) directly.
\end{rmk}

\begin{que}[Barthelm\'e-Frankel-Mann; {\cite[Question 1]{BFM22}}]\label{question for plane}
	Let $G$ be a group acting on a bifoliated plane $(\mathcal{P}, \mathcal{F}^{+}, \mathcal{F}^{-})$ faithfully,
	which is an Anosov-like action (Definition \ref{Anosov-like}).
	\begin{enumerate}[(a)]
	    \item 	Assume that $(\mathcal{P}, \mathcal{F}^{+}, \mathcal{F}^{-})$ is neither trivial nor skew.
	Is there a closed $3$-manifold $M$ that realizes the action of $G$ on $\mathcal{P}$ in the following way:
	$G = \pi_1(M)$ and 
	the action of $G$ on $\mathcal{P}$ is conjugate to
	the $\pi_1$-action on the orbit space of a pseudo-Anosov flow of $M$?
 \item If the answer to (a) is negative, how can we characterize the actions of $G$ on $\mathcal{P}$ which can not be realized by such a closed $3$-manifold? Which additional conditions are required to make the conclusion of (a) true?
	\end{enumerate}
\end{que}


The idea of constructing a bifoliated plane from an action on the circle with circle laminations, and constructing a 3-manifold from the bifoliated plane has been popularized recently. One of the earliest work in this direction is Thurston's extended convergence group \cite{thurston1997three}. Another systematic study can be found in a series of work of Schleimer-Segerman: in particular, \cite{schleimer2024loom} constructs a class of bifoliated planes called \emph{loom spaces} which is in some sense a flattened version of a veering triangulation of $\mathbb{R}^3$. Baik-Jung-Kim \cite{baik2022groups} follows this theme and gave a sufficient condition for pairs of circle laminations to produce loom spaces, hence produce 3-manifolds with veering triangulations. More recently, Barthelm\'e-Bonatti-Mann \cite{BBM} generalizes this greatly and gives a necessary and sufficient condition for pairs of prelaminations to produce bifoliated planes with isolated prong singularities. We will review the main result of \cite{BBM} in the next section. In this paper, we (and also an upcoming paper of Barthelm\'e-Fenley-Mann) build 3-manifolds from bifoliated planes in this general setting. 

\subsection{(Almost) laminations on $S^{1}$ and laminar groups}\label{sec:BBM}
We will characterize the pairs of circle (almost) laminations that induce a bifoliation on the plane which arise from pseudo-Anosov flows of 3-manifolds. In this section, we collect preliminaries for this. 

Recall that a circle lamination is a closed subset of $\mathcal{M} = S^1 \times S^1 - \Delta/(x,y) \sim (y,x)$ consisting of unlinked pairs of points. Any circle lamination has a geometric realization which is a geodesic lamination on $\mathbb{H}^2$ such that the set of endpoints of leaves is precisely the given circle lamination under a suitable identification between the circle and the ideal boundary of $\mathbb{H}^2$. A gap of a circle lamination is a connected component of the complement (i.e., a complementary region) of its geometric realization in $\mathbb{H}^2$. A leaf of the lamination which is part of the boundary of a gap (i.e., a geodesic contained in the closure of the gap) is called a side of the gap. 

If we only want bifoliated planes coming from pseudo-Anosov flows without perfect fits, the notion of circle lamination is enough and this is already done by Baik-Jung-Kim \cite{baik2022groups}. Recall that a bifoliated plane admits a natural compactification into a closed disk by a circle boundary. If the endpoints of the leaves of a bifoliation on the circle boundary produce a pair of laminations $\Lambda^\pm$, then we say $\Lambda^\pm$ is induced by the bifoliation. One of the main results of \cite{baik2022groups} is to give a sufficient condition for $\Lambda^\pm$ to be induced by a pseudo-Anosov flow without perfect fits and without cataclysms (as we introduce below). 
   
On the other hand, we would like to include the case of pseudo-Anosov flows with perfect fits and cataclysms. In order to do this, we need to generalize the definition of lamination a little bit. In fact, a recent preprint of Barthelm\'e-Bonatti-Mann \cite{BBM} already handles this. We develop our own language here, but we also review and compare their conventions. Whenever appropriate, we borrow their terminologies and results with explicit remarks. 

   A subset of $\mathcal{M}$ consisting of unlinked pairs of points is called an almost lamination if it can be obtained from a circle lamination by removing at most one side of each gap. An almost lamination may not be closed in $\mathcal{M}$. Hence, this is a special case of what is called a prelamination in Barthelm\'e-Bonatti-Mann \cite{BBM}. 
   
    A gap with one side removed is called a \emph{cataclysm} and the removed side (which is a geodesic in the closure of the cataclysm but not an element of the almost lamination) is called the pivot of the cataclysm. Note that our cataclysm does not have exactly the same meaning as the cataclysm in the sense of Calegari-Dunfield \cite{CD03}. Their cataclysm means a collection of non-separated leaves. Our cataclysm also includes the case that these leaves have a well-defined limit but the limit is not a leaf (such as a half-plane gap with boundary removed). 
   
   We also recall a few commonly used definitions / notations (for instance, as in \cite{baik2022groups}). For a pair of laminations or almost laminations, they are said to be \emph{transverse} if there is no common leaf. For an almost lamination $\Lambda$, we use $\lame(\Lambda)$ to denote the set $\{a \in S^1 : \{a, b\} \in \Lambda\}$ i.e., the set of endpoints of leaves. When we have two laminations $\Lambda^+, \Lambda^-$, we say a gap $P^+$ in $\Lambda^+$ and a gap $P^-$ in $\Lambda^-$ \emph{interleave} if their vertices (endpoints of the sides) appear in an alternating way on $S^1$ (\cite{BBM} uses the term \emph{coupled} instead of interleaving).   
Following Barthelm\'e-Bonatti-Mann \cite{BBM}, a transverse pair of foliations on the plane is called a \emph{pA-bifoliation} if they have only isolated prong singularities and no leaf contains more than one singularities. Here is the main theorem of Barthelm\'e-Bonatti-Mann:
\begin{THM}[Barthelm\'e-Bonatti-Mann \cite{BBM}] \label{thm:BBM}
    A transverse pair of almost laminations $(\Lambda^+, \Lambda^-)$ is induced by a pA-bifoliation on the plane if and only if  the pair $(\Lambda^+, \Lambda^-)$ is bifoliar, i.e., the following conditions hold:
    \begin{itemize}
        \item[(i)] $\lame(\Lambda^+) \cup \lame(\Lambda^-)$ is dense in $S^1$.
        \item[(ii)] For any $\alpha, \beta \in \Lambda^+ \cup \Lambda^-$, there exists a sequence $\alpha = \alpha_1, \ldots, \alpha_k = \beta$ such that $\alpha_i$ and $\alpha_{i+1}$ are linked for all $i = 1, \ldots, k-1$. 
        \item[(iii)] At most countably many leaves in $\Lambda^+ \cup \Lambda^-$ can share a common endpoint. 
        \item[(iv)] All gaps of $\Lambda^+, \Lambda^-$ are either finite-sided ideal polygons or cataclysms. 
        \item[(v)] For any ideal polygon gap $P$ of one of $\Lambda^+, \Lambda^-$, there exists an ideal polygon gap $P'$ in another lamination so that $P$ and $P'$ interleave. 
        \item[(vi)] No two ideal polygon gaps share a side. 
        \item[(vii)] For any cataclysm $C$ of one of $\Lambda^+, \Lambda^-$, leaves of another lamination intersecting $C$ must be linked with the pivot of $C$. 
    \end{itemize}
    In this case, the pA-bifoliation is unique up to a homeomorphism. Hence, a group action on $S^1$ preserving such a pair of almost laminations extends to an action on the closed disk preserving the associated pA-bifoliation. 
\end{THM}

We make a remark about condition (vi) in the above theorem. Even though a side cannot be shared by two polygons,  two cataclysms, or one cataclysm and one polygon, can still share a side.

\section{Realizing topological flow} \label{sec:realizingtopflow}

In this section, we will construct $3$-manifolds with flows from group action on bifoliated planes.
In addition to the properties (A1)-(A6) in the definition of an Anosov-like action on the bifoliated plane, we introduce another property which is always satisfied by the $\pi_1$-action on the orbit space of a transitive pseudo-Anosov flow. 

An action of a group $G$ on a bifoliated plane $(\PP, \Fa, \Fb)$ has Property (A7) if 

\begin{defn}[Property (A7)]\rm\label{A.7}
	Let $\lambda$ be a leaf of $\mathcal{F}^{+}$ or $\mathcal{F}^{-}$,
	and let $x, y$ be distinct points in $\lambda$.
	Then there are neighborhoods $U, V$ of $x, y$ respectively such that
	there is no $g \in G - \{1\}$ with
	$g(x) \in U, g(y) \in V$.
\end{defn}

The following lemma shows that indeed we can consider Property (A7) as the seventh property of the action on the orbit space. 

\begin{LEM}\label{property 1}
	Let $\phi$ be a transitive pseudo-Anosov flow in a closed $3$-manifold $M$.
	Then the $\pi_1$-action on the orbit space satisfies Property (A7).
\end{LEM}

\begin{proof}\rm
Because $\phi$ is transitive, by Theorem \ref{Fried's surgery theorem},
there is an oriented punctured surface $S$, a pseudo-Anosov homeomorphism $\varphi: S\rightarrow S$ such that $M$ is a Dehn filling of the mapping torus of $\varphi$, which we denote as $N$, and $\phi$ is induced by the suspension flow. We let $\widetilde{N}$ be the universal cover of $N$.

Let $(\mathcal{L}^{s}, \mu^{s}), (\mathcal{L}^{u}, \mu^{u})$ be the stable and unstable foliations of $\varphi$ with the corresponding transverse measures, and let $\lambda$ be the stretch factor of $\varphi$.
For any two distinct points $p, q$ on the same leaf $l$ of $\mathcal{L}^{s}$, and any integer $n$, we have
\[d_{\mu^{u}}(\varphi^{n}(p), \varphi^{n}(q)) = \lambda^{n} d_{\mu^{u}}(p,q)\]


Let $\phi_0$ be the suspension flow of $N$,
and let $\w{\phi_0}, \w{\mathcal{L}^{s}}, \w{\mathcal{L}^{u}}$ 
be the pull-back of $\phi_0, \mathcal{L}^{s}, \mathcal{L}^{u}$ in $\w{N}$.
The metric $\mu^{u}$ on $\mathcal{L}^{u}$ pulls-back to
a metric on the $\w{\mathcal{L}^{u}}$, induced by a preimage of $S$ in $\w{N}$,
which we still denote by $\mu^{u}$.
Let $\w{p}, \w{q}$ be two points in the orbit space of $\w{\phi_0}$ contained 
in the same leaf of $\w{\mathcal{L}^{s}}$.
Suppose $g\in \pi_1(N)$ moves both $\w{p}$ and $\w{q}$ only slightly, $d_{\mu^u}(g(\w{p}), g(\w{q}))$ must be very close to $d_{\mu^u}(\w{p}, \w{q})$ (so that the difference is smaller than the amount of stretching by $\w{\phi_0}$), which means that $g$ must be an element of $\pi_1(S)$ but since the action of $\pi_1(S)$ is discrete, this  implies $g=1$. The case when $\w{p}$ and $\w{q}$ are on the same leaf of $\w{\mathcal{L}^{u}}$ is similar. This shows that Property (A7) holds for the action of $\pi_1(N)$ on the orbit space of $\widetilde{N}$.

Because $M$ is a Dehn filling of $N$,
we can canonically identify $N$ with a subspace of $M$. Let $\widetilde{M}$ be the universal covering of $M$, $P: \widetilde{M}\rightarrow M$ be the covering map, and let $\widetilde{N_*} = P^{-1}(N)$, then $\widetilde{N_*}$ is a cover of $N$,
and $\widetilde{N}$ is the universal cover of $\widetilde{N_*}$.

Let $\widetilde{\phi_1}$ be the pull-back of $\phi_0$ to $\widetilde{N_*}$, and let $\mathcal{O}_0$ and $\mathcal{O}_1$ be the orbit space of $\widetilde{\phi_0}$ and $\widetilde{\phi_1}$ respectively, then $\mathcal{O}_1$ is homeomorphic to a punctured plane, and $\mathcal{O}_0$ its universal cover. Let $p': \widetilde{N}\rightarrow \widetilde{N}_*$ and $p'': \mathcal{O}_0 \to \mathcal{O}_1$ be the covering maps.

Let $\Or$ denote the orbit space of $\phi$.
As $\w{N_*}$ is a subspace of $\w{M}$,
we can canonically identify $\Or_1$ as a subspace of $\Or$.
Let $\Omega = \Or - \Or_1$,
which is a discrete set.
Let $p, q$ be two points on 
the same leaf of either of the two foliations in $\mathcal{O}$.
Let $[p,q]$ denote the closed interval on this leaf with endpoints $p, q$.
We first assume that
$[p,q]$ contains no element of $\Omega$.
Then there exists a (simply connected) compact set $U$ of $\Or_1$ such that
$p,q \in Int(U)$.
$U$ can be lifted from $\mathcal{O}_1$ to 
a compact set $\w{U}$ in $\mathcal{O}_0$.
Let $\w{p}, \w{q}$ be the corresponding lifts of $\w{p}, \w{q}$.
The argument above implies that,
there are sufficiently small neighborhoods $U_p, U_q$ of $\w{p}, \w{q}$ such that
there is no nontrivial element of the $\pi_1$-action on $\Or_0$ that
sends $\w{p}, \w{q}$ into $U_p, U_q$.
Therefore,
if $g\in\pi_1(N)$ sends both $p$ and $q$ to a sufficiently small neighborhood of them, 
then $g=1$. 

Now suppose that $[p,q]$ contains a point $x \in \Omega$.
As $x$ is a periodic point of $\Or$,
the orbit of $x$ is a discrete set.
For any compact $V$ of $\Or$ with $p,q \in Int(V)$,
$V$ contains only finitely many points in the orbit of $x$.
This implies that there are finitely many $g \in G$ that
takes $p,q$ to some neighborhoods of them. By transitivity, shrinking the neighborhoods further we can get $g=1$.




\end{proof}

We provide an alternative proof below, as this second proof works under a slightly more general assumption, which we will discuss after the proof.

\begin{proof}
Let $p: \w{M} \to M$ be the universal cover of $M$ and
let $\w{\phi}$ denote the pull-back of $\phi$ in $\w{M}$.
Let $\Or$ denote the orbit space $\Or(\phi)$,
let $\pi: \w{\phi} \to \Or$ denote the projection map,
and let $\Ors, \Oru$ denote the pair of 
invariant singular foliations on $\Or$ 
(induced from the stable and unstable foliations of $\phi$ respectively).
Let $d$ be a metric of $M$,
which pulls-back to a metric of $\w{M}$.
For convenience,
we still denote this pulled-back metric by $d$.
For any leaf $\lambda$ of $\Ors$,
$d$ induces a path metric on $p^{-1}(\lambda)$ (which we denote by $d_\lambda$),
where this metric is a Finsler metric when $\lambda$ is a singular leaf.
We refer the reader to \cite[Pages 627-628]{Fen99}.
For any $t \in \w{M}$ and $\epsilon > 0$,
we denote by $N_\epsilon(t)$ the closed $\epsilon$-neighborhood of $t$ in $\w{M}$.
For each $x \in \Or$,
we denote by $\gamma_x$ the orbit $\pi^{-1}(x)$ of $\w{\phi}$.

    Let $\lambda$ be a leaf of $\Ors$ and let $x,y$ be distinct points in $\lambda$.
    If $x$ is a periodic point,
    then the orbit of $x$ under $G$ is a discrete set in $\PP$,
    and the orbit of $y$ under the stabilizer of $x$ is also a discrete set in 
    $\PP - \{x\}$.
    Therefore, there exist neighborhoods $V_x$ and $V_y$ of $x$ and $y$, respectively,
    such that there is no $g \in G - \{1\}$ with $g(x) \in V_x$, $g(y) = V_y$.
    Similarly, the same holds if $y$ is a periodic point.

    Now suppose that both of $x,y$ are not periodic points.
    Let $\{\gamma_x(t)\}_{t \in \R}, \{\gamma_y(s)\}_{s \in \R}$ be
any parameterizations of $\gamma_x, \gamma_y$, respectively, such that 
    the increasing orientations on $\R$ are consistent with the positive orientations on
    $\gamma_x, \gamma_y$.
    We choose $t \in \R$,
    let $k_0 = d_\lambda(\gamma_x(t),\gamma_y)$,
    and let $s \in \R$ for which
    $d_\lambda(\gamma_x(t),\gamma_y(s)) = k_0$.
    We choose $0 < \epsilon << \frac{1}{3}k_0$ so that
    the images of $N_{2\epsilon}(\gamma_x(t)), N_{2\epsilon}(\gamma_y(s))$ in $\PP$ under
    $\pi: \w{\phi} \to \Or$ have empty intersection.

    We choose $K > k_0 + 2\epsilon$.
    There are $m,n > 0$ such that
    $d_\lambda(\gamma_x(u),\gamma_y) > K$ for all $u < -m$,
    and $d_\lambda(\gamma_x(u),\gamma_y) < \epsilon$ for all $u > n$
    \cite[Lemma 2.3]{Fen99}.    
    Note that $p(\gamma_x([-m,n]))$ is a compact segment in $M$,
    and thus $p(\gamma_x(t))$ has a neighborhood $U$ in $M$ such that
    $U \cap p(\gamma_x([-m,n]))$ is 
    a connected segment containing $p(\gamma_x(t))$.
    Let $\w{U}$ denote the lift of $U$ in $\w{M}$ with $\gamma_x(t) \in \w{U}$.
    Then \[\w{U} \cap \{g(\gamma_x([-m,n])) \mid g \in G\} = \w{U} \cap \gamma_x([-m,n]).\]
    There is a neighborhood $\w{V}$ of $\gamma_y(s)$ in $\w{M}$ such that,
    if $u \in \w{U}$, $v \in \w{V}$ and
    $\pi(u), \pi(v)$ are contained in the same leaf $\lambda^{'}$ of
    $\Ors$,
    then $d_{\lambda^{'}}(u,v) \in (k_0 - 2\epsilon, k_0 + 2\epsilon)$.
    
    There are neighborhoods $V_x, V_y$ of $x, y$ in $\Or$ such that
    for any $x_1 \in V_x, y_1 \in V_y$,
    we have $\gamma_{x_1} \cap (\w{U} \cap N_\epsilon(\gamma_x(t))) \ne \emptyset$ and
    $\gamma_{y_1} \cap (\w{V} \cap N_\epsilon(\gamma_y(s))) \ne \emptyset$.    
    Now suppose that there is $g \in G - \{1\}$ such that
    $g(x) \in V_x$, $g(y) \in V_y$.
    We choose $t_0, s_0 \in \R$ for which 
    $g(\gamma_x(t_0)) \in \w{U} \cap N_\epsilon(\gamma_x(t))$,
    $g(\gamma_y(s_0)) \in \w{V} \cap N_\epsilon(\gamma_x(t))$.
    Then \[|d_{g(\lambda)}(g(\gamma_x(t_0)),g(\gamma_y(s_0))) - k_0| < 2\epsilon.\]

    As $\w{U} \cap \{h(\gamma_x([-m,n])) \mid h \in G - \{1\}\} = \emptyset$,
    we have $t_0 \notin [-m,n]$.
    It follows that $d_{g(\lambda)}(g(\gamma_x(t_0)),g(\gamma_y)) \notin [\epsilon,K]$.
    As $d_{g(\lambda)}(g(\gamma_x(t_0)),g(\gamma_y(s_0))) < k_0 + 2\epsilon < K$,
    we have $d_{g(\lambda)}(g(\gamma_x(t_0)),g(\gamma_y)) < \epsilon$.

    Because $g(\gamma_x(t_0)) \in N_\epsilon(\gamma_x(t))$,
    we have 
    \[d(\gamma_x(t),g(\gamma_y)) < d(\gamma_x(t),g(\gamma_x(t_0))) +
    d(g(\gamma_x(t_0)),g(\gamma_y))
    < \epsilon + \epsilon = 2\epsilon.\]
    However,
    $d(\gamma_y(s),g(\gamma_y)) < \epsilon$. 
    Recall that the images of $N_{2\epsilon}(\gamma_x(t)), N_{2\epsilon}(\gamma_y(s))$
    under $\pi: \w{\phi} \to \Or$ have empty intersection.  
    This contradicts 
    \[\gamma_x \cap N_{2\epsilon}(\gamma_x(t)), 
    \gamma_y \cap N_\epsilon(\gamma_y(s)) \ne \emptyset.\]

    Hence,
    if $g \in G$ satisfies $g(x) \in V_x, g(y) \in V_y$,
    then $g = 1$.
\end{proof}

\begin{rmk}\rm\label{non-transitive}
    In the second proof of Lemma \ref{A.7},
    we do not need the condition that $\phi$ is transitive.
    Thus Lemma \ref{A.7} also holds for all non-transitive pseudo-Anosov flows.

    Suppose that $\phi$ is a non-transitive topological Anosov flow of $M$.
    Then the second proof of Lemma \ref{A.7} still holds for $\phi$.
    The only place that does not follow directly from Definition \ref{topological Anosov flow} is the existence of $m > 0$ for which $d_\lambda(\gamma_x(u),\gamma_y) > K$ when $u < -m$,
    where this follows from \cite[Lemma 5.12, Proposition 5.14]{BFP23}.
\end{rmk}

\begin{THM}\label{topological flow}
	Assume that a group $G$ acts on bifoliated plane $\mathcal{P}$ that sends leaves to leaves, and is {\rm orientation-preserving}, namely preserves the given leafwise orientation on $\mathcal{F}^{+}$.
	Suppose $\mathcal{P}$ is a bifoliated plane without singular leaves which satisfies Property (A7). 
	Then there is a (possibly non-compact) $3$-manifold with fundamental group $G$ that admits a flow $\phi$ realizing this action.
\end{THM}

\begin{proof}
	For each $x \in \mathcal{P}$,
	let $l_x$ be the leaf of $\mathcal{F}^{+}$ that contains $x$, let $\gamma^{'}_{x}: [0, +\infty) \to l_x$ be a parameterization of the positive side of $x$ in $l_x$ which we denote as $r'_x$,
	and let $\gamma_x: (0, +\infty) \to l_x$ be a parametrization of the interior of $r^{'}_{x}$, which we denote as $r_x$.
    Note that, because $G$ is orientation-preserving,
	if $g(x) = y$ ($g \in G$),
	then $g(r_x) = r_y$.
 
	Let $N =\{(x, y)\in\mathcal{P}^2: y\in r_x\}$. It is easy to see that $N$ is homeomorphic to $\mathbb{R}^3$. Also note that the product pairs of the type as in Figure \ref{flow box figure} (a) can be used as charts of $N$ (see definitions in Section \ref{sec:frombifoliatedplane}). Hence $N$ naturally admits a 1-dimensional foliation whose leaves are of the form $\{x\} \times r_x$. The $G$-action on $\mathcal{P}$ induces a $G$-action on $N$ by $g(x, y)=(gx, gy)$, and the action preserves the 1-dimensional foliation. We give $N$ the subspace topology of $\mathcal{P}\times\mathcal{P}$.
	Property (A7) guarantees that the action of $G$ on $N$ is free and discrete; hence $N / G$ is a $3$-manifold and $\{\{x\} \times r_x\}$ descends to a foliation in $N / G$, which we can then use to define a flow. 
\end{proof}

    We call actions that satisfy (A1), (A4) and (A7) \emph{flowable}. Also note that we translate the definition of flowable action using the language of almost laminations in Section \ref{sec:circle-lamination}.

\section{Characterizations of the Anosov flows}\label{sec:char_anosov}

In this section, we give some necessary and sufficient conditions on the bifoliated plane with group action such that the flow constructed in the previous section is topologically Anosov. In fact, we give two characterizations of Anosov flows in this framework. The first one is given in Section \ref{subsec:firstchar} in terms of convergence and divergence properties. These properties correspond to topological stretching and compressing behavior along the leaves of the bifoliation. The second one is given in Section \ref{subsec:secondchar} in terms of expansive property. If the action on the bifoliated plane is expansive, this would give us an Anosov flow via work of either \cite{inaba1990nonsingular} or \cite{paternain1993expansive}. See 
Section \ref{subsec:secondchar} for details. 

Throughout this section, we will use the same notation as in the previous section.  
Let $(\PP, \Fa, \Fb)$ be a bifoliated plane without singular leaves, fix an orientation on $\mathcal{F}^{\pm}$, and let $G$ be a group acting on $\PP$ that preserves the orientation on $\Fa$. 

\begin{conv}\rm
\begin{enumerate}[(a)]
    \item Suppose that $\lambda$ is a leaf of $\Fa$,
$\mu$ is a leaf of $\Fb$,
and $\lambda$ intersects $\mu$.
We will use $(\lambda,\mu)$ to denote the point $\lambda \cap \mu \in \PP$. 
\item For any point $x \in \PP$, we will always denote by $\lambda_x, \mu_x$ the leaves of $\mathcal{F}^{\pm}$ for which $x = (\lambda_x,\mu_x)$. We will always use $r_x$ to denote the component of $\lambda_x - \{x\}$ lying on the positive side of $x$.
\end{enumerate}
\end{conv}

Let $N = \bigcup_{x \in \PP}(\{x\} \times r_x)$ and
let foliation $\w{\phi} = \{\{x\} \times r_x \mid x \in \mathcal{P}\}$.
By Theorem \ref{topological flow},
$G$ acts on $N$ freely and discretely.
Let $M = N / G$,
which is a $3$-manifold.
Let $\phi = \w{\phi} / G$,
which is a flow of $M$.
We let 
$$\rho_x = \{x\} \times r_x,$$
$$\rho_x(t) = (x, t) \in \{x\} \times r_x$$
Here $t\in r_x$. Hence, we can now write $\rho_x(t)$, where $x=(\lambda, \mu)$, as $\rho_x(\lambda_t,\mu_t)$, where $t=(\lambda_t, \mu_t)$. 
We assign to each flow line $\gamma=\{x\}\times r_x$ of $\w{\phi}$ 
a canonical orientation consistent with
the positive orientation on $r_x$.
This orientation on $\w{\phi}$ induces an orientation on the flow lines of $\phi$.

Note that 
$$\{\bigcup_{x \in \lambda} \rho_x \mid \lambda \text{ is a leaf of } \Fa\},$$
$$\{\bigcup_{x \in \mu} \rho_x \mid \mu \text{ is a leaf of } \Fb\}$$
are a pair of foliations in $N$,
and they descend to a pair of foliations in $M$.
We denote them by $\Ea, \Eb$ respectively.

\begin{prop}
    Suppose that $\phi$ is an Anosov flow of $M$. 
    Then $\Ea$ is the stable foliation of $\phi$ and
    $\Eb$ is the unstable foliation of $\phi$.
\end{prop}
\begin{proof}
    Suppose that $x$ is a periodic point of $\PP$,
    and let $g$ be the generator of the stabilizer of $x$ such that
    $g$ topologically expands $\lambda_x$ and topologically contracts $\mu_x$.
    Then $g$ topologically expands $r_x \subseteq \lambda_x$.

    Let $L = \bigcup_{t \in \lambda_x} \rho_t$,
    which is a leaf of the pulled-back foliation of $\Ea$ in $N$.
    As $g$ translates $L$ along the positive orientation on $\rho_x$,
    all $\rho_t$ with $t \in r_x$ topologically contract along 
    the positive orientation on $\rho_x$ and thus are forward asymptotic to $\rho_x$.
    Hence $\Ea$ is the stable foliation of $\phi$.
\end{proof}

We denote by $L(\Fa), L(\Fb)$ the leaf spaces of $\Fa, \Fb$.

\begin{defn}[Flow box]\rm\label{flow box}
A \emph{flow box} $B$ is is a subspace 
$$B = (\Ia \times \Ib) \times \Ir$$ of $N$,
where $\Ia$, $\Ib$, and $\Ir \cong [0,1]$ satisfy:
\begin{itemize}
    \item $\Ia$ (resp. $\Ib$) is a closed interval in $L(\Fa)$ (resp. $L(\Fb)$),
and any leaf of $\Ia$ intersects all leaves of $\Ib$.
Then
$\Ia \times \Ib$ is canonically identified a product chart of $\PP$,
where $\Ia \times \{t\}$ is a closed interval in a leaf of $\Fb$ for all $t \in \Ib$,
$\{s\} \times \Ia$ is a closed interval in a leaf of $\Fa$ for all $s \in \Ia$.
\item $\Ir$ is a closed interval in $L(\Fb)$ such that,
for any $(x,y) \in \Ia \times \Ib \subseteq \PP$,
all leaves in $\Ir$ intersect $r_{(x,y)}$.
\end{itemize}

For any $x \in \Ia, y \in \Ib, z \in \Ir$,
the point $((x,y),z)$ in this flow box will always refer to the point 
$\rho_{(x,y)}(x,z) \in N$ (where $(x,y) \in \PP$, and $(x,z) \in r_{(x,y)}$).
\end{defn}

\begin{defn}[Compactness property]\label{cptprop}\rm	
There are finitely many flow box $B_i = (\Ia_i \times \Ib_i) \times \Ir_i$
($i = 1,\ldots,n$) such that,
for any $p \in \PP$ and $t \in r_p$,
there is $i \in \{1,\ldots,n\}$, 
$x \in \Ia_i, y \in \Ib_i, z \in \Ir_i$, $g \in G$ such that
$g(\rho_p(t)) = ((x,y),z)$.
Equivalently,
$\{g(B_i) \mid g \in G, i \in \{1,\ldots,n\}\}$ covers $N$.
\end{defn}
 
Note that if the action of $G$ on $\PP$ satisfies the Compactness property, then $M$ is a closed $3$-manifold. This is simply because any topological space that can be covered by finitely many compact subsets is compact. 

\subsection{First Characterization} \label{subsec:firstchar}

In this section, we introduce the Convergence and Divergence properties of an action on the bifoliated plane that correspond to \ref{topological Anosov flow} (c) and (d) for the associated flow. 

\begin{defn}[Convergence property]\label{convprop}\rm	
Let $B_i = (\Ia_i \times \Ib_i) \times \Ir_i$ ($i = 1,\ldots,n$) be a collection of flow boxes such that
$\{g(Int(B_i)) \mid g \in G, i \in \{1,\ldots,n\}\}$ covers $N$. We say the {\rm Convergence Property} holds, if:
\begin{enumerate}[(I)]
    \item Let $\lambda$ be a leaf of $\Fa$ and 
let $a = (\lambda,\mu_a)$, $b = (\lambda,\mu_b)$ be
two points on $\lambda$.
Given an arbitrary finite collection of open connected intervals
$\{U_\alpha \mid \alpha \in \Psi\}$ (where $\Psi$ is an index set) that 
covers $\bigcup_{i=1}^{n}\Ib_i$,
there exists $u \in r_a \cap r_b$ such that,
for all $t = (\lambda,\mu_t) \in r_u$,
there is $g \in G$ and $i \in \{1,\ldots,n\}$ such that
$g(\rho_a(t)), g(\rho_b(t)) \in B_i$ and 
$g(\mu_a), g(\mu_b) \in U_\alpha$ for some $\alpha \in \Psi$.
\item Let $\mu$ be a leaf of $\Fb$ and 
let $a = (\lambda_a,\mu)$, $b = (\lambda_b,\mu)$ be
two points on $\mu$.
For any finite collection of open connected intervals
$\{U_\alpha\}_{\alpha \in \Psi}$ (where $\Psi$ is an index set) that covers 
$\bigcup_{i=1}^{n}\Ia_i$,
there exists $u \in r_a$, $v \in r_b$ such that
$u, v$ are contained in the same leaf of $\Fb$ and the following properties hold:
For all $t_1 \in r_a\backslash r_u$, $t_2 \in r_b\backslash r_v$ that 
are contained in the same leaf $\mu_t$ of $\Fb$,
there is $g \in G$ and $i \in \{1,\ldots,n\}$ such that
$g(\rho_a(t_1)), g(\rho_b(t_2)) \in B_i$,
and $g(\lambda_a), g(\lambda_b) \in U_\alpha$ for some $\alpha \in \Psi$.
\end{enumerate}
\end{defn}

\begin{prop} \label{prop:def2.2c}
If the action of $G$ on $\PP$ satisfies the Convergence property,
then the flowlines on the same leaf of $\Ea$ converge in the positive direction,
the flowlines on the same leaf of $\Eb$ converge in the negative direction.
\end{prop}
\begin{proof} Pick the $U_\alpha$ to be sufficiently small, such that when restricted to the flow boxes they represent flow lines that are really close to one another, then one can see the conclusion.
\end{proof}

\begin{defn}[Divergence property]\rm\label{divprop}
We say the {\rm Divergence Property} holds, if
for any connected compact region $R$  of $N$,
let $C$ be the image of $R$ under the projection 
$\pi: N \to \PP$, we have: \begin{enumerate}[(I)]
    \item Let $\lambda$ be a leaf of $\Fa$ and 
let $a = (\lambda,\mu_a), b = (\lambda,\mu_b)$ be distinct points on $\lambda$.
Then there is $u \in r_a$ such that,
for any $t = (\lambda, \mu_t) \in r_a - r_u$,
there is no $g \in G$ such that 
$g(\rho_a(t)) \in R$ and
$g(b) \in C$.
\item Let $\mu$ be a leaf of $\Fb$ and 
let $c = (\lambda_c,\mu), d = (\lambda_d,\mu)$ be distinct points on $\mu$.
Then there is $u \in r_c$ such that,
for any $t = (\lambda_c, \mu_t) \in r_u$,
there is no $g \in G$ such that 
$g(\rho_c(t)) \in R$ and
$g(b) \in C$.
\end{enumerate}
\end{defn}

\begin{prop} \label{prop:def2.2d}
If the action of $G$ on $\PP$ satisfies the Divergence Property,
then 

(a)
For any two flowlines $\alpha, \beta: \mathbb{R} \to M$
on the same leaf $\lambda$ of $\Ea$, 
$\lim_{t \to -\infty}d_\lambda(\alpha(t), \beta) = +\infty$ 
(where $d_\lambda$ is the path metric on $\lambda$).

(b)
For any two flowlines $\gamma, \delta: \mathbb{R} \to M$
on the same leaf $\mu$ of $\Eb$, 
$\lim_{t \to +\infty}d_\mu(\gamma(t), \delta) = +\infty$ 
(where $d_\mu$ is the path metric on $\mu$).
\end{prop}
\begin{proof}
Lift the path metrics to the leaves of the lifted foliations in the universal cover of $M$, and pick $R$ to be the set of points reachable by concatenated paths on these foliations whose total length, under the path metric, is bounded by some $C>>1$.
\end{proof}

\begin{prop}
If the Compactness, Convergence, Divergence properties hold,
then $\phi$ is an Anosov flow.
\end{prop}
\begin{proof} From Proposition \ref{prop:def2.2c} and Proposition \ref{prop:def2.2d}, we know that the Convergence property implies Definition \ref{topological Anosov flow} (c),
and the Divergence property implies Definition \ref{topological Anosov flow} (d).
\end{proof}

Now we prove

\begin{THM}
Let $\phi_0$ be an Anosov flow of a closed $3$-manifold $M_0$.
Suppose that the action of $G$ on $(\PP,\Fa,\Fb)$ is
the $\pi_1$-action of $\Or(\phi_0)$,
then the action of $G$ on $(\PP,\Fa,\Fb)$ satisfies
the Compactness, Convergence, Divergence properties.
\end{THM}

By Proposition \ref{equivalent},
there is a homeomorphism $h: M_0 \to M$ with $h(\phi_0) = \phi$.
We assign $M$ a metric $d$ defined by
$d_M(x,y) = d_{M_0}(h^{-1}(x), h^{-1}(y))$ for all $x,y \in M$.
Replacing $(\phi_0,d_{M_0})$ by $(\phi,d_M)$,
we can ensure that $\phi$ is a topological pseudo-Anosov flow.
The metric $d_M$ on $M$ pulls-back to a metric on $N$.
We denote by $d$ the pull-back metric of $d_M$ on $N$. 

\begin{proof}[The proof of the Compactness property]
A set in $M$ is called a \emph{flow box} in $M$ if 
it can be lifted to a flow box in $N$.

For each point $x$ of $M$,
there is a flow box $B_x$ of $M$ such that $x \in Int(B_x)$.
    So there exist finitely many flow boxes in $M$ such that
    their interiors cover $M$.
\end{proof}

\begin{proof}[The proof of the Convergence property (I)]
Let $B_i = (\Ia_i \times \Ib_i) \times \Ir_i$ ($i = 1,\ldots,n$) be 
a collection of flow boxes such that
$\{g(Int(B_i)) \mid g \in G, i \in \{1,\ldots,n\}\}$ covers $N$.
Given an arbitrary finite collection of open connected intervals
$\{U_\alpha \mid \alpha \in \Psi\}$ (where $\Psi$ is an index set) that 
covers $\bigcup_{i=1}^{n}\Ib_i$.

There exists $\epsilon > 0$ such that,

$\bullet$
If $d(x,y) < \epsilon$ for $x,y \in N$,
then there is $g \in G$ such that $g(x), g(y) \in B_i$ for some $i$.

$\bullet$
Given $x = (\lambda,\mu_x) ,y = (\lambda,\mu_y) \in B_i$ contained in the same leaf
$\lambda$ of $\Fa$,
if $d(\rho_x(t),\rho_y) < \epsilon$ for some $t \in \Ir_i$,
then there is 
$\alpha \in \Psi$ such that
$\mu_x, \mu_y \in U_\alpha$.

Let $\lambda$ be a leaf of $\Fa$ and 
let $a = (\lambda,\mu_a)$, $b = (\lambda,\mu_b)$ be two points on $\lambda$.
Now we prove that the condition of Convergence property (I) holds for $a,b$.
Because $\phi$ is a topological Anosov flow,
there is $u \in r_a \cap r_b$ such that
    for any $t = (\lambda,\mu_t) \in r_u$,
    $d(\rho_a(t),\rho_b) < \epsilon$.
    
We choose $t \in r_u$.
Then there is $s \in r_b$ with
$d(\rho_a(t),\rho_b(s)) < \epsilon$.
Then there is $g \in G$ such that
$g(\rho_a(t)), g(\rho_b(s)) \in B_i$ for some $i \in \{1,\ldots,n\}$.
And this implies that
there is $\alpha \in \Psi$ such that $g(\mu_a), g(\mu_b) \in U_\alpha$.

From $g(\rho_a(t)) \in B_i$,
we can know that $g(t) \in \Ir_i$,
and therefore $g(\rho_b(t)) \in B_i$.
This completes the proof.
\end{proof}

\begin{proof}[The proof of the Convergence property (II)]
Let $B_i = (\Ia_i \times \Ib_i) \times \Ir_i$ ($i = 1,\ldots,n$) be 
a collection of flow boxes such that
$\{g(Int(B_i)) \mid g \in G, i \in \{1,\ldots,n\}\}$ covers $N$.
Given an arbitrary finite collection of open connected intervals
$\{U_\alpha \mid \alpha \in \Psi\}$ (where $\Psi$ is an index set) that 
covers $\bigcup_{i=1}^{n}\Ia_i$.

There exists $\epsilon > 0$ such that,

$\bullet$
If $d(x,y) < \epsilon$ for $x,y \in N$,
then there is $g \in G$ such that $g(x), g(y) \in B_i$ for some $i$.

$\bullet$
Given $x = (\lambda_x,\mu) ,y = (\lambda_y,\mu) \in B_i$ contained in the same leaf
$\mu$ of $\Fb$,
if $d(\rho_x(t),\rho_y) < \epsilon$ for some $t \in \Ir_i$,
then there is 
$\alpha \in \Psi$ such that
$\mu_x, \mu_y \in U_\alpha$.

Let $\mu$ be a leaf of $\Fa$ and 
let $a = (\lambda_a,\mu)$, $b = (\lambda_b,\mu)$ be two points on $\mu$.
Because $\phi$ is a topological Anosov flow,
there is $u \in r_a$ such that
    for any $t = (\lambda,\mu_t) \in r_a - r_u$,
    $d(\rho_a(t),\rho_b) < \epsilon$.
    
We choose $t \in r_a - r_u$.
Then there is $s \in r_b$ with
$d(\rho_a(t),\rho_b(s)) < \epsilon$,
and so there is $g \in G$ such that
$g(\rho_a(t)), g(\rho_b(s)) \in B_i$ for some $i \in \{1,\ldots,n\}$.
This implies that
there is $\alpha \in \Psi$ such that $g(\lambda_a), g(\lambda_b) \in U_\alpha$.

Hence $g(t) \in \Ir_i$.
Let $t_0$ be the intersection of $\lambda_b$ and $\mu_t$.
Then both of $g(\rho_a(t)), g(\rho_b(t_0))$ are contained in $B_i$.
This completes the proof.
\end{proof}

\begin{proof}[The proof of the Divergence property (I)]
Let $R$ be a connected compact region of $N$,
and let $C$ be the image of $R$ under the projection 
$\pi: N \to \PP$.

Recall that $\Ea = \{\bigcup_{x \in \lambda} \rho_x(t) \mid 
\lambda \text{ is a leaf of } \Fa\}$ is a foliation of $N$.
There is a constant $K > 0$ such that,
if $l$ is a leaf of $\Ea$ and $u,v \in l \cap R$,
then $d_l(u,v) < K$,
where $d_l$ denotes the path metric on $l$ induced from the metric $d$ of $N$.

Let $\lambda$ be a leaf of $\Fa$ and 
let $a = (\lambda,\mu_a), b = (\lambda,\mu_b)$ be distinct points in $\lambda$.
Let $l$ denote the leaf $\bigcup_{x \in \lambda} \rho_x(t)$ of $\Ea$,
and let $d_l$ denote the path metric on $l$.

There is $u \in r_a$ such that,
for any $t = (\lambda, \mu_t) \in r_a - r_u$,
$d_l(t,\rho_b) > K$.
Suppose that there is $g \in G$ such that
$g(\rho_a(t)) \in R$ and
$g(b) \in C$.
Then there is $s \in r_b$ such that $g(\rho_b(s)) \in R$,
and thus
$d_l(\rho_a(t),\rho_b(s)) < K$.
This is a contradiction.
\end{proof}

\begin{proof}[The proof of the Divergence property (II)]
This is completely analogous to the proof of (I) above.
\end{proof}

We finish this section by observing that the Convergence Property can be stated in terms of the metrics on leaf spaces as well. Recall that the maximal Hausdorff quotients of $L(\Fa), L(\Fb)$ 
are underlying spaces of $\R$-trees \cite[Section 4]{RSS03}.
Let $T_+$ (resp. $T_-$) denote the maximal Hausdorff quotient of 
$L(\Fa)$ (resp. $L(\Fb)$),
and we assign a metric $d_+$ (resp. $d_-$) to $T_+$ (resp. $T_-$).
Then $L(\Fa)$ (resp. $L(\Fb)$) has a metric induced from
$d_+$ (resp. $d_-$).
For simplicity,
we still denote by $d_+$ (resp. $d_-$) the metric on $L(\Fa)$ (resp. $L(\Fb)$).

Then we can restate the convergence property as follows:

\begin{defn}[Convergence property]\rm	
Let $B_i = (\Ia_i \times \Ib_i) \times \Ir_i$ ($i = 1,\ldots,n$) be 
a collection of flow boxes such that
$\{g(Int(B_i)) \mid g \in G, i \in \{1,\ldots,n\}\}$ covers $N$.

(I)
Let $\lambda$ be a leaf of $\Fa$ and 
let $a = (\lambda,\mu_a)$, $b = (\lambda,\mu_b)$ be
two points on $\lambda$.
Given an arbitrary $\epsilon > 0$,
there exists $u \in r_a \cap r_b$ such that,
for all $t = (\lambda,\mu_t) \in r_u$,
there is $g \in G$ and $i \in \{1,\ldots,n\}$ such that
$g(\rho_a(t)), g(\rho_b(t)) \in B_i$ and 
$d_-(g(\mu_a),g(\mu_b)) < \epsilon$.

(II)
Let $\mu$ be a leaf of $\Fb$ and 
let $a = (\lambda_a,\mu)$, $b = (\lambda_b,\mu)$ be
two points on $\mu$.
Given an arbitrary $\epsilon > 0$,
there exists $u \in r_a$, $v \in r_b$ such that
$u, v$ are contained in the same leaf of $\Fb$ and the following properties hold:
For all $t_1 \in r_a - r_u$, $t_2 \in r_b - r_v$ that 
are contained in the same leaf $\mu_t$ of $\Fb$,
there is $g \in G$ and $i \in \{1,\ldots,n\}$ such that
$g(\rho_a(t_1)), g(\rho_b(t_2)) \in B_i$,
and $d_+(g(\lambda_a),g(\lambda_b)) < \epsilon$.
\end{defn}

\subsection{Second Characterization} \label{subsec:secondchar}

Throughout this section, we always assume that the Compactness property holds. We make use of \cite[Definition 5.7]{BFP23} as follows:

\begin{defn}[Expansive property]\rm
Let $B_i = (\Ia_i \times \Ib_i) \times \Ir_i$ ($i = 1,\ldots,n$) be 
a collection of flow boxes such that
$\{g(Int(B_i)) \mid g \in G, i \in \{1,\ldots,n\}\}$ covers $N$.
We say the {\em Expansive Property} holds, if
there is some $\epsilon > 0$ such that:\\

Let
$a = (\lambda_a, \mu_a), b = (\lambda_b,\mu_b) \in \PP$ be
a pair of distinct points.
There exists $g \in G$ such that
$g(a) \in B_i$ for some $i \in \{1,\ldots,n\}$ and
$d_+(g(\lambda_a), g(\lambda_b)) + d_-(g(\mu_a),g(\mu_b)) > \epsilon$.
\end{defn}

\begin{LEM}
    The Expansive Property implies that $\phi$ is an Anosov flow.
\end{LEM}

\begin{proof}
    For any $C>0$, we shall prove that there is $k>0$ such that: for any distinct $x, y\in M$, let $\alpha, \beta$ be flow lines starting at $x$ and $y$, (i.e. $\alpha, \beta: \mathbb{R}\rightarrow M$ are flow lines such that $\alpha(0)=x$, $\beta(0)=y$), such that if there is some increasing homomorphism $h: \mathbb{R}\rightarrow\mathbb{R}$ such that $d(\alpha(t), \beta(h(t)))<k$ for all $t$ then $y\in\alpha([-C, C])$. 

    Let $K>0$ be a positive number smaller than the systole of $M$. Now we show that for any $\epsilon>0$, there is some $k_2 > 0$ such that, for any $i$, any $c = (\lambda_c, \mu_c) \in \Ia_i \times \Ib_i$ and any
    $d \in \PP$ with $d_+(\lambda_c,\lambda_d) + d_-(\mu_c,\mu_d) > \epsilon$,
    we have $d(\rho_c(t),\rho_d) > k_2$ for some $t \in \Ir_i$. Here $\rho_d$ is the flow line passing through $d$ as a set. To show this,
    for any $c \in \Ia_i \times \Ib_i$,
    let
    \[V_\epsilon(c) = 
    \{p \in \PP \mid d_+(\lambda_c,\lambda_p) + d_-(\mu_c,\mu_p) \leqslant \epsilon\}\]
    then there is $v_c > 0$ such that,
    if $d(\rho_c(t),\rho_p) \leqslant v_c$ for all $t \in \Ir_i$, 
    then $p \in V_\epsilon(c)$.
    As $\bigcup_{i=1}^{n} (\Ia_i \times \Ib_i)$ is a compact set,
    there is $k_2 > 0$ such that $k_2 < v_c$ for all $c \in \Ia_i \times \Ib_i$.

    Now we let $k = \min\{K/3,k_2\}$.
    Let $x, y \in M$ such that the flow lines $\alpha$, $\beta$ passing through them have $d(\alpha(t), \beta(h(t)))<k$ for all $t\in\mathbb{R}$, then the paths can be lifted to 
    $\w{\alpha}, \w{\beta}: \R \to N$ such that
    $d(\w{\alpha}(0), \w{\beta}(0) < k$. We shall show that $d(\w{\alpha}(t), \w{\beta}(h(t)) < k$ for all $t$. If otherwise,
    then there must be some minimum $t_1\not=0$ with $d(\w{\alpha}(t_1), \w{\beta}(h(t_1))) = k$. But we have $d(\alpha(t_1),\beta(t_1)) < k$, so there must be a non trivial loop in $M$ of length no more than $2k<K$ which is impossible.

    Now apply the Expansive condition, we know that
    if $\w{\alpha}$ and $\w{\beta}$ correspond to distinct points $a, b\in\PP$, there is $g \in G$ such that $g(a) \in B_i$ for some $i \in \{1,\ldots,n\}$ and
    $d_+(g(\lambda_a), g(\lambda_b)) + d_-(g(\mu_a),g(\mu_b)) > \epsilon$.
    This implies that $d(g(\rho_a(t)), g(\rho_b)) > k$ for 
    some $t \in r_a$, which contradicts the above conclusion. Hence $\w{\alpha} = \w{\beta}$.
    Pick a smaller $k$ if necessary,
    we can ensure that 
    $y \in \alpha([-C,C])$.
    Therefore,
    $\phi$ satisfies the conditions of expansive flow in \cite[Definition 5.7]{BFP23}.
    As $\phi$ has non-singular invariant foliations,
    $\phi$ is an Anosov flow by \cite[Theorem 5.9]{BFP23}, which is based on \cite[Theorem 1.5]{inaba1990nonsingular} or \cite[Lemma 7]{paternain1993expansive}.
\end{proof}

\begin{LEM}
    Anosov flows satisfy the Expansive property.
\end{LEM}

\begin{proof}
By Proposition 6.2, an Anosov flow
$\phi$ is also a topological Anosov flow. Compactness of $M$ implies that it can be covered by finitely many flow boxes, lift them to $N$ we get the $B_i$s. 

Let $a=(\lambda_a, \mu_a), b=(\lambda_b, \mu_b)\in\PP$ be a pair of distinct points. By 
Definition \ref{topological Anosov flow} (c), (d), we see that there is some $t \in r_a$ such that
$d(\rho_a(t),\rho_b) > \delta$.

There is $g \in G$ such that
$g(\rho_a(t)) \in B_i$ for some $B_i$, because $\{g(Int(B_i)) \mid g \in G, i \in \{1,\ldots,n\}\}$ covers $N$.

Furthermore, compactness implies that for any $K>0$, there is some $\delta>0$, such that any two points $u, v\in N$ in the same leaf $\lambda$ of $\Ea$, if $d(u,v) > \delta$, then $d_\lambda(u,v) > K$. Similarly for $\Eb$. Combining with this fact, we complete the proof.
\end{proof}







\section{The construction when the bifoliated plane has even-index singularities} \label{sec:frombifoliatedplane}

In this section,
we aim to reconstruct a $3$-manifold $M$ with a reduced pseudo-Anosov flow $\phi$ from
a flowable action on the bifoliated plane.

Let $G$ be a flowable action on a bifoliated plane $(\PP, \Fa, \Fb)$ such that $\Fa$ is orientable and the action of $G$ on $\PP$ preserves the orientations on $\Fa$. Note that each orientation on $\Fa$ induces a transverse orientation on it and therefore induces an orientation on $\Fb$. Hence $\Fb$ is also orientable, and the action of $G$ on $\PP$ also preserves the orientations on $\Fb$. We note that $\Fa, \Fb$ only have singularities with even number of prongs as they are orientable. We fix an orientation on $\Fa$ and an orientation on $\Fb$.

For any point $x \in \PP$,
we still denote by $\lambda_x$ the leaf of $\mathcal{F}^{+}$ containing $x$ and
denote by $\mu_x$ the leaf of $\mathcal{F}^{-}$ containing $x$.

\begin{conv}\rm\label{conv: leaves and segments}
(a)
Let $\lambda$ be a singular leaf of $\Fa$,
and let $s$ denote the singularity on $\lambda$.
Let $r$ be a component of $\lambda - \{s\}$.
We call $\overline{r}$ a \emph{positive half-leaf} if $r$ is in the positive side of $s$,
and call $\overline{r}$ a \emph{negative half-leaf} if $r$ is in the negative side of $s$.
Note that there exists $m \in \N$ with $n \geqslant 2$ such that
$\lambda$ has $2m$ half-leaves,
$m$ of them are positive and the other $m$ are negative.

(b)
Let $\lambda$ be a leaf of $\Fa$.
For any distinct points $x,y \in \lambda$,
we denote by $[x,y]_\lambda$ the set of points in $\lambda$ that separates $x,y$
(where $x,y \in [x,y]_\lambda$).
Note that $[x,y]_\lambda$ is always a path in $\lambda$ between $x, y$,
and thus we can assign it an orientation which goes from $x$ to $y$.
Similarly,
for a leaf $\mu$ of $\Fb$ and distinct points $u,v \in \mu$,
we denote by $[u,v]_\mu$ the set of points in $\mu$ that separate $u,v$.
\end{conv}

Let $\lambda$ be a singular leaf of $\Fa$,
and let $s$ denote the singularity on $\lambda$.
Let $p_1, q_1, p_2, q_2, \ldots, p_m, q_m$ denote all of the half-leaves of $\lambda$,
up to the clockwise sequence.
Let $\langle g \rangle$ denote the stabilizer of $\lambda$,
which is also the stabilizer of $s$.
Then $\langle g \rangle$ acts on $\{p_1, q_1, p_2, q_2, \ldots, p_m, q_m\}$ by
rotation.
Let $d \in \mathbb{N}$ be the minimum positive number such that
$g^{d}(p_i) = p_i$, $g^{d}(q_j) = q_j$.
Note that $d \mid m$.
Let $n \in \Z \mod m$ such that
$g(p_i) = p_{i+n}$.

By Property (A1),
either of the following two cases holds:
(1)
$g$ topologically expands all $p_i$ and topologically contracts all $q_j$,
(2) 
$g$ topologically contracts all $p_i$ and topologically expands all $q_j$.
Without loss of generality,
we may assume that (1) holds.

\begin{LEM}[{\cite{TZZ}}]\label{equivalence class}
There is a map $h_\lambda: \lambda \to \R$ such that

(i)
$h_\lambda(0) = 0$.

(ii)
The restriction of $h_\lambda$ to each $p_i$ is 
an orientation-preserving homeomorphism to $[0,+\infty)$,
and the restriction of $h_\lambda$ to each $q_i$ is 
an orientation-preserving homeomorphism to $(-\infty,0]$.

(iii)
If $h_\lambda(x) = h_\lambda(y)$ for $x,y \in \lambda$,
then $h_\lambda(g(x)) = h_\lambda(g(y))$.
\end{LEM}

\begin{proof}
We choose an arbitrary $t_1 \in p_1$,
and we choose $t_1 = x_0 < x_1 < x_2 < \ldots < x_{d-1} < x_d = g^{d}(t_1)$.
Let $J_i = [x_i, x_{i+1}]_\lambda$. 
We can parameterize each $J_i$ by $J_i(t)$ ($t \in I$),
where the increasing orientation on $t \in I$ is consistent with the orientation on 
$[x_i, x_{i+1}]_\lambda$.
For convenience, 
we may consider $J_{kd+s} = g^{kd}(J_s)$ for all $k \in \Z$.

Let $h^{+}_{1}: \bigcup_{i=0}^{d-1} p_{1+in} \to [0,+\infty)$ be the map such that
$h^{+}_{1}: p_1 \to [0,+\infty)$ is an an arbitrary orientation-preserving homeomorphism,
and $h^{+}_{1}(g^{k}(J_i(t))) = h_\lambda(J_{i+k}(t))$ for all $i,k \in \Z$.
Then we have $h^{+}_{1}(g(x)) = h^{+}_{1}(g(y))$ for all 
$x,y \in \bigcup_{i=0}^{d-1} p_{1+in}$.

We define $h^{+}_{k}: \bigcup_{i=0}^{d-1} p_{k+in} \to [0,+\infty)$ ($k = 2,\ldots,n-1$)
with the similar property.
Then we define 
$h^{-}_{k}: \bigcup_{i=0}^{d-1} q_{k+in} \to (-\infty,0]$ ($k = 1,\ldots,n-1$) similarly.
Let $h_\lambda: \lambda \to \R$ be the map such that
its restriction to $\bigcup_{i=0}^{d-1} p_{k+in}$ is $h^{+}_{k}$,
and its restriction to $\bigcup_{i=0}^{d-1} q_{k+in}$ is $h^{-}_{k}$.
\end{proof}

Guaranteed by the proof of Lemma \ref{equivalence class},
we can assume further that $h_\lambda$ has the following property:

\begin{assume}\label{further assumption}
    For any $x \in \lambda - \{s\}$,
    there is a neighborhood $V$ of $x$ in $\lambda$ such that
    $h_\lambda(V) \cap h_\lambda(\{g^{k}(x) \mid k \in \Z\})=\{h_\lambda(x)\}$
\end{assume}

We assign each singular leaf $\lambda$ of $\Fa$ a map $h_\lambda$ 
with properties given above,
and we may assume 

\begin{assume}\label{equivariant}
    $h_{f(\lambda)} = h_\lambda \circ f^{-1}$ for any singular leaf $\lambda$ of $\Fa$ and $f \in G$.
\end{assume}

For each singular leaf $\lambda$ of $\Fa$,
we denote by $\stackrel{\lambda}{\sim}$ the equivalence relation on $\lambda$ 
so that $u \stackrel{\lambda}{\sim} v$ if $h_\lambda(u) = h_\lambda(v)$.
For convenience,
for each non-singular leaf $\lambda$ of $\Fa$,
we also denote by $\stackrel{\lambda}{\sim}$ the trivial equivalence relation.
By Assumption \ref{equivariant},
for any leaf $\lambda$ of $\Fa$ and any $u,v \in \lambda$,
$u \stackrel{\lambda}{\sim} v$ if and only if
$f(u) \stackrel{f(\lambda)}{\sim} f(v)$.

\begin{defn}\rm
Let $x \in \PP$.
We define $R_x,r_x$ case by case:
\begin{itemize}
\item \textbf{Case I.}
$\lambda_x$ is a regular leaf.
Let $R_x$ be the component of $\lambda_x - \{x\}$ in the positive side of $x$,
and let $r_x = R_x$.
\item \textbf{Case II.}
$\lambda_x$ is a singular leaf,
$x$ is contained in a positive half-leaf of $\lambda_x$,
and $x$ is not the singularity of $\lambda_x$.
Let $R_x$ be the component of $\lambda_x - \{x\}$ in the positive side of $x$,
and let $r_x = R_x$.
\item \textbf{Case III.}
$\lambda_x$ is a singular leaf,
$x$ is contained in a negative half-leaf of $\lambda_x$,
and $x$ is not the singularity of $\lambda_x$.
Now let $s$ denote the singularity on $\lambda$,
and let $q$ denote the negative half-leaf of $\lambda_x$ that contains $x$.
Let $p^{+}, p^{-}$ be the two positive half-leaves of $\lambda_x$ that 
are adjacent to $q$,
such that $p^{+}$ is in the positive side of $q$ and
$p^{-}$ is in the negative side of $q$
(with respect to the clockwise orientation on $s$).
Define $r^{-}_{x} = ([x,s]_{\lambda_x} - \{x\}) \cup p^{-}$,
$r^{+}_{x} = ([x,s]_{\lambda_x} - \{x\}) \cup p^{+}$,
and $R_x = r^{-}_{x} \cup r^{+}_{x}$.
Define $r_x = R_x / \stackrel{\lambda_x}{\sim}$ endowed with the quotient topology.
\item\textbf{Case IV.}
$\lambda_x$ is a singular leaf and
$x$ is the singularity of $\lambda_x$.
Let $R_x$ be the union of positive half-leaves of $\lambda_x$,
and let $r_x = R_x / \stackrel{\lambda_x}{\sim}$ with the quotient topology.
\end{itemize}
\end{defn}

Be Lemma \ref{equivalence class},
$r_x$ is always homeomorphic to $(0,+\infty)$.

\subsection{The first construction}

\begin{figure}
	\centering
	\subfigure[]{
		\includegraphics[width=0.4\textwidth]{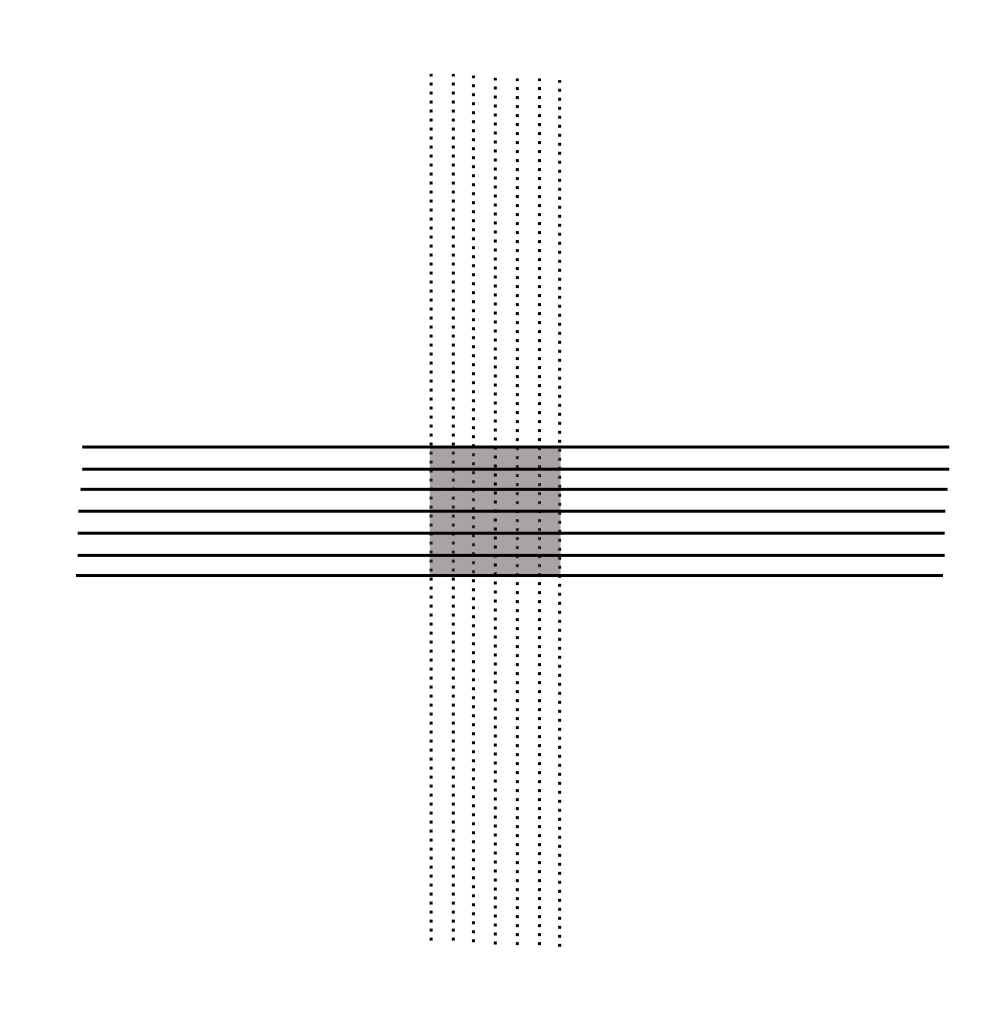}}
	\subfigure[]{
		\includegraphics[width=0.4\textwidth]{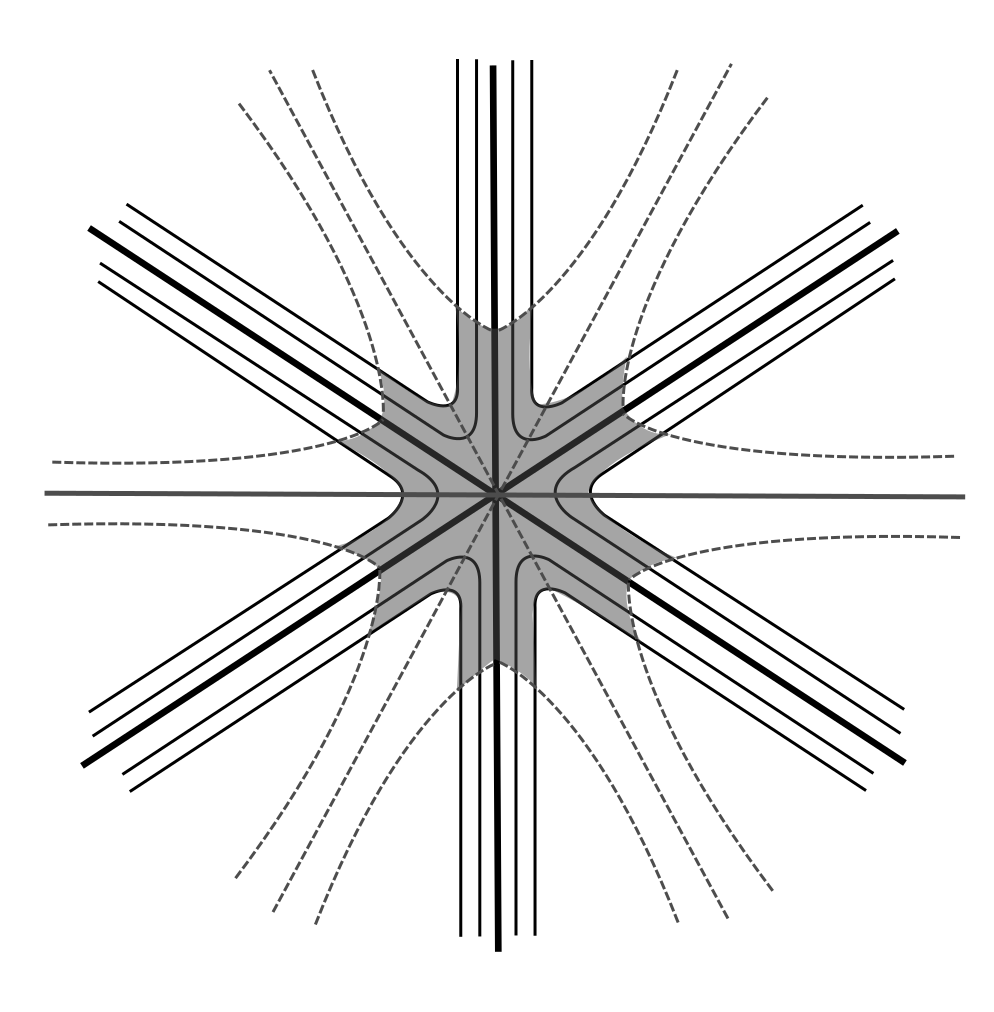}}
  \caption{(a) is the picture of a product rectangle, and (b) is the picture of a product polygon at a singularity of $\Fa$.}\label{product region}
  \end{figure}

A closed region $R$ of $\PP$ is called a \emph{product rectangle} if
there are two leaves $\lambda_1, \lambda_2$ of $\Fa$ and
two leaves $\mu_1, \mu_2$ of $\Fb$ such that
$\lambda_i \cap \mu_j \ne \emptyset$ for any $i,j \in \{1,2\}$),
and $R$ is the closed region bounded by $\lambda_1, \lambda_2,\mu_1,\mu_2$.
See Figure \ref{product region} (a).
We call $\lambda_1, \lambda_2$ the two \emph{$\Fa$-sides} of $R$ and 
call $\mu_1,\mu_2$ the two \emph{$\Fb$-sides} of $R$.

\begin{figure}
	\centering
	\subfigure[]{
		\includegraphics[width=0.3\textwidth]{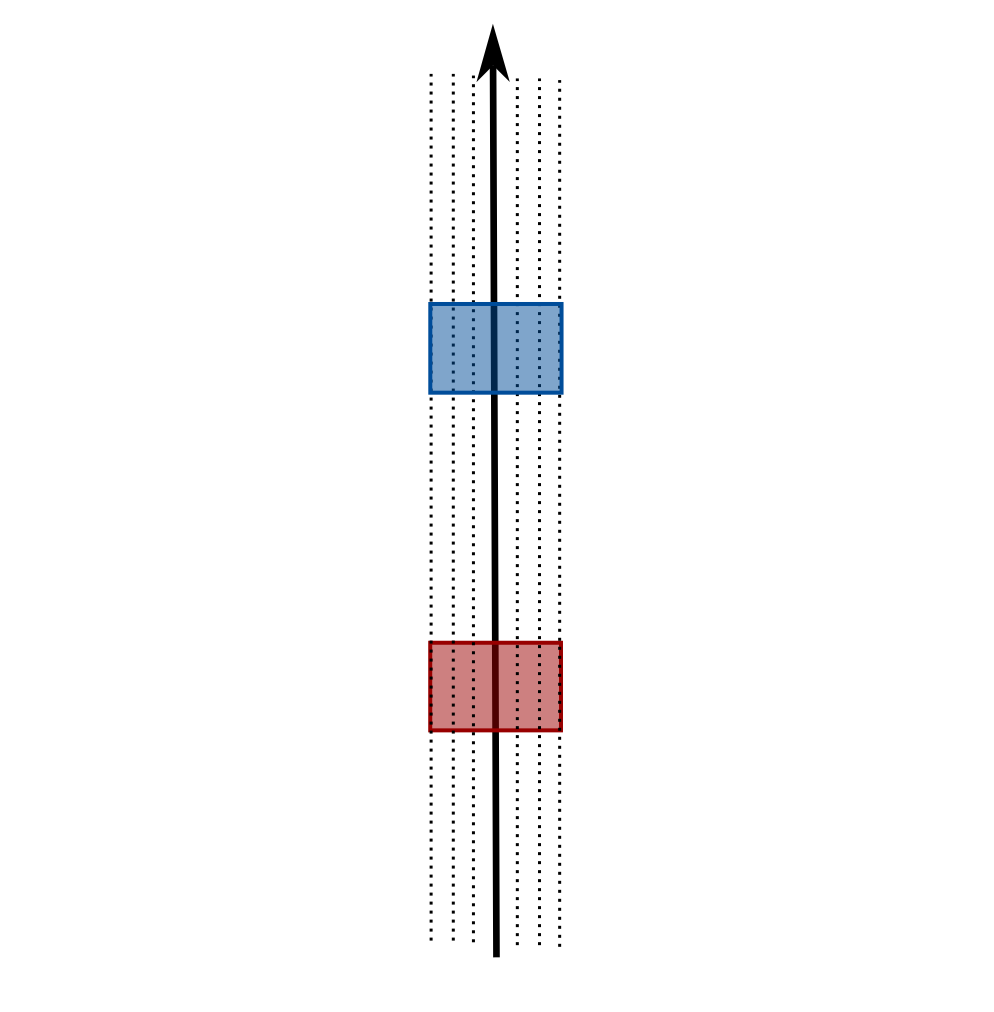}}
  \subfigure[]{
		\includegraphics[width=0.3\textwidth]{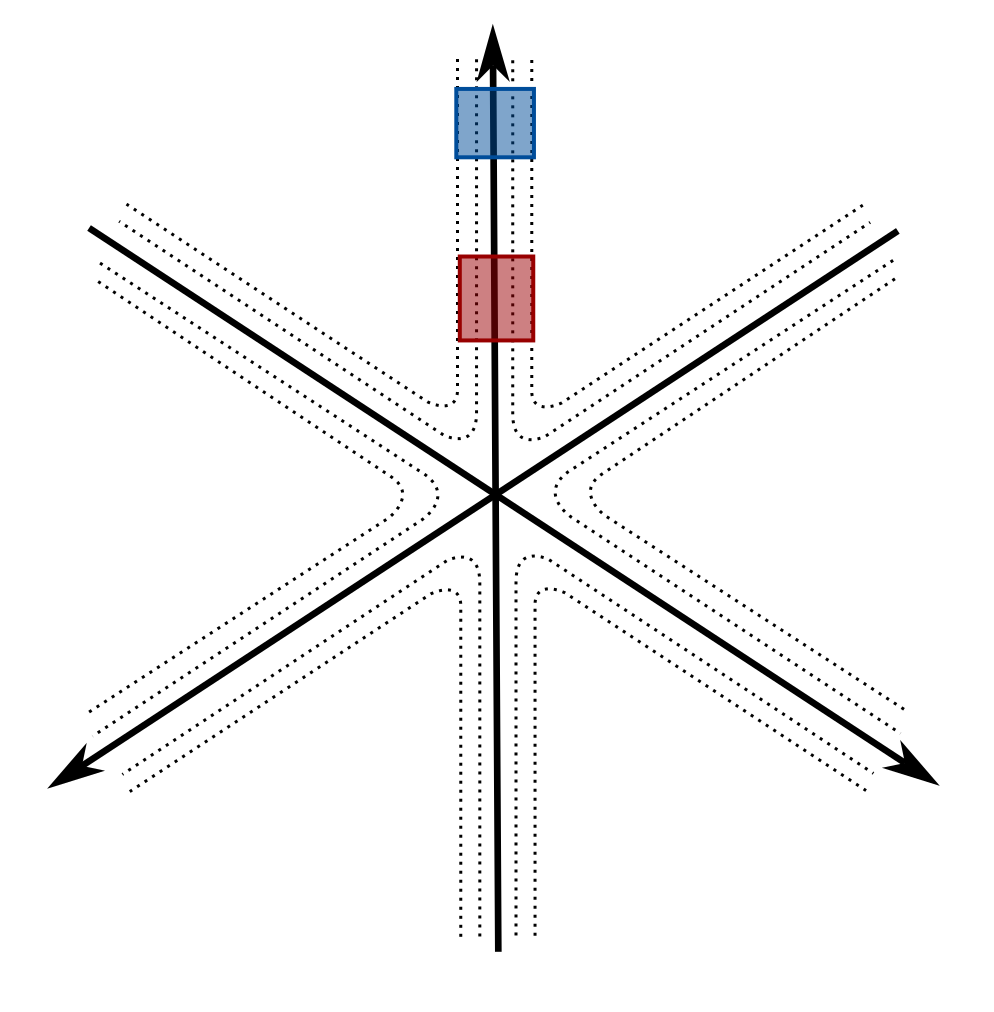}}
  \subfigure[]{
		\includegraphics[width=0.3\textwidth]{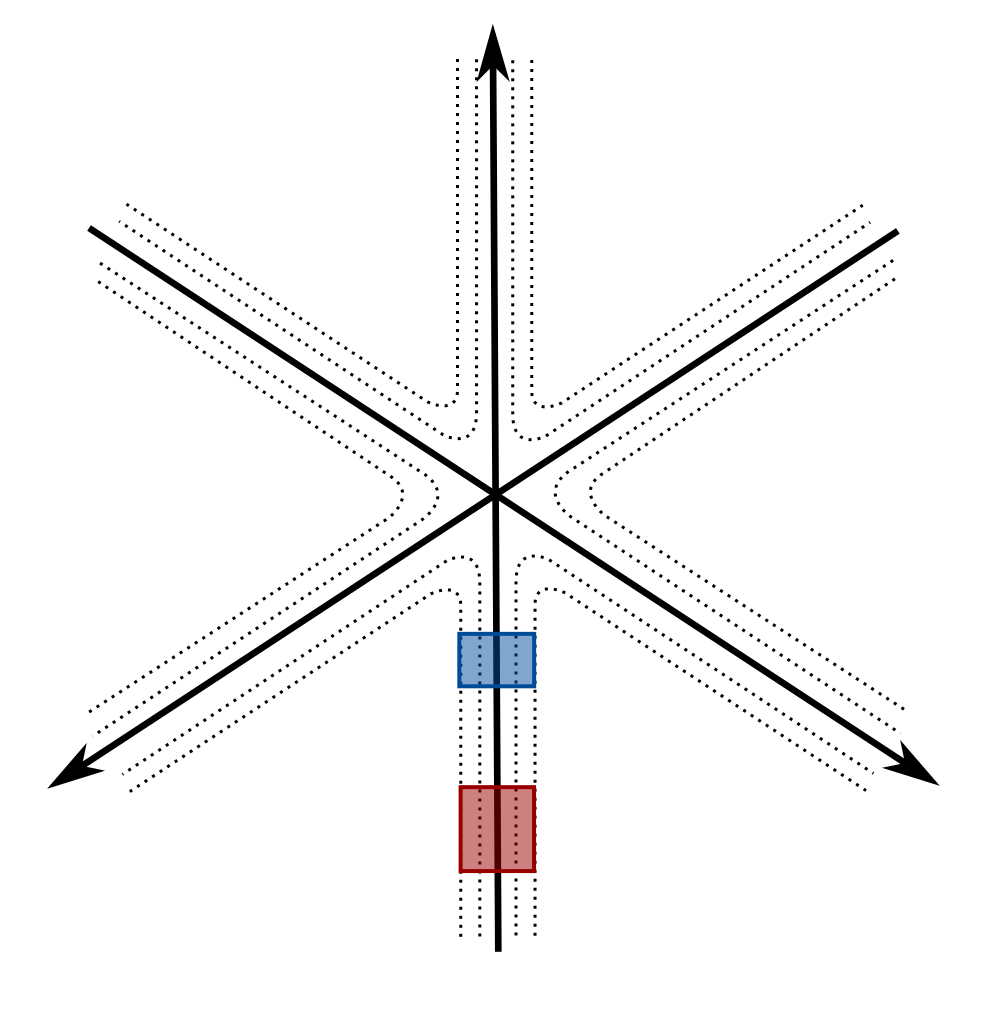}}
  \subfigure[]{
		\includegraphics[width=0.3\textwidth]{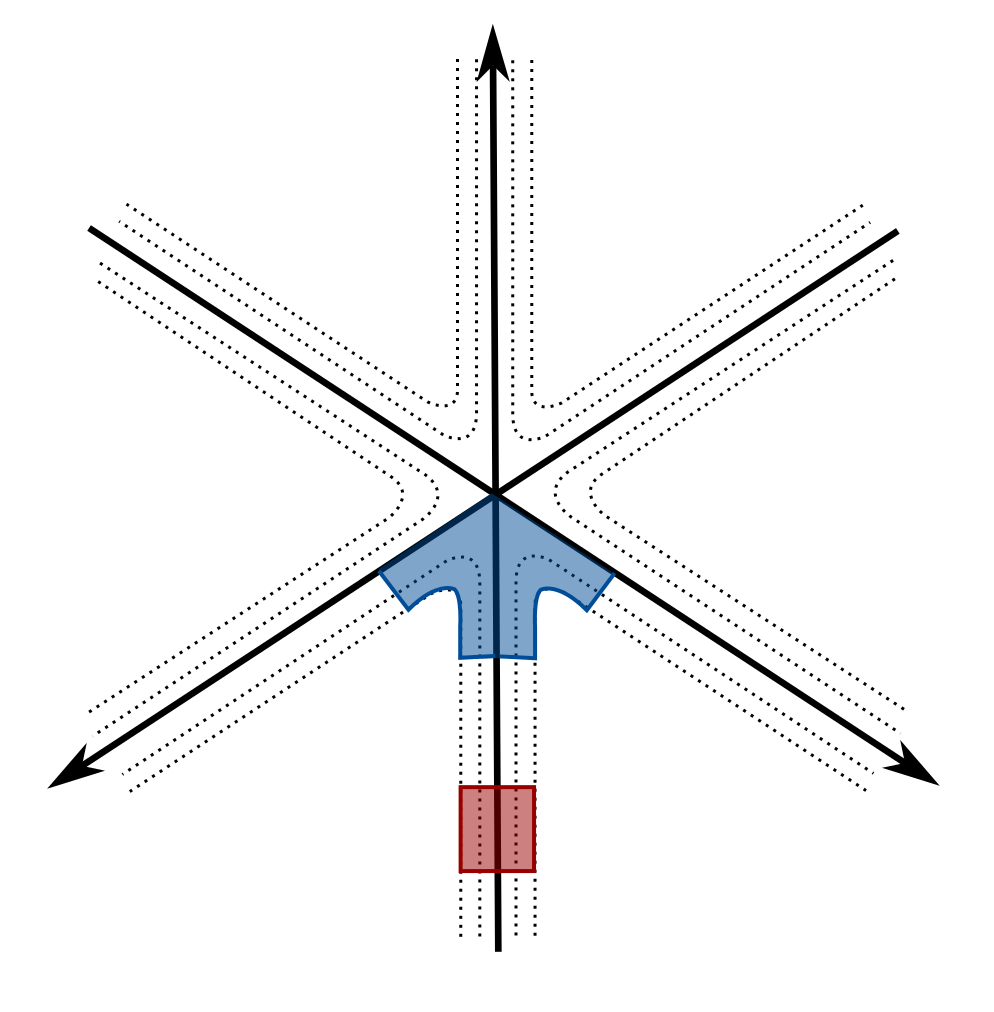}}
  \subfigure[]{
		\includegraphics[width=0.3\textwidth]{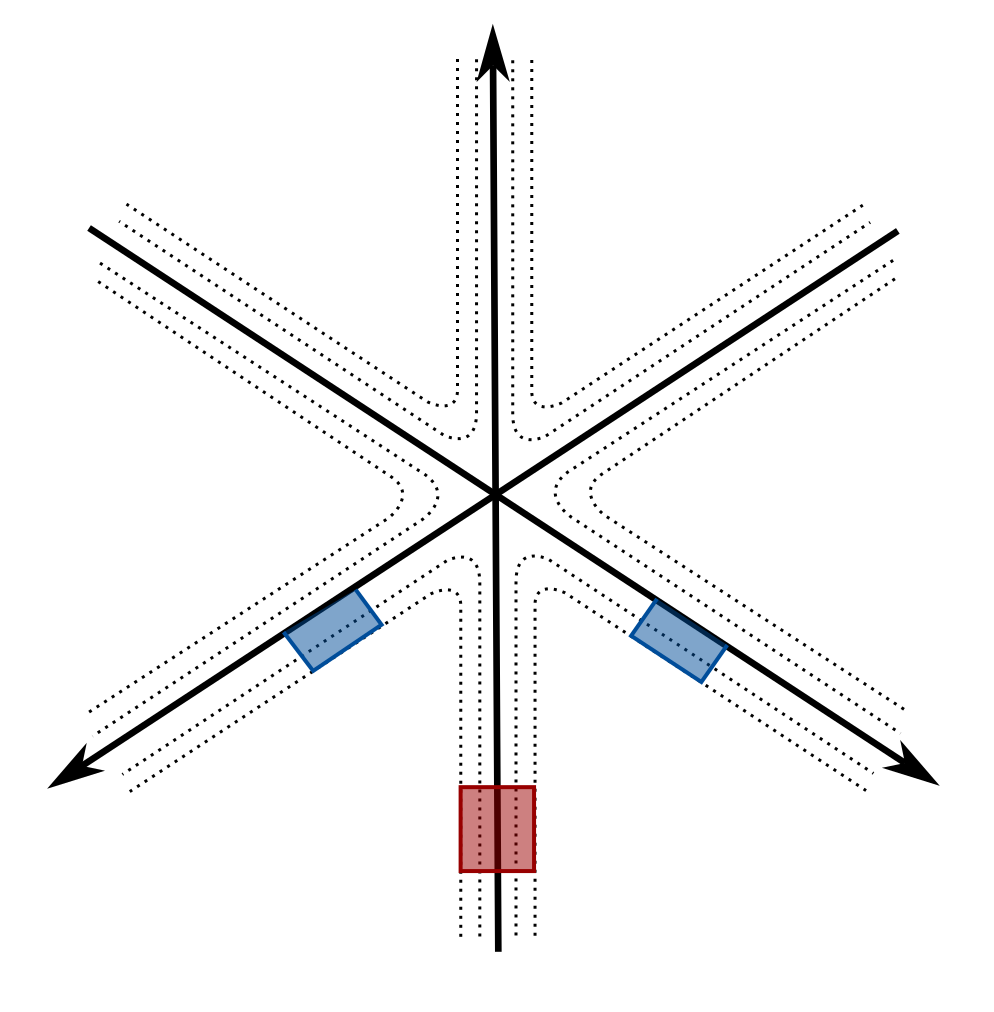}}
  \subfigure[]{
		\includegraphics[width=0.3\textwidth]{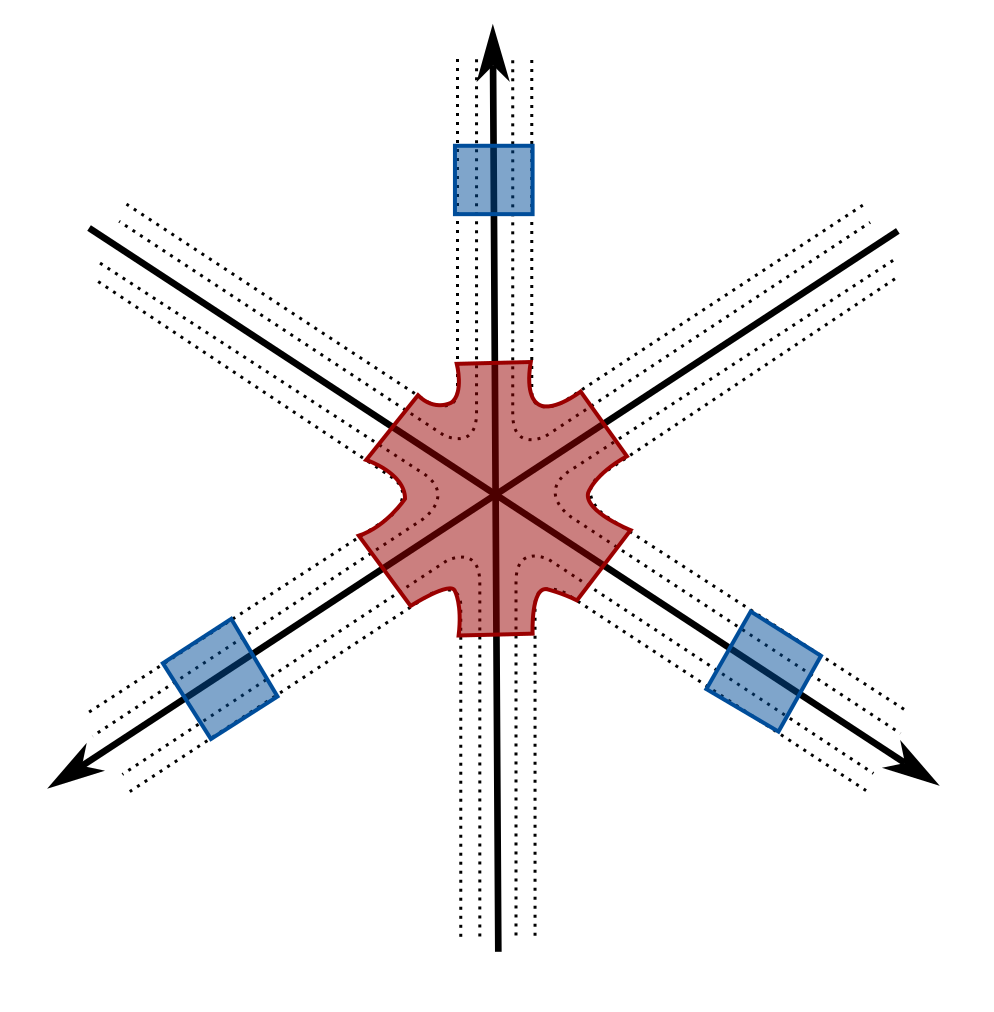}}
  \caption{In the pictures (a)$\sim$(e), we describe the local model of a product pair $(C,D)$ near $(x,\lambda_x \cap D)$ for some $x \in Int(C)$, where $\lambda_x$ is highlighted as the bolded leaf and other leaves of $\Fa$ are described by the doted leaves. The red region describes a product rectangle $N_x \subseteq C$ containing $x$, and the blue region describes the region in $D \cap \bigcup_{t \in N_x} \lambda_t$.
  In the picture (f), we describe the model of a singular product pair $(C,D)$ at a $6$-prong singularity $x$ of $\Fa$ in a similar way: the red region is the product polygon $C$ and the blue region is $D$.}\label{flow box figure}  
  \end{figure}

\begin{defn}\rm
A \emph{product pair} $(C,D)$ consists of
a pair of subsets $C,D$ of $\PP$ with the following properties:

$\bullet$
$C$ is a product rectangle,
$D$ is the union of finitely many product rectangles,
and
$C \cap D = \emptyset$.

$\bullet$
If $y \in D$,
then there is $x \in C$ such that $y \in R_x$.

$\bullet$
There are finitely many singular leaves $\lambda_1,\ldots,\lambda_n$ ($n \in \N$) 
of $\Fa$ such that,
$D$ can be decomposed into $n+1$ product rectangles $D_1,\ldots,D_{n+1}$,
and for each $i \in \{1,\ldots,n\}$,
$D_i$ has a $\Fa$-side $J^{i}_{1}$ and $D_{i+1}$ has a $\Fa$-side $J^{i+1}_2$ satisfying
$J^{i}_{1}, J^{i+1}_{2} \subseteq \lambda_i$ and
$J^{i}_{1} \stackrel{\lambda_i}{\sim} J^{i+1}_{2}$. 
\end{defn}

Here we describe the local models of $(x,y) \in (C,D)$ case by case.

\begin{rmk}
Let $(C,D)$ be a product pair and let $x \in Int(C)$.
Then $x$ is not a singularity of $\Fa$.
We choose a product rectangle $N_x \subseteq C$ with $x \in Int(N_x)$,
and let $D(N_x) = D \cap (\bigcup_{t \in N_x} \lambda_t)$.

\textbf{Case I.}
If $\lambda_x$ is non-singular leaf,
then $\lambda_x \cap D$ is a connected closed interval,
and $(N_x,D(N_x))$ is modeled as Figure \ref{flow box figure} (a) 
when $N_x$ is sufficiently small.

\textbf{Case II.}
Suppose that $\lambda_x$ is a singular leaf and 
$x$ is contained in a positive half-leaf of $\lambda_x$.
Then $\lambda_x \cap D$ is a connected closed interval and
$(N_x,D(N_x))$ is modeled as Figure \ref{flow box figure} (b) 
when $N_x$ is sufficiently small.

\textbf{Case III.}
Suppose that $\lambda_x$ is a singular leaf and 
$x$ is contained in a negative half-leaf of $\lambda_x$.
When $N_x$ is sufficiently small,
there are two connected closed intervals $J_1 \subseteq r^{+}_{x}$,
$J_2 \subseteq r^{-}_{x}$ with $h_{\lambda_x}(J_1) = h_{\lambda_x}(J_2)$, 
such that $D(N_x)$ is the union of two product rectangles for which
$J_1$ is a $\Fa$-side of one of them and
$J_2$ is a $\Fa$-side of the other.
$(N_x, D(N_x))$ is modeled as one of Figure \ref{flow box figure} (c), (d), (e).
\end{rmk}

\begin{defn}\rm
Let $x$ be a singularity of $\Fa$,
and we assume that $\lambda_x$ has $2n$-prongs ($n \in \N_{\geqslant 2}$).
A closed region $P$ of $\PP$ is called a \emph{product polygon} at $x$ if
there are leaves $\lambda_1, \ldots, \lambda_{2n}$ of $\Fa$ and
leaves $\mu_1, \ldots, \mu_{2n}$ of $\Fb$ such that
\begin{enumerate}[(1)]
\item $\lambda_x$ separates $\mu_1, \ldots, \mu_{2n}$, and
$\mu_x$ separates $\lambda_1, \ldots, \lambda_{2n}$.
\item For any $i \in \{1,\ldots,2n\}$,
$\lambda_i \cap \mu_i \ne \emptyset$,
$\lambda_i \cap \mu_{i+1} \ne \emptyset$,
where $\mu_{2n+1} = \mu_1$.
\item $P$ is the closed region bounded by $\lambda_1, \ldots, \lambda_{2n}$ and
$\mu_1,\ldots,\mu_{2n}$
(see Figure \ref{product region} (b)).
\end{enumerate}
\end{defn}

\begin{defn}\rm\label{singular product pair}
Let $x$ be a singularity of $\Fa$,
and we assume that $\lambda_x$ has $2n$-prongs ($n \in \N_{\geqslant 2}$).
A \emph{singular product pair} $(C,D)$ (at $x$) consists of
a pair of subsets $C, D$ of $\PP$ with the following properties:
\begin{itemize}
\item $C$ is a product polygon at $x$,
$D$ is the union of $n$ disjoint product rectangles $D_1,\ldots,D_n$,
and
$C \cap D = \emptyset$.
\item If $y \in D$,
then there is $x \in C$ such that $y \in R_x$.
\item For each leaf $\lambda$ of $\Fa$ such that $\lambda \ne \lambda_x$,
$\lambda \cap D$ is a connected closed interval.
\item For each $i \in \{1,\ldots,n\}$,
$R_x \cap D_i$ is a connected closed interval $J_i$ in a positive half-leaf of $\lambda_x$.
And $h_{\lambda_x}(J_1) = \ldots = h_{\lambda_x}(J_n)$.
\end{itemize}
\end{defn}

See Figure \ref{flow box figure} (f) for a singular product pair of
a $6$-prong singularity of $\Fa$.

\begin{cons}\rm\label{construction of N}
\begin{enumerate}[(a)]
\item Let \[N = \{(x,y) \mid x \in \PP, y \in r_x\}.\]
\item For any product pair or singular product pair $(C,D)$ of $\PP$,
we denote by $\B_{(C,D)} = \{(x,y) \mid x \in C, y \in D \cap R_x\}$
(endowed with the subspace topology of $\PP \times \PP$).
For any $x \in \PP$,
we denote by $q_x: R_x \to r_x$ the quotient map
induced from the equivalence relation $\stackrel{\lambda_x}{\sim}$.
We denote by $q_{(C,D)}: \B_{(C,D)} \to N$ the map defined by
$q_{(C,D)}(x,y) = (x,q_x(y))$ for all $(x,y) \in \B_{(C,D)}$,
and we denote the image of $\B_{(C,D)}$ under $q_{(C,D)}$ by $B_{(C,D)}$.
We call $B_{(C,D)}$ a \emph{flow box} of $N$ if $(C,D)$ is a product pair of $\PP$ and
call $B_{(C,D)}$ a \emph{singular flow box} of $N$ if $(C,D)$ is 
a singular product pair of $\PP$,
and we endow $B_{(C,D)}$ with the quotient topology of $\B_{(C,D)}$.
We note that all flow boxes and singular flow boxes of $N$ are homeomorphic to 
$[0,1] \times [0,1] \times [0,1]$,
and so their sets of interior points are homeomorphic to $\R^{3}$.
We set
\[\{\text{Sets of interior points of flow boxes and singular flow boxes of } N\}\]
be the topology basis of $N$. 
\item Recall that for any $x \in \PP$ and $y,z \in \lambda_x$,
$y \stackrel{\lambda_x}{\sim} z$ if and only if $g(y) \stackrel{g(\lambda_x)}{\sim} g(z)$,
the transformation $g \mid_{\lambda_x}: \lambda_x \to g(\lambda_x)$ induces
a transformation $g_x: r_x \to r_{g(x)}$.
So the action of $G$ on $\PP$ induces an action of $G$ on $N$ for which
\[g(x,y) = (g(x),g_x(y)), \forall g \in G, x \in \PP, y \in r_x.\]
\item Let $\w{\phi} = \{(x,r_x) \mid x \in \PP\}$.
Then $\w{\phi}$ is a flow of $N$ equivariant under the action of $G$.
\end{enumerate}
\end{cons}

\begin{prop}
    $N \cong \R^{3}$.  
\end{prop}
\begin{proof}
    For any point $x \in N$,
    there exists a flow box or singular flow box of $N$ containing $x$.
    So $N$ is a $3$-manifold and is a $\R$-bundle over $\PP$.
    It follows that $N \cong \R^{3}$.  
\end{proof}

Before showing that the action of $G$ on $N$ is free and discrete,
we first describe our second construction of $N$.

\subsection{The second construction}\label{subsec: the second construction}

\begin{figure}
	\centering
	\subfigure[]{
		\includegraphics[width=0.3\textwidth]{2.png}}
	\subfigure[]{
		\includegraphics[width=0.3\textwidth]{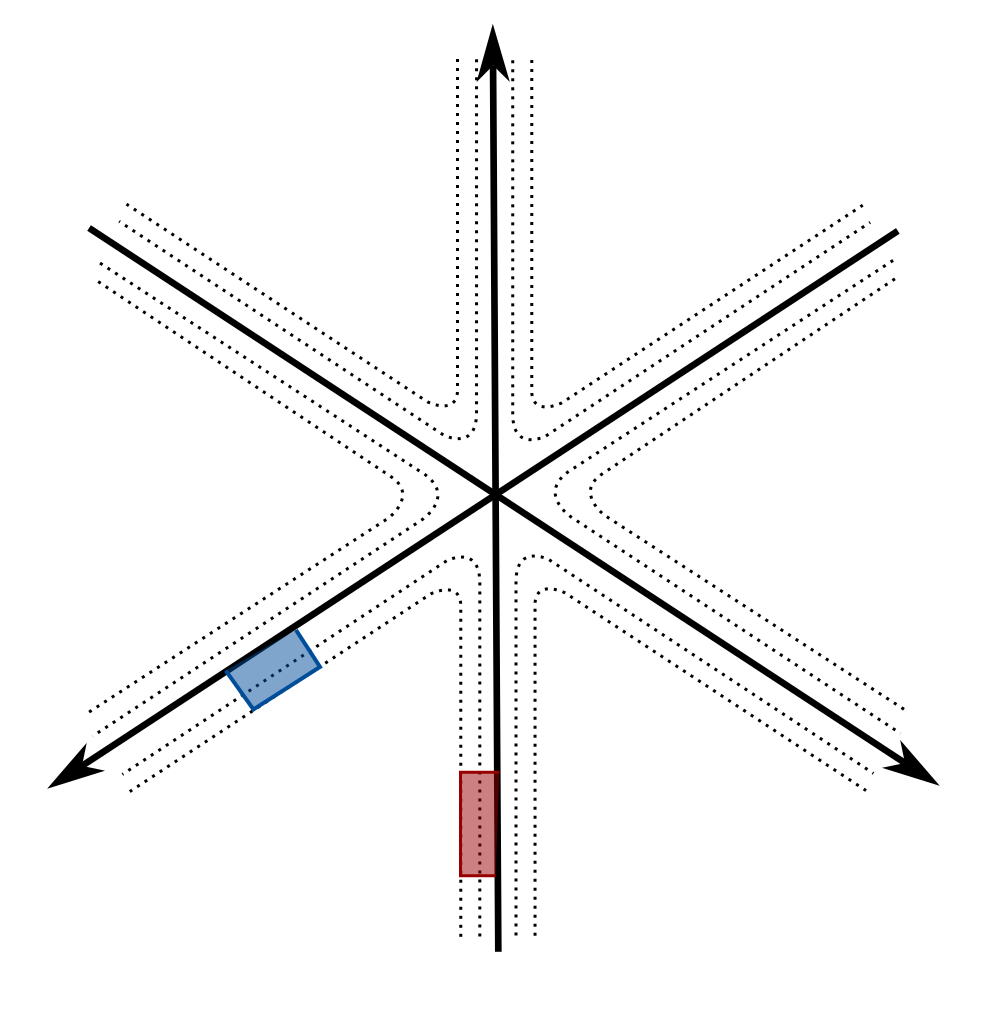}}
  \subfigure[]{
		\includegraphics[width=0.3\textwidth]{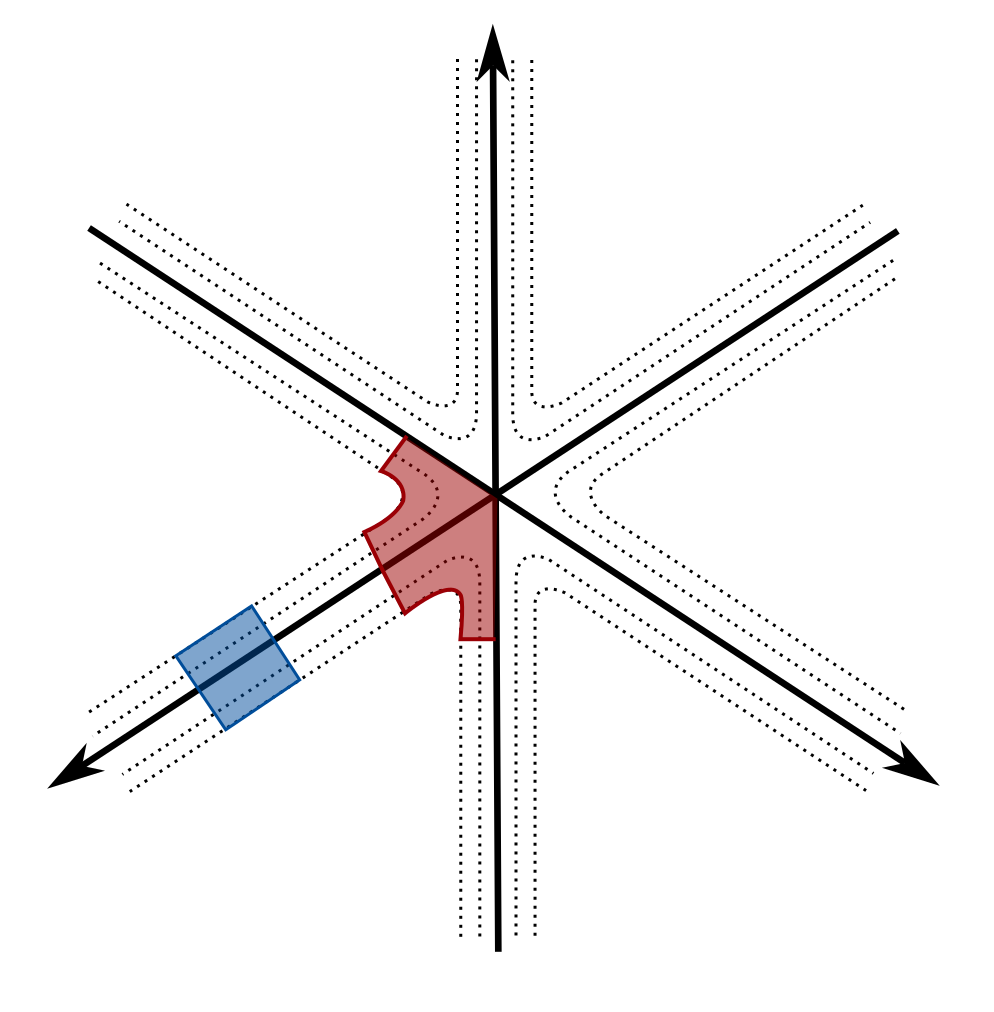}}
  \caption{Let $(x,y) \in N_2$ for which $\lambda_x$ is a singular leaf and 
  $[x,y]_{\lambda_x}$ contains the singularity of $\lambda_x$,
  we describe the local model of $(C_x,D_y)$ for which
  $\B_{(C_x,D_y)}$ is a neighborhood of $(x,y)$ in $N_2$.
  $\lambda_x$ is highlighted as the bolded leaf and other leaves of $\Fa$ are described by the dotted leaves. The red and blue regions describe $C_x, D_y$ respectively.}\label{neighborhood}
  \end{figure}

Let $N_2 = \{(x,y) \mid x \in \PP, y \in R_x\}$,
endowed with the subspace topology of $\PP \times \PP$.
Let $(x,y) \in N_2$.
Next, we describe the neighborhood of $(x,y)$ in $N_2$ as follows.

We choose a neighborhood $N_x$ of $x$ in $\PP$ and a neighborhood $N_y$ of $y$ in $\PP$ 
for which $N_x$ (resp. $N_y$) is a product rectangle if $x$ (resp. $y$) is not a singularity of $\Fa$ and is a product polygon if $x$ (resp. $y$) is a singularity of $\Fb$.
We may assume $N_x, N_y$ are sufficiently small so that $N_x \cap N_y = \emptyset$.
Then $N_x \times N_y$ is a neighborhood of $(x,y)$ in $\PP \times \PP$.
Let \[C_x = \{t \in N_x \mid R_t \cap N_y \ne \emptyset\},\]
\[D_y = \{t \in N_y \mid t \in R_s \text{ for some } s \in N_x\}.\]
It is not hard to see that
\[(N_x \times N_y) \cap N_2 =
\{(x,y) \in \PP \times \PP \mid x \in C_x, y \in R_x \cap D_y\}\]
(that is, it is $\B_{(C_x,D_y)}$ once we know that $(C_x,D_y)$ is a product pair or a singular product pair),
which is a neighborhood of $(x,y)$ in $N_2$.

Now we describe $(C_x,D_y)$ and $(N_x \times N_y) \cap N_2$ case by case.

\begin{case}\rm\label{case 1}
Assume that $\lambda_x$ is a non-singular leaf.
Then $C_x,D_y$ are product rectangles with $x \in Int(C_x), y \in Int(D_y)$.
Compare with Figure \ref{flow box figure} (a) for the picture of
$(C_x,D_y)$.
As a result, $(x,y)$ has a neighborhood
$\B_{(C_x,D_y)}$ in $N_2$ homeomorphic to $[0,1] \times [0,1] \times [0,1]$.
\end{case}

\begin{case}\rm\label{case 2}
Assume that $\lambda_x$ is a singular leaf,
$x$ is contained in a positive half-leaf of $\lambda_x$ and 
is not the singularity of $\lambda_x$.
Then $C_x,D_y$ are product rectangles with $x \in Int(C_x), y \in Int(D_y)$,
compare with Figure \ref{flow box figure} (b).
Hence $(x,y)$ has a neighborhood
$\B_{(C_x,D_y)}$ in $N_2$ homeomorphic to $[0,1] \times [0,1] \times [0,1]$.
\end{case}

\begin{case}\rm\label{case 3}
Assume that $\lambda_x$ is a singular leaf,
$x$ is contained in a negative half-leaf of $\lambda_x$ and 
is not the singularity of $\lambda_x$.
Then there are the following cases:

(a)
Assume that $y$ is also contained in this negative half-leaf and 
is not the singularity of $\lambda_x$.
Then $C_x,D_y$ are product rectangles such that
$x \in Int(C_x), y \in Int(D_y)$,
compare with Figure \ref{flow box figure} (c).
And $(x,y)$ has a neighborhood
$\B_{(C_x,D_y)}$ in $N_2$ homeomorphic to $[0,1] \times [0,1] \times [0,1]$.

(b)
Assume that $y$ is the singularity of $\lambda_x$.
Then $(C_x,D_y)$ is modeled as Figure \ref{neighborhood} (a).
The neighborhood $\B_{(C_x,D_y)}$ of $(x,y)$ in $N_2$ is homeomorphic to
the space obtained from cutting off $[0,1] \times [0,1] \times [0,1]$ along
$\{\frac{1}{2}\} \times [0,1] \times [0,\frac{1}{2}]$,
where $(x,y)$ is identified with the point $(\frac{1}{2},\frac{1}{2},\frac{1}{2})$ in it.

(c)
Assume $y$ is contained in a positive half-leaf of $\lambda_x$ and
$x$ is not the singularity of $\lambda_x$.
Then $C_x,D_y$ are product rectangles such that
$x$ is contained in a $\Fa$-side of $C_x$ and $y$ is contained in a $\Fa$-side of $D_y$,
which is modeled as Figure \ref{neighborhood} (b).
Note that $((x,y),\B_{(C_x,D_y)})$ is homeomorphic to
$((1,\frac{1}{2},\frac{1}{2}),[0,1] \times [0,1] \times [0,1])$.
\end{case}

\begin{case}\rm\label{case 4}
Assume that $\lambda_x$ is a singular leaf and $x$ is the singularity of $\lambda_x$.
Then $(C_x,D_y)$ are modeled as Figure \ref{neighborhood} (c).
In this case,
$((x,y),\B_{(C_x,D_y)})$ is also homeomorphic to
$((1,\frac{1}{2},\frac{1}{2}),[0,1] \times [0,1] \times [0,1])$.
\end{case}

Let $\stackrel{h}{\sim}$ be the equivalence relation on $N_2$ defined as follows.
$(x_1,y_1) \stackrel{h}{\sim} (x_2,y_2)$ if and only if

(1)
$x_1 = x_2$.

(2)
Suppose (1) holds.
Let $\lambda$ denote the leaf of $\Fa$ containing $x_1, x_2$.
Then
$h_\lambda(y_1) = h_\lambda(y_2)$ if $\lambda$ is a singular leaf,
and $y_1 = y_2$ if otherwise.

For each $x \in \PP$,
the image of $(x,R_x)$ under the equivalence relation $\stackrel{h}{\sim}$ is
$(x,r_x)$,
and thus the equivalence relation $\stackrel{h}{\sim}$ defines a map
$e: N_2 \to N$.
It's not hard to see that the flow boxes or singular flow boxes of $N$ pulls-back to
a union of sets in $N_2$ described as above,
and thus the topology of $N$ is consistent with the quotient topology from $N_2$.
Therefore,
$N$ is the quotient space of $N_2$ under the equivalence relation $\stackrel{h}{\sim}$
(endowed with the quotient topology).

\begin{prop}
    The action of $G$ on $N$ is free and discrete.
\end{prop}
\begin{proof}
    Let $(x,y) \in N$. We prove that $(x,y)$ has a neighborhood $U_{(x,y)}$ so that
    $g(x,y) \notin U_{(x,y)}$ for all $g \in G - \{1\}$.

    Let $(x_1,y_1) \in N_2$ for which $e(x_1,y_1) = (x,y)$.
    If $(x_1,y_1)$ is in either of Cases \ref{case 1}, \ref{case 2} and \ref{case 3} (a),
    then $(x_1,y_1)$ is uniquely determined by $(x,y)$,
    and there is a neighborhood $\B_{(C_{x_1},D_{y_1})}$ of $(x_1,y_1)$ in $N_2$ such that
    the restriction of $e$ on $\B_{(C_{x_1},D_{y_1})}$ is a homeomorphism.
    Note that $\B_{(C_{x_1},D_{y_1})}$ is a neighborhood of $(x_1,y_1)$ in $\PP \times \PP$.
    It follows from Property A.7 directly that $g(x_1,y_1) \notin \B_{(C_{x_1},D_{y_1})}$ 
    for all $g \in G - \{1\}$ when $\B_{(C_{x_1},D_{y_1})}$ is sufficiently small.

    Now assume that $(x_1,y_1)$ is in Case \ref{case 3} (b).
    Then $(x_1,y_1)$ is also uniquely determined by $(x,y)$.
    Let $\B_{(C_{x_1},D_{y_1})}$ be the neighborhood of $(x_1,y_1)$ in $N_2$ 
    as given in Case \ref{case 3} (b),
    then $e(\B_{(C_{x_1},D_{y_1})})$ is a neighborhood of $(x,y)$ in $N$.
    It also follows from Property A.7 directly that we can choose 
    a sufficiently small neighborhood $U$ of $(x_1,y_1)$ in $\PP \times \PP$ such that
    $g(x_1,y_1) \notin U$ for all $g \in G - \{1\}$,
    and we can choose $\B_{(C_{x_1},D_{y_1})}$ sufficiently small so that
    it is contained in $U$.

    \begin{figure}
	\centering
	\subfigure[]{
		\includegraphics[width=0.3\textwidth]{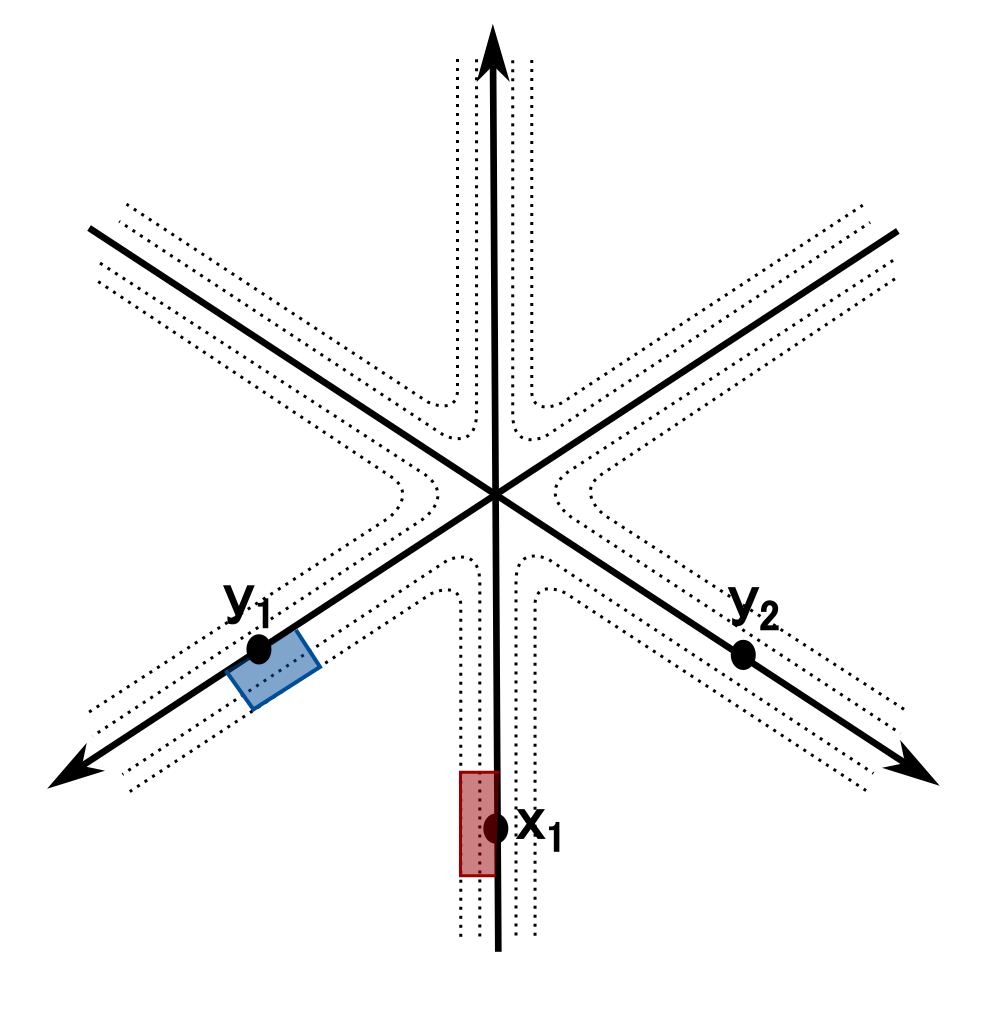}}
	\subfigure[]{
		\includegraphics[width=0.3\textwidth]{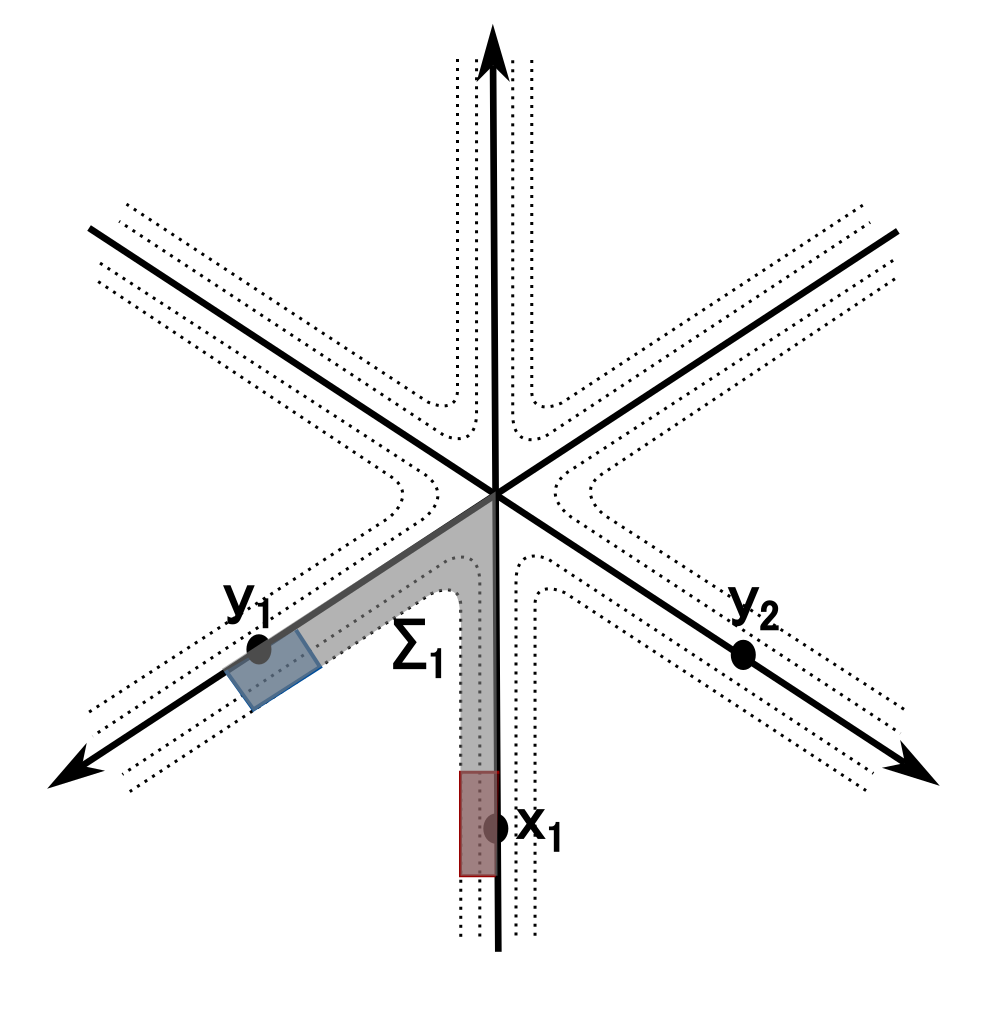}}
  \caption{\label{Sigma_1}Suppose that $(x_1,y_1)$ is in Case \ref{case 3} (c).
  (a) is the picture of $x_1, y_1$ and $y_2 \in R_{x_1} - \{y_1\}$ for which
  $e(x_1,y_2) = e(x_1,y_1)$.
  The shadow region of (b) describes the product rectangle $\Sigma_1$.}
  \end{figure}
  
    Now assume that $(x_1,y_1)$ is in Case \ref{case 3} (c).
    Then there is a unique $(x_2,y_2) \in N_2 - \{(x_1,y_1)\}$ such that
    $e(x_1,y_1) = e(x_2,y_2)$.
    Note that $x_2 = x_1$.
    Let $\B_{(C_{x_1},D_{y_1})}$ be the neighborhood of $(x_1,y_1)$ in $N_2$ 
    as given in Case \ref{case 3} (c),
    and let $\Sigma_1$ be the product rectangle such that
    (1) the $\Fa$-sides of $\Sigma_1$ are same as the $\Fa$-sides of $C_{x_1}$ and $D_{y_1}$,
    (2) the $\Fb$-sides consist of a $\Fb$-side of $C_{x_1}$ and a $\Fb$-side of $D_{y_1}$,
    (3) $C_{x_1}, D_{y_1} \subseteq \Sigma_1$.
    See Figure \ref{Sigma_1} (b) for a picture of $\Sigma_1$.
    Let $s$ denote the singularity of $\lambda_{x_1}$.
    For any $g \in G$ with $g(x_i) \in C_{x_1}, g(y_i) \in D_{y_1}$ ($i \in \{1,2\}$),
    we have $g(s) \in \Sigma_1$.
    As $s$ is a singularity of $\Fa$,
    the orbit of $s$ in $\PP$ is discrete,
    so we can choose $C_{x_1}, D_{y_1}$ sufficiently small so that
    $\Sigma_1$ contains no $\{g(s) \mid g \in G, g(s) \ne s\}$,
    and hence $g(x_i,y_i) \in \B_{(C_{x_1},D_{y_1})}$ ($i \in \{1,2\}$) implies $g = 1$.
    Similarly,
    we can choose a sufficiently small neighborhood of $(x_2,y_2)$ in $N_2$ that
    contains no image of $(x_i,y_i)$ ($i \in \{1,2\}$) under non-identity elements of $G$.
    
    Now assume that $(x_1,y_1)$ is in Case \ref{case 4}.
    Then $x_1$ is the singularity of $\lambda_{x_1}$.
    Let $(x_1,y_2), \ldots, (x_1,y_n)$ denote the other preimages of $(x,y)$ under $e$
    (where $y_2,\ldots,y_n \in \lambda_{x_1}$).
    Let $H$ denote the stabilizer of $\lambda_{x_1}$ under $G$,
    and let $f$ be a generator of $H$.
    We can choose a neighborhood $U$ of $x_1$ in $\PP$ that 
    contains no image of $x_1$ under $G - H$.
    By Assumption \ref{further assumption}, 
    we can choose a neighborhood $V_1$ of $y_1$ in $\PP$ that contains 
    no image of $y_2, \ldots,y_n$ under $H$.
    Similarly,
    for each $i \in \{2,\ldots,n\}$,
    we can choose a neighborhood $V_i$ of $y_i$ in $\PP$ that contains 
    no image of $\{y_1,\ldots,y_n\} - \{y_i\}$ under $H$.
    We can choose a neighborhood $W$ of $(x,y)$ in $N$ contained in
    the image of $\{(u,v) \mid u \in U, v \in (\bigcup_{i=1}^{n}V_i) \cap R_u\}$ under $e$.
    Then no element of $G - \{1\}$ takes $(x,y)$ into $W$.
    \end{proof}

Thus,
the action of $G$ on $N$ is free and discrete.
So $N / G$ is a $3$-manifold,
and $\w{\phi}$ induces a flow $\phi = \w{\phi} / G$ of $M$.

\begin{defn}\rm\label{invariant foliation}
(a)
    For each leaf $\lambda$ of $\Fa$,
    we define $l^{s}_{\lambda} = \bigcup_{x \in \lambda} (x,r_x)$.
    Then \[\{l^{s}_{\lambda} \mid \lambda \text{ is a leaf of } \Fa\}\] is a singular foliation of $N$
    equivariant under the action of $G$.
    We denote this foliation by $\w{\Fs}$,
    and let $\Fs = \w{\Fs} / G$.

    (b)
    For each leaf $\mu$ of $\Fb$,
    we define $l^{u}_{\mu} = \bigcup_{x \in \mu} (x,r_x)$.
    Then \[\{l^{u}_{\mu} \mid \mu \text{ is a leaf of } \Fb\}\] is also a singular foliation of $N$
    equivariant under the action of $G$.
    We denote this foliation by $\w{\Fu}$,
    and let $\Fu = \w{\Fu} / G$.
    \end{defn}

\begin{rmk}
Before constructing the manifold $N$ with the flow $\w{\phi}$,
we have fixed the monotone maps 
$\{h_\lambda \mid \lambda \text{ is a singular leaf of } \Fa\}$,
so $N, \w{\phi}$ are determined by these monotone maps.

Now suppose that the action of $G$ on $\PP$ is conjugate to 
the $\pi_1$-action of a transitive pseudo-Anosov flow $\phi_0$ in 
a closed $3$-manifold $M_0$.
By Proposition \ref{equivalent},
$M$ is homeomorphic to $M_0$ and $\phi$ is orbitally equivalent to $\phi_0$.
It follows that $(M,\phi)$ is uniquely determined by 
$\{h_\lambda \mid \lambda \text{ is a singular leaf of } \Fa\}$ up to 
self-homeomorphisms on $M$.
However,
we don't know if $(M,\phi)$ is uniquely determined by these maps 
(up to self-homeomorphisms on $M$) in other cases.
\end{rmk}

\subsection{Taut foliations transverse to $\phi$}\label{subsec: foliation}

Now we establish Theorem \ref{transverse foliation} in our setting. Let $(C,D)$ be a product pair or singular product pair of $\PP$.
Let $\Fb_D$ denote the set of leaves of $\Fb$ that has nonempty intersection with $D$.
For each leaf $\mu$ of $\Fb$ with $\mu \cap D \ne \emptyset$,
let 
\[l^{\mu}_{(C,D)} = \{(x,y) \mid y \in D \cap \mu, x \in C \cap \lambda_y\}.\]
Then $\{l^{\mu}_{(C,D)} \mid \mu \in \Fb_D\}$ is a partition of $\B_{(C,D)} \subseteq \PP \times \PP$
into $2$-dimensional surfaces.
The quotient map $q_{(C,D)}: \B_{(C,D)} \to B_{(C,D)}$ 
(defined in Construction \ref{construction of N} (b)) descends 
the partition $\{l^{\mu}_{(C,D)} \mid \mu \in \Fb_D\}$ of $\B_{(C,D)}$ to
a two-dimensional foliation (denoted $\F_{(C,D)}$) of $B_{(C,D)}$.

For all (singular) flow boxes $B_{(C,D)}$ in $N$,
their partitions $\F_{(C,D)}$ agree and 
form a $2$-dimensional foliation of $N$, denoted $\w{\F}$.
It is not hard to see that all $g: N \to N$ identifies 
the foliation $\F_{(C,D)}$ of $B_{(C,D)}$ with the foliation
$\F_{(g(C),g(D))}$ of $B_{(g(C),g(D))}$.
So $\w{\F}$ is equivariant under the action of $G$ on $N$.

Set $\F = \w{\F} / G$. Then $\F$ is a foliation of $M$.
It is clear that $\w{\phi}$ is transverse to $\w{\F}$
(and thus $\phi$ is transverse to $\F$).
Hence $\F$ is a co-orientable foliation.


\begin{prop}\label{taut}
    If $M$ is closed,
    then $\F$ is a taut foliation.
\end{prop}

Prior to prove this proposition,
we first prepare some properties of $\w{\F}$ in the following several paragraphs.
Recall that $e: N_2 \to N$ is the quotient map defined in the last subsection.

\begin{LEM}\label{l_y}
    Let $y \in \PP$.
    Let $E_y = \{x \in \lambda_y \mid y \in R_x\}$.
    For any component $E$ of $E_y$,
    there is a leaf $l$ of $\w{\F}$ such that
\[e(\{(x,y) \mid x \in E\}) \subseteq l\].
\end{LEM}
\begin{proof}
    We first prove that,
    any $x \in E$ has a neighborhood $U_x$ in $E$ such that,
    if $t_1, t_2 \in U_x$,
    then $e(t_1,y), e(t_2,y)$ are contained in the same leaf of $\w{\F}$.    

    Let $x \in E$.
    There is a product pair or singular product pair $(C,D)$ such that
    $x \in Int(C)$ and $y \in D$.
    Note that for all $t_1, t_2 \in \lambda_y \cap C$,
    $(t_1,y), (t_2,y)$ are both contained in the set $l^{\mu_y}_{(C,D)}$ 
    in the partition $\{l^{\mu}_{(C,D)} \mid \mu \in \Fb_D\}$ of $\B_{(C,D)}$,
    which implies that $e(t_1,y), e(t_2,y)$ are contained in the same leaf of $\w{\F}$.
    Set $U_x = \lambda_x \cap C$.

    As $E$ is a connected set,
    for any $t_1, t_2 \in E$,
    there are finitely many of $\{U_x \mid x \in E\}$ that
    cover the path in $E$ from $t_1$ to $t_2$.
    The above discussion implies that
    $e(t_1,y), e(t_2,y)$ are contained in the same leaf of $\w{\F}$.
    This completes the proof.
\end{proof}

Under the assumption of Lemma \ref{l_y},
if $y$ is not a singularity of $\Fb$,
then $E_y$ is connected,
and thus there is a leaf of $\w{\F}$ containing
$\{(x,y) \mid y \in R_x\}$.
We will always denote this leaf by $l_y$.
Now we assume that $y$ is the singularity of 
a $2n$-prong ($n \geqslant 2$) singular leaf of $\Fb$.
Let $q_1, \ldots, q_n$ denote the $n$ negative half-leaves of $\mu_y$.
Then $E_y$ has $n$ components which are exactly
$q_1 - \{y\}, \ldots, q_n - \{y\}$.
In fact, 
there are $n$ distinct leaves $l$ satisfying that $e(\{(x,y) \mid y \in R_x\}) \subseteq l$.
We leave this case here,
and we will discuss it in Proposition \ref{asymptotic}.
Note that the term $l_y$ is only well-defined when $y$ is not a singularity of $\Fb$.

For a path $\alpha$ in $\PP$,
we will always abuse the notation and say that $\alpha$ is \emph{transverse} to $\Fa$ if 
$\alpha$ intersects each leaf of $\Fa$ at a single point.
This allows $\alpha$ to pass the singularities of $\Fa$.
Recall that $\Fa, \Fb$ have well-defined co-orientations induced from their orientations.
In addition,
the orientation on $\Fb$ induces continuously varying orientations on 
the flowlines of $\w{\phi}$,
which induce a co-orientation on $\w{\F}$.

\begin{LEM}\label{transversal}
    Let $\gamma: I \to \PP$ be a positively oriented transversal of $\Fb$ that
    doesn't intersect the singularities of $\Fb$.
    Then each $t \in I$ has a neighborhood $[a,b]$ in $I$ such that
    there is transversal from $l_{\gamma(a)}$ to $l_{\gamma(b)}$.
\end{LEM}
\begin{proof}
    For each $t \in I$,
    there is a product pair $(C,D)$ such that $\gamma(t) \in Int(D)$
    (as $\gamma(t)$ is not a singularity of $\Fb$).
    Then there is a closed interval $[a,b] \subseteq I$ such that $t \in (a,b)$ and
    $\gamma([a,b]) \subseteq D$.
    As $\gamma([a,b]) \subseteq D$ and $\gamma([a,b])$ is transverse to $\Fb$,
    there is a transversal $\tau$ of $\F_{(C,D)}$ in $B_{(C,D)}$ such that
    $\tau$ is projected to $\gamma([a,b])$ under the projection
    $B_{(C,D)} \cong \B_{(C,D)} \to D$.
\end{proof}

\begin{proof}[The proof of Proposition \ref{taut}]
By Property (A4) of Definition \ref{Anosov-like},
for any leaf $\mu$ of $\Fb$,
the stabilizer of $\mu$ is isomorphic to either $\Z$ or the identity.
Thus $\Fu$ contains no compact leaf.
We can split $\Fu$ along the singular leaves and filling with monkey saddles 
to obtain a co-orientable foliation $\Fu_*$ \cite{Gab92b}.
Note that $\Fu_*$ also contains no compact leaf and thus is a taut foliation.
We assign $\Fu$ a co-orientation induced from the co-orientation on $\Fb$,
and we assign $\Fu_*$ a co-orientation induced from the co-orientation on $\Fu$.

Because $M$ is a closed $3$-manifold,
there are finitely many product charts $\{U_\alpha\}_{\alpha \in \Psi}$ of $\F$ that
cover $M$ (where $\Psi$ is an index set).
Now we show that for each $\alpha \in \Psi$,
there is a simple closed curve in $M$ transverse to $\F$ that
intersects all leaves of $\F \mid_{U_\alpha}$.

Let $\w{U_\alpha}$ be a lift of $U_\alpha$ in $N$,
and let $(x,y) \in \w{U_\alpha}$.
We may assume that $y$ is not a singularity of $\Fa$.
Because $\Fu_*$ is a co-orientable taut foliation,
there is a simple closed curve $\eta^{'}$ in $M$ positively transverse to $\Fu_*$ that 
intersects all leaves of $\Fu_*$,
which induces a simple closed curve $\eta$ in $M$ positively transverse to $\Fu$
in the following sense: 
$\eta$ can be decomposed to
finitely many oriented subsegments $\eta_1,\ldots,\eta_k$ so that
each $\eta_i$ is positively transverse to $\Fu$ and $Int(\eta_i)$ doesn't pass through
the singular sets of all singular leaves of $\Fu$
(however, $\eta$ may pass through the singular sets of singular leaves of $\Fu$).

The curve $\eta$ induces a path $\gamma: I \to \PP$ positively transverse to $\Fb$ that
starts at $y$ and ends at $g(y)$ for some $g \in G - \{1\}$.
We first modify $\gamma$ to a broken path $\gamma_*$ in $\PP$,
which is the union of finitely many paths $\gamma_1,\ldots,\gamma_m: I \to \PP$ such that
(1)
each $\gamma_i$ is positive transverse to $\Fb$,
and $Int(\gamma_i)$ contains no singularity of $\Fb$,
(2)
$\gamma_1(0) = y, \gamma_m(1) = g(y)$,
(3)
for each $i \in \{1,\ldots,m-1\}$),
$\gamma_i(1), \gamma_{i+1}(0)$ are contained in 
two positive half-leaves of same singular leaf $\lambda_i$ of $\Fa$ such that
$\gamma_i(1) \stackrel{\lambda_i}{\sim} \gamma_{i+1}(0)$.

\begin{figure}
	\centering
	\subfigure[]{
		\includegraphics[width=0.3\textwidth]{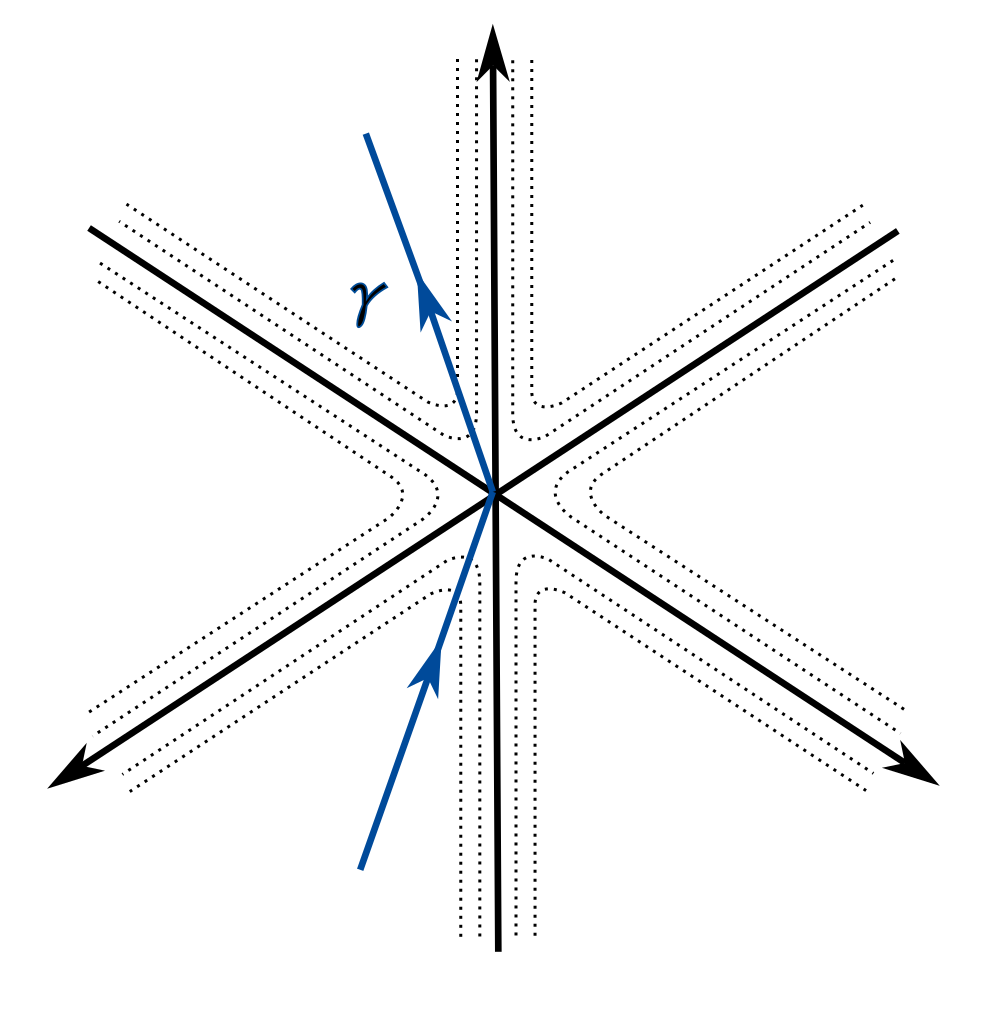}}
	\subfigure[]{
		\includegraphics[width=0.3\textwidth]{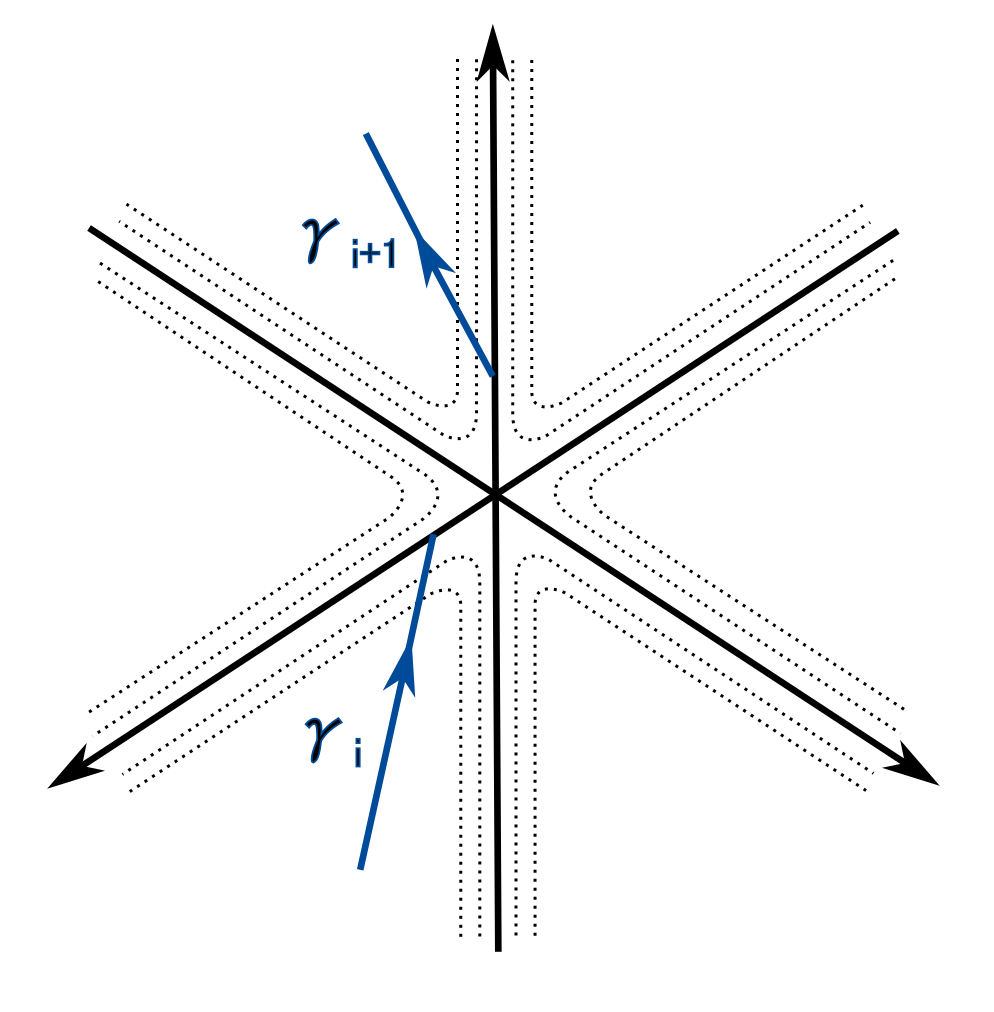}}
  \caption{The leaves in the figures are leaves of $\Fb$.
  In Figure (a),
  $\gamma$ is a transversal of $\Fb$.
  Figure (b) describes the broken path $\gamma_* = \bigcup_{i=1}^{m} \gamma_i$ 
  modified from $\gamma$, such that,
  for any $i \in \{1,\ldots,m-1\}$,
  $\gamma_i(1) \stackrel{\lambda_i}{\sim} \gamma_{i+1}(0)$ for some singular leaf 
  $\lambda_i$ containing both of $\gamma_i(0), \gamma_i(1)$,
  and $\gamma_i(0), \gamma_i(1)$ are contained in 
  two positive half-leaves of $\lambda_i$.}\label{gamma modification}
  \end{figure}

By Lemma \ref{transversal},
for each $i \in \{1,\ldots,m\}$,
there are $0 = t^{i}_{0} < t^{i}_{1} < \ldots < t^{i}_{s_i} = 1$ so that
there is a transversal from
$l_{\gamma_i(t^{i}_{j})}$ to $l_{\gamma_i(t^{i}_{j+1})}$ for each $j \in \{0,\ldots,s_i-1\}$.

Let $i \in \{1,\ldots,m-1\}$,
and we denote by $s_i$ the singularity of $\lambda_i$.
Note that $e(s_i,\gamma_i(1))$ is contained in $l_{\gamma_i(1)}$,
$e(s_i,\gamma_{i+1}(0))$ is contained in $l_{\gamma_{i+1}(0)}$,
and $e(s_i,\gamma_i(1)), e(s_i,\gamma_{i+1}(0))$ are the same point in $N$.
It follows that $l_{\gamma_i(1)} = l_{\gamma_{i+1}(0)}$.
Combined with the above discussion,
there is a sequence of points $y_0, y_1, \ldots, y_n \in \PP$ with
$y_0 = y, y_n = g(y)$ such that,
there is a transversal $\tau_i: I \to N$ of $\w{\F}$ with 
$\tau_i(0) \in l_{y_i}, \tau_i(1) \in l_{y_{i+1}}$, 
for each $i \in \{0,\ldots,n-1\}$.
We can connect every $\tau_i(1), \tau_{i+1}(0)$ through $l_{y_{i+1}}$ to
obtain a path $\gamma_1$ from $l_{y_0} = l_y$ to $l_{y_n} = l_{g(y)}$ that
is either positively transverse to $\w{\F}$ or tangent to $\w{\F}$.
Then we can isotope $\gamma_1$ to obtain a path $\gamma_2$ such that
$\gamma_2$ starts at $(x,y) \in l_y$ and ends at $(g(x),g(y)) \in l_{g(y)}$ and
$\gamma_2$ is positively transverse to $\w{\F}$.
Then $\gamma_2$ descends to a path in $M$ transverse to $\F$ that
starts and ends at the same point in $U_\alpha$,
which necessarily intersects all leaves of $\F \mid_{U_\alpha}$.
This completes the proof.
\end{proof}

In the remainder of this subsection,
we will discuss the relation between $\phi$ and $\F$.
We first recall the concept ``regulating'' for flows transverse to $\R$-covered foliations.

Let $X$ be a closed $3$-manifold,
let $\mathcal{E}$ be a co-orientable taut foliation of $X$,
and let $\varphi$ be an oriented flow in $X$ transverse to $\mathcal{E}$.
We denote by $\w{\mathcal{E}},\w{\varphi}$ 
the pull-backs of $\mathcal{E},\varphi$ in the universal cover of $X$.
In the case where $\mathcal{E}$ is an $\R$-covered foliation
(i.e. the leaf space of $\w{\mathcal{E}}$ is homeomorphic to $\R$),
$\varphi$ is said to be \emph{regulating} for an $\mathcal{E}$ if 
every flowline of $\w{\varphi}$ intersects every leaf of $\w{\mathcal{E}}$
\cite[Definition 9.4]{Cal07}.
It's known that for any $\R$-covered foliation in an atoroidal $3$-manifold,
there exists a transverse pseudo-Anosov flow regulating for it \cite{Cal00}, \cite{Fen02}.

We say that a flowline $\rho$ of $\w{\varphi}$ is \emph{asymptotic} to
a leaf $l$ of $\w{\mathcal{E}}$ if 
there is a sequence of flowlines $\{\rho_n\}_{n \in \N}$ converging to $\rho$ in 
the orbit space of $\w{\varphi}$ such that
$\rho_n \cap l \ne \emptyset$ for all $n \in \N$,
but $\rho \cap l = \emptyset$.
When $\mathcal{E}$ is $\R$-covered,
it's not hard to observe that $\varphi$ is regulating for $\mathcal{E}$ if and only if
no flowline of $\w{\varphi}$ is asymptotic to some leaf of $\w{\mathcal{E}}$.
From this point of view,
we could define the ``nowhere-regulating'' property for flows transverse to foliations.

\begin{defn}\rm
    $\varphi$ is said to be \emph{nowhere-regulating} for $\mathcal{E}$ if
    every flowline of $\w{\varphi}$ is asymptotic to some leaf of $\w{\mathcal{E}}$.
\end{defn}

Now we show that

\begin{prop}
    $\phi$ is nowhere-regulating for $\F$.
\end{prop}

\begin{proof}
For any $t \in \PP$,
we denote by $\rho_t$ the flowline $(t,r_t)$ of $\w{\phi}$.

Let $y \in \PP$ and
let $E_y = \{x \in \lambda_y \mid y \in R_x\}$.
We choose a component $E$ of $E_y$ and
choose a sequence of points $y_1, y_2, \ldots$ on $E$ such that 
$\lim_{n \to \infty} y_n = y$.
If $y$ is not a singularity of $\Fb$,
then $\rho_{y_n}$ intersects $l_y$ at $e(y_n,y)$ for each $n \in \N$,
but $\rho_y \cap l_y = \emptyset$.
It follows that $\rho_y$ is asymptotic to $l_y$.

Now assume that $y$ is a singularity of $\Fb$.
By Lemma \ref{l_y},
there is a leaf $l$ of $\w{\F}$ such that $e(x,y) \in l$ for all $x \in E$.
Then for any $n \in \N$, $\rho_{y_n}$ intersects $l$ at $e(y_n,y)$,
but $\rho_y \cap l = \emptyset$.
Hence $\rho_y$ is asymptotic to $l$.
\end{proof}

We generalize the asymptotic behavior to a relation between 
leaves of $\w{\Fs}$ and $\w{\F}$ as follows.
Let $l^{u}$ be a leaf of $\w{\Fu}$ and let $l$ be a leaf of $\w{\F}$.
If $l^{u}$ is a non-singular leaf,
we say that $l$ is asymptotic to $l^{u}$ if all flowlines of $\w{\phi}$ contained in $l^{u}$
are asymptotic to $l$.
Now assume that $l^{u}$ is a singular leaf and $l \cap l^{u} = \emptyset$.
Let $J$ be the component of $N - l^{u}$ containing $l$.
We say that $l$ is asymptotic to $l^{u}$ if,
for each flowline $\rho$ of $\w{\phi}$ contained in $l^{u} \cap \overline{J}$,
$\rho$ is asymptotic to $l$.
Now we prove that

\begin{prop}\label{asymptotic}
    (a)
    $\Fs, \Fu$ are transverse to $\F$.

    (b)
    For any leaf $l^{u}$ of $\w{\Fu}$,
    there is a leaf of $\w{\F}$ asymptotic to $l^{u}$.
    In particular,
    if $l^{u}$ is a $2n$-index singular leaf of $\w{\Fu}$ ($n \geqslant 2$),
    i.e. the projection of $l^{u}$ to the orbit space of $\w{\phi}$ is a $2n$-prong,
    then there are $n$ distinct leaves of $\w{\F}$ asymptotic to $l^{u}$ simultaneously.
\end{prop}
\begin{proof}
    (a) can be known from the local structure of $\w{\F}, \w{\Fs}, \w{\Fu}$ 
    in all flow boxes and singular flow boxes.
    We prove (b) as follows.
    For any leaf $\mu$ of $\Fb$,
    we will always denote by $l^{u}_{\mu}$ the leaf of $\w{\Fu}$ projected to $\mu$.
    For any $t \in \PP$,
    we still denote by $\rho_t$ the flowline $(t,r_t)$ of $\w{\phi}$ and
    denote by $E_t = \{x \in \lambda_t \mid t \in R_x\}$.
    Let $\mu$ be a leaf of $\Fb$,
    and we prove (b) for $l^{u}_{\mu}$ as follows.

    We first assume that $\mu$ is a non-singular leaf.
    Let $y \in \mu$,
    and we choose an infinite sequence of points $y_1, y_2, \ldots$ on $E_y$ such that 
    $\lim_{n \to \infty} y_n = y$.
    Let $\mu_1,\mu_2,\ldots$ be the collection of leaves of $\Fb$ for which $y_i \in \mu_i$.

    Now choose any $z \in \mu$.
    There is a sufficiently large $n \in \N$ such that
    $\mu_k \cap E_z \ne \emptyset$ for all $k \geqslant n$.
    We denote $\mu_k \cap E_z$ by $z_k$ for all $k \geqslant n$.
    Now fix $k \in \Z_{\geqslant n}$.
    We can connect $e(z_k,z)$ with $e(y_k,y)$ by
    finitely many flow boxes that cover $e(\mu_k \cap E_t,t)$ for all $t \in [y,z]_{\mu}$,
    then we can ensure that $e(\mu_k \cap E_t,t)$ is contained in $l_y$ 
    for any $t \in [y,z]_{\mu}$.    
    Hence $e(z_k,z) \in l_y$.
    This implies that $l_z = l_y$ and $l_y \cap \rho_{z_k} \ne \emptyset$.
    However,
    $\rho_z \cap l_y = \emptyset$ as $l_z = l_y$.
    It follows that $\rho_z$ is asymptotic to $l_z = l_y$.
    Therefore,
    $l_y$ is asymptotic to $l^{u}_{\mu}$.
    
    Now we consider the case that $\mu$ is a singular leaf.
    Let $y$ be the singularity of $\mu$,
    and let $E$ be a component of $E_y$.
    Let $l_E$ be the leaf of $\w{\F}$ for which $e(x,y) \in l_E$ for all $x \in E$.
    We choose an infinite sequence of points $y_1, y_2, \ldots$ on $E$ with
    $\lim_{n \to \infty} y_n = y$,
    and let $\mu_n$ denote the leaf of $\Fb$ for which $y_n \in \mu_n$ for each $n \in \N$.    

    Let $J$ be the component of $\PP - \mu$ containing $E$,
    and we choose $z \in (\mu \cap \overline{J}) - \{y\}$.
    Then there is a sufficiently large $n \in \N$ such that
    $\mu_k \cap E_z \ne \emptyset$ for all $k \geqslant n$.    
    We denote $\mu_k \cap E_z$ by $z_k$ for all $k \geqslant n$.
    Similar to the discussion above, we can confirm that $\rho_{z_k}$ intersects $l_E$ at $e(z_k,z)$ and 
    that $l_E = l_z$ does not intersect $\rho_z$.
    It follows that $l_E$ is asymptotic to $l^{u}_{\mu}$.

    Note that $N - l^{u}_{\mu}$ has $2n$ distinct components,
    and for any two distinct components $E_1, E_2$ of $E$,
    $l_{E_1}, l_{E_2}$ are contained in two distinct components of $N - l^{u}_{\mu}$.
    Hence there are $n$ distinct leaves of $\w{\F}$ asymptotic to $l^{u}_{\mu}$
    simultaneously.
\end{proof}

From the singular foliation $\Fu$ (or $\Fu$),
we can also construct a co-orientable taut foliation as in \cite{Gab92b}.
Here we describe the difference between our foliation $\F$ and
the foliation constructed from \cite{Gab92b}.

\begin{rmk}\label{comparison}
    For convenience,
    we pull-back the two constructions to the universal cover $N$ of $M$ and
    compare them near the singular leaves of $\w{\Fu}$ as follows.
    Let $l^{u}$ be a $2n$-index singular leaf of $\w{\Fu}$.
    In Gabai's construction,
    $l^{u}$ is split to $2n$ boundary leaves of a solid torus gut,
    where $n$ of them are positively oriented and 
    the other $n$ of them are negatively oriented 
    (with respect to the co-orientation on $\w{\Fu}$),
    then this solid torus gut is filled with monkey saddles.
    Note that these $n$ positively oriented leaves are accumulated by monkey saddles simultaneously from their negative sides.

    In our construction,
    there are $n$ distinct leaves of $\w{\F}$ asymptotic to 
    the $2n$-index singular leaf $l^{u}$ simultaneously,
    denoted $F_1,\ldots,F_n$,
    which are accumulated by some leaves of $\w{\F}$ simultaneously from their positive sides.
    However, those leaves converging to $F_1,\ldots,F_n$ do not have a structure similar to the monkey saddles in Gabai's construction; we cannot always find $n$ non-separated leaves accumulated by them simultaneously from the negative sides.
\end{rmk}

We note that not all (pseudo-)Anosov flows are transverse to taut foliations.
So we cannot generalize the construction of $\F$ to (pseudo-)Anosov flows without
orientable stable foliations.

\begin{exmp}\label{pretzel example}
Let $K$ be the $(-2,3,7)$-pretzel knot in $S^{3}$.
Then $K$ is a hyperbolic fibered knot \cite{O84} with 
degeneracy slope $18$ \cite{FS80},
where the degeneracy locus on $K$ has multiplicity $1$.
In addition,
$K$ is an L-space knot \cite{LM16},
which implies that $S^{3}_{K}(s)$ (the surgery on $K$ with slope $s \in \Q$) 
admits no co-orientable taut foliation for all $s \geqslant 9$
(\cite{OS04}, \cite{KMOS07}).
Let $M = S^{3}_{K}(20)$.
As the distance between the slopes $18,20$ is $2$,
$M$ admits an Anosov flow $\varphi$ \cite{Fri83}.
However,
$M$ admits no co-orientable taut foliation.
Thus, 
$\phi$ is necessarily a non-orientable Anosov flow.
As any taut foliation in $M$ transverse to $\varphi$ 
has a well-defined co-orientation induced from 
the orientations on the orbits of $\varphi$,
there is no taut foliation in $M$ transverse to $\varphi$.
\end{exmp}

We finish this section with remark that by Agol-Tsang \cite{AT22}, we can give a sufficient and necessary condition for smooth pseudo-Anosov flow by using Markov partitions. 

\begin{rmk}\label{sufficient and necessary condition}
    In Section \ref{sec:char_anosov}, we give a necessary and sufficient condition for
    $M$ being closed and $\phi$ being topological Anosov.
    The Compactness property (Definition \ref{cptprop}), Convergence property (Definition \ref{convprop}) and Divergence Property (Definition \ref{divprop}) 
    can be generalized to the context in this section,
    allowing us to derive a necessary and sufficient condition for $M$ to be closed and $\phi$ to be topological pseudo-Anosov.    
    Note that a flow box or singular flow box defined in Construction \ref{construction of N} is a generalization of the flow box of Section \ref{sec:char_anosov}.

    By Agol-Tsang \cite[Definition 5.9, Theorem 5.11]{AT22},
    if $M$ is closed, $\phi$ is transitive, and $\phi$ satisfies \cite[Definition 5.9]{AT22} (i.e. $\phi$ satisfies Definition \ref{topological pseudo-Anosov flow} except (d), and $\phi$ has a Markov partition),
    then $\phi$ is orbitally equivalent to a transitive smooth pseudo-Anosov flow.
    Since a Markov decomposition of $\phi$ lifts to a $\pi_1$-equivariant decomposition of $\w{\phi}$, we can describe such a decomposition from $\PP \times \PP$ and thus
    give a sufficient and necessary condition for the existence of a Markov partition of $\phi$.
    Therefore, we can obtain a necessary and sufficient condition so that
    $M$ is closed and $\phi$ is orbitally equivalent to a transitive smooth pseudo-Anosov flow.
    \end{rmk}
   
\section{Characterization as laminar groups} \label{sec:circle-lamination}

   A group that acts faithfully on the circle by orientation-preserving homeomorphisms is called a laminar group if it preserves a circle lamination. We will abuse the terminology and call a group a laminar group even if it preserves an almost lamination and not a genuine lamination. The goal of this section is to show that Anosov-like actions on bifoliated planes can be recovered from certain laminar groups and we will do this on the top of Section \ref{sec:BBM}. 

   Recall that a lamination is said to be quite full if every gap is either a finite-sided polygon or a crown. Here, a crown is a gap with countably infinite number of vertices with a single accumulation point which is accumulated by the vertices from both sides. The unique accumulation point is called the pivot of the crown. 

    We generalize the notion of quite full laminations here by allowing two more types of gaps, blown-up crowns and half-planes. A gap is a blown-up crown if one replaces the pivot of a crown by an arc and adds a new leaf connecting the endpoints of this newly added arc. This new leaf is called the pivotal side of the blown-up crown. In case of a half-plane gap, the only boundary leaf of the gap is called the pivotal side of the half -plane. A circle lamination is called a generalized quite full lamination if each of its gaps is either a finite-sided polygon, a crown, a blown-up crown, or a half-plane.

    \begin{figure}[htb]
\begin{tikzpicture}[scale=1]
    \draw[thick] (0,0) circle (2 cm);
    
    \definecolor{purple}{rgb}{0.5, 0, 0.5}
    
    \draw[thick, purple] (1.732050808, 1) arc (300:160:0.7279404685);
    \draw[thick, purple] (0.6840402867, 1.879385242) arc (340:200:0.7279404685);
    \draw[thick, purple] (-0.6840402867, 1.879385242) arc (20:-120:0.7279404685);
    \draw[thick, purple] (-1.732050808, 1) arc (60:-80:0.7279404685);
    
    \draw[thick, purple] (-1.969615506, -0.3472963553) arc (100:-50:0.5358983849);
    \draw[thick, purple] (-1.532088886, -1.285575219) arc (130:-25:0.4433893253);
    \draw[thick, purple] (-0.8452365235, -1.812615574) arc (155:-10:0.2633049952);
    \draw[thick, purple] (-0.3472963553, -1.969615506) arc (170:-5:0.08732188582);
    \draw[thick, purple] (-0.1743114855, -1.992389396) arc (175:-2:0.05237184314);
    
    \draw[thick, purple] (0.06979899341, -1.998781654) arc (182:5:0.05237184314);
    \draw[thick, purple] (0.1743114855, -1.992389396) arc (185:10:0.08732188582);
    \draw[thick, purple] (0.3472963553, -1.969615506) arc (190:25:0.2633049952);
    \draw[thick, purple] (0.8452365235, -1.812615574) arc (205:50:0.4433893253);
    \draw[thick, purple] (1.532088886, -1.285575219) arc (230:80:0.5358983849);
    \draw[thick, purple] (1.732050808, 1) arc (120:260:0.7279404685);
    
    \draw (0,-2) node{$\bullet$} node[below]{$p$};

    \draw[thick] (5,0) circle (2 cm);
    
    \definecolor{purple}{rgb}{0.5, 0, 0.5}
    \definecolor{red}{rgb}{1, 0, 0}
    
    \draw[thick, purple] (6.732050808, 1) arc (300:160:0.7279404685);
    \draw[thick, purple] (5.6840402867, 1.879385242) arc (340:200:0.7279404685);
    \draw[thick, purple] (5-0.6840402867, 1.879385242) arc (20:-120:0.7279404685);
    \draw[thick, purple] (5-1.732050808, 1) arc (60:-80:0.7279404685);
    \draw[thick, purple] (5-1.969615506, -0.3472963553) arc (100:-50:0.5358983849);
    \draw[thick, purple] (5-1.532088886, -1.285575219) arc (130:-25:0.3433893253);
    \draw[thick, purple] (5-0.9952365235, -1.712615574) arc (155:-20:0.1633049952);
    \draw[thick, purple](5-0.69, -1.86) arc (160:-15:0.05);
    
    \draw[thick, purple] (5.9952365235, -1.712615574) arc (205:50:0.3433893253);
    \draw[thick, purple] (5.69, -1.86) arc (200:35:0.1633049952);
    \draw[thick, purple](5.6, -1.88) arc (195:-20:0.05);
    \draw[thick, purple] (6.532088886, -1.285575219) arc (230:80:0.5358983849);
    \draw[thick, purple] (6.732050808, 1) arc (120:260:0.7279404685);
    \draw[thick, red] (5-0.5, -1.95) arc(170:10:0.52);
    
    \draw (5-0.5,-1.95) node{$\bullet$} node[below left]{$p$};
    \draw (5.5,-1.95) node{$\bullet$} node[below right]{$q$};
\end{tikzpicture}
    \caption{Left: a crown with the pivot p. Right: a blown-up crown where the pivotal side is colored in red.}
    \label{Fig:crownexample}
\end{figure}
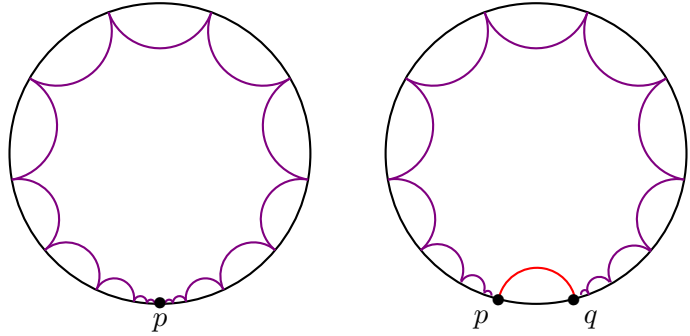

  An almost lamination $\Lambda$ is called quite full if it is obtained from a generalized quite full lamination with no crowns (i.e., all gaps are ideal polygons and blown-up crowns) by removing the pivotal side of each blown-up crown or a half-plane and removing at most one side of each ideal polygon (so all blown-up crowns, all half-planes and some polygons become cataclysms). 
  
  From now on, we will consider transverse bifoliar pairs of quite full almost laminations. Theorem \ref{thm:BBM} of \cite{BBM} guarantees that all bifoliar pairs are induced by pA-bifoliations. But the following proposition says that considering bifoliar pairs of quite full almost laminations is enough in the study of pseudo-Anosov flows of 3-manifolds.

   \begin{prop} \label{prop:bifoliarpA}
      Let $(\Lambda^+, \Lambda^-)$ be a bifoliar pair of almost laminations. $\Lambda^\pm$ are quite full if and only if $\mathcal{F}^\pm$ is a pA-bifoliation obtained from a pseudo-Anosov flow in 3-manifold. 
  \end{prop}
  \begin{proof}
      The only remaining thing to show is that the cataclysms of a quite full laminations are the only cataclysms that arise in the space of flowlines for a pseudo-Anosov flow. This follows from \cite[Theorem 4.9]{Fen99}.
  \end{proof}
  
  Moreover, in the case of bifoliar pair of quite full almost laminations, one can describe the construction of the corresponding bifoliated plane very explicitly using Moore's theorem (see \cite{moore1929concerning} and also \cite{timorin2010moore}). Consider geometric realizations of a bifoliar pair $\Lambda^\pm$ as geodesic (almost) laminations on the hyperbolic plane (as defined in Section \ref{subsec:fromcircle}). We first define a relation $\sim'$ on $\mathbb{H}^2$. For $x, y \in \mathbb{H}^2$, we say $x \sim' y$ if one of the followings holds:
  \begin{itemize}
      \item[(i)] $x, y \in \lambda \in |\Lambda^+|$ and $x, y$ are contained in the same component of $\lambda \setminus |\Lambda^-|$.
      \item[(ii)] $x, y \in \mu \in |\Lambda^-|$ and $x, y$ are contained in the same component of $\mu \setminus |\Lambda^+|$.
      \item[(iii)] $x \in |\Lambda^+| \cup |\Lambda^-|$, $y \in \mathbb{H}^2 \setminus (|\Lambda^+| \cup |\Lambda^-|)$, and there exists a path $\gamma$ from $x$ to $y$ in $\mathbb{H}^2$ such that $\gamma$ is contained in $(\mathbb{H}^2 \setminus (|\Lambda^+| \cup |\Lambda^-|) ) \cup \{x\}$. 
  \end{itemize}
  Let $\sim$ be the closed equivalence relation generated by $\sim'$, and we call the equivalence classes \emph{slits}. It is important to remember that the slits are closed. We collapse each slit which is contained in a compact subset of $\mathbb{H}^2$ to a single point, and simply remove the other slits. For example, we remove any half-plane gap together with its boundary. This quotienting certainly satisfies the conditions of Moore's theorem. Hence, the quotient is homeomorphic to an open disk again, call it $\mathcal{P}$. The detail of this process is completely analogous to the argument given in Baik-Jung-Kim \cite{baik2022groups}. Also, it is evident that the image of a leaf of geometric realizations under this quotient map is homeomorphic to the real line (let $q$ denote this quotient map). Hence $\Lambda^\pm$ descend to foliations $\mathcal{F}^\pm$. Hence, we get a bifoliated plane $(\mathcal{P}, \mathcal{F}^\pm)$. Verifying that the pair of foliations $\mathcal{F}^\pm$ is a pA-bifoliation requires additional work, but we omit it since it is taken care of in \cite{BBM}. 

  We remark that for a cataclysm, there exists a sequence of leaves converging to the removed leaf. By collapsing slits contained in the cataclysm, one sees that in the bifoliated plane, the image of such a sequence of leaves converges to the union of leaves which are images of the sides of the cataclysm, except the half-plane case. Hence, any two sides of a cataclysm are not separable in the leaf space of the corresponding foliation after collapsing and vice versa.

In the rest of the section, we explain how to interpret the Anosov-like actions on bifoliated planes as group actions on the circle with invariant bifoliar pair of almost laminations. 
  
  Recall that for a leaf $\ell_0$ of $\Lambda^+$, the set $\{(\ell_0, \ell): \ell \in \Lambda^- \mbox{ which is linked with } \ell_0\}$ is called a thread in $\Lambda^+$. The endpoints of $\ell_0$ are also called the endpoints of the corresponding thread. By abusing the notation, we also call the set $\{\ell: \ell \in \Lambda^- \mbox{ which is linked with } \ell_0\}$  a thread in $\Lambda^+$. We define threads in $\Lambda^-$ similarly. Note that there is a natural linear order on each thread (up to reversing the order).
 
 A nontrivial convex subset of a thread with minimal and maximal points is called an interval in $\Lambda^\pm$. An interval $I_+$ of $\Lambda^+$ and an interval $I_-$ of $\Lambda^-$ are said to be a linking pair if every element of $I_+$ is linked with every element of $I_-$ and vice versa. 

   Suppose we have points $x_1, x_2, \dots, x_{2n}$ on $S^1$ such that, as an ordered tuple, they are positively oriented according to the natural cyclic order on $S^1$. If $(x_i, x_{i+1})$ is a leaf of $\Lambda^+$ for each odd $i$ and $(x_i, x_{i+1})$ is a leaf of $\Lambda^-$ for each even $i$, then $\{x_1, \ldots, x_{2n}\}$ is called an ideal $2n-$chain of the pair $(\Lambda^+, \Lambda^-)$. 

   We consider a triple $(G, \Lambda^+, \Lambda^-)$ where $G$ is a group of orientation-preserving homeomorphisms of the circle and $\Lambda^\pm$ is a $G$-invariant bifoliar pair of quite full almost laminations. Such a triple $(G, \Lambda^+, \Lambda^-)$ is called an Anosov-like triple if 
   \begin{enumerate}
       \item [(B1)] Each nontrivial element $g$ of $G$ fixes at most one element of each thread. Such a fixed point is the intersection of a thread from $\La$ and a thread from $\Lb$. Then, $g$ topologically expands one of the threads while topologically contracting the other. 
       \item[(B2)] There exists a linked pair of leaves $\lambda_\pm \in \Lambda^\pm$ such that for every linking pair $I_\pm$ of intervals of $\Lambda^\pm$, there exists $g \in G-\{1\}$ with $g(\lambda_+) \in I_+$ and $g(\lambda_-) \in I_-$. 
       \item[(B3)] Let $I_\pm$ be a linking pair of intervals of $\Lambda^\pm$. Then there exists $g \in G-\{1\}$ such that either $g(I_+) \subset I_+, I_- \subset g(I_-)$ or $g(I_-) \subset I_-, I_+ \subset g(I_+)$. 
       \item[(B4)] For every linked pair of leaves $\lambda_\pm \in \Lambda^\pm$, the stabilizer in $G$ is either trivial or isomorphic to $\mathbb{Z}$. If both $\lambda_\pm$ are singular, then the stabilizer must be nontrivial, ie., isomorphic to $\mathbb{Z}$. 
       \item[(B5)] If $\lambda_1, \lambda_2$ are two sides of a cataclysm in $\Lambda^+$ or $\Lambda^-$, then there exists $g \in G-\{1\}$ which fixes $\lambda_1, \lambda_2$ simultaneously. 
       \item[(B6)] $(\Lambda^+, \Lambda^-)$ has no ideal $4$-chain. 
   \end{enumerate}

We say a triple $(G, \Lambda^+, \Lambda^-)$ satisfies Property (B7) if the following holds: suppose that $\ell_1 \in \Lambda^+$ and $\ell_2, \ell_3 \in \Lambda^-$ so that both $\ell_2$ and $\ell_3$ are linked with $\ell_1$ (so we are describing two elements in a single thread). Consider the product $Y$ of the leaf space of $\Lambda^+$ and the leaf space of $\Lambda^-$. Then for each $i = 2, 3$, there exists a neighborhood $U_i$ of $(\ell_1, \ell_i)$ in $Y$ so that no nontrivial element $g$ of $G$ satisfies $g((\ell_1, \ell_i)) \in U_i$ for both $i = 2, 3$. The same is true when we replace the roles of $\Lambda^+$ and $\Lambda^-$.  
This property corresponds to Property (A7) (see Definition \ref{A.7} in Section 3). Philosophically speaking, this property tells us that the action of $G$ on the space of linked pairs of leaves is in some sense acylindrical.

\begin{THM} A group $G$ acts faithfully on $S^1$ by orientation-preserving homeomorphisms with an invariant bifoliar pair of laminations $\Lambda^\pm$ so that $(G, \Lambda^+, \Lambda^-)$ is an Anosov-triple satisfying Property (B7) if and only if $G$ admits an Anosov-like action on the bifoliated plane $(\mathcal{P}, \mathcal{F}^\pm)$ which satisfies Property (A7) and induces $\Lambda^\pm$.  
\end{THM}
\begin{proof}
    With Proposition \ref{prop:bifoliarpA}, it is enough to show that $(G, \Lambda^\pm)$ admits Properties (B1)-(B7) for if and only if $(\mathcal{P}, \mathcal{F}^\pm)$ has Properties (A1)-(A7). We already noted above that (B7) directly translates into (A7). 

    First, the threads of $\La, \Lb$ are mapped to the leaves of $\Fa, \Fb$ under the quotient map $q$ which collapses the slits. Hence, condition (A1) directly translates into $(B1)$ and vice versa. 

    A linked pair of leaves corresponds to a point on $\mathcal{P}$ and link pair of intervals corresponds to a product neighborhood of a point in $\mathcal{P}$. Hence (B2) is equivalent to having a dense orbit in $\mathcal{P}$, i.e., the condition (A2). Again in this perspective, (B3) says that any product chart has a fixed point of some nontrivial element of $G$, hence the set of such points are dense in $\mathcal{P}$, the condition (A3). It is also evident that (A4) directly translates into (B4) and vice versa. 

    By the remark before Proposition \ref{prop:bifoliarpA}, (A5) holds if and only if (B5) holds. To be more precise, two leaves in $\Fa$ or $\Fb$ are not separated in $\mathcal{F}$ if and only if there exists a collection of leaves including these two leaves where some infinite sequence of leaves converges the union of the collection of leaves. This happens precisely when the preimages of these two leaves under $q$ are threads along two leaves that are sides of a single cataclysm. 

    Clearly, (A6) directly translates into (B6) and vice versa. 
\end{proof}

The discussion in this section shows that we can move back and forth between group actions on the circle with a bifoliar pair of almost laminations and group actions on the bifoliated plane when the actions are Anosov-like. Both frameworks carry the same amount of information. The construction of a 3-manifold with a (pseudo-)Anosov flow from the former is described in Section \ref{subsec:fromcircle}, while the construction from the latter is outlined in Section \ref{sec:realizingtopflow} and Section \ref{sec:frombifoliatedplane}. As a result of this discussion, one can directly translate between the two constructions.

In the construction given in Section \ref{sec:realizingtopflow} and \ref{sec:frombifoliatedplane}, not all conditions of the Anosov-like action are required. In fact, the properties (A1), (A4), (A7) are the ones we need in the constructions. Hence, we call an action of $G$ on the bifoliated plane satisfying the properties (A1), (A4), and (A7) \emph{flowable}. The corresponding action in the circle (i.e. $G$ action on the circle with a bifoliar pair of quite full almost laminations satisfying properties (B1), (B4), and (B7)) is also called \emph{flowable}. 

As discussed in Subsection \ref{subsec:fromcircle}, a flowable action of $G$ on circle yields a 3-manifold with (pseudo-)Anosov flow and here we describe how one can obtain this from previous sections.  

Let $(\Lambda^+, \Lambda^-)$ be a bifoliar pair of almost laminations on $S^{1}$,
and we may identify this $S^{1}$ with the ideal boundary of the Poincar\'e disk $\DD$.
Recall from Subsection \ref{subsec:fromcircle},
we denote by $|\Lambda^{+}|, |\Lambda^{-}|$ the geometric realizations of 
$\Lambda^{+}, \Lambda^{-}$ respectively,
and for any leaf $\lambda, \mu$ of $\La, \Lb$ respectively,
we denote by $|\lambda|$ (resp. $|\mu|$) the leaf of $\La$ (resp. $\Lb$) whose endpoints
are exactly $\lambda$ (resp. $\mu$).
Based on the language built in this section,
we have a well-defined quotient map $\pi: \DD \to \PP$, 
where $\PP$ is a bifoliated plane with transverse singular foliations $\Fa, \Fb$,
such that
$\pi$ takes all polygons of $\La$ (resp. $\Lb$) to singular leaves of $\Fa$ (resp. $\Fb$) and
takes other leaves of $|\La|$ (resp. $|\Lb|$) to non-singular leaves of $\Fa$ (resp. $\Fb$).
Now assume that $\La, \Lb$ are orientable,
and we consider a flowable action of a torsion-free group $G$ on $\PP$ that
preserves the orientations on the leaves of $\La$ and $\Lb$.
Then we can translate Lemma \ref{equivalence class} to the following version:
for each polygon $P$ of $|\La|$,
we pull-back the map $h_{\pi(P)}$ as given in Lemma \ref{equivalence class} to
a monotone map $h_P = h_{\pi(P)} \circ \pi: P \to \R$,
then $h_P$ is equivariant under the stabilizer of $P$ and satisfies other conditions
as given in Subsection \ref{subsec:fromcircle}.
Therefore,
we can translate
the construction given in Section \ref{sec:frombifoliatedplane} to
the construction described in Subsection \ref{subsec:fromcircle}.

\section{Reconstruction flows in hyperbolic $3$-manifolds from 
convergence group actions on $S^{2}$}\label{sec: convergence group action}

Throughout this section,
pseudo-Anosov flows will always refer to topological pseudo-Anosov flows. 

Recall from Definition \ref{flow ideal boundary} and Theorem \ref{bounded to convergence},
a quasigeodesic (topological) pseudo-Anosov flow $\phi$ in a closed $3$-manifold $M$
(with Gromov hyperbolic group) is \emph{bounded},
and it has a canonical flow ideal boundary $\RR(\phi)$ 
with an induced uniform convergence group action of $\pi_1(M)$.
For a bifoliated plane $\PP$ with certain properties,
Definition \ref{flow ideal boundary} generalizes to an \emph{abstract flow ideal boundary},
which is still homeomorphic to $S^{2}$.
According to Cannon's conjecture \cite{CS98},
if an action of $G$ on $\PP$ induces a uniform convergence group action on this $S^{2}$,
then there is a closed hyperbolic $3$-manifold with deck group $G$.
In this section, we give a partial answer to this conjecture by reconstructing the hyperbolic $3$-manifolds with quasigeodesic topological pseudo-Anosov flows from certain uniform convergence group actions on $S^2$. 

Let $G$ be a torsion-free group acting on a bifoliated plane $(\PP, \Fa, \Fb)$
satisfying Properties (A1), (A4) in Definition \ref{Anosov-like}.
We adopt Convention \ref{conv: leaves and segments} for $(\PP, \Fa, \Fb)$.
Furthermore,
for a point $x \in \PP$,
let $\lambda^{+}_{x}$ (resp. $\lambda^{-}_{x}$) be the union of components of $\lambda_x - \{x\}$ in the positive (resp. negative) side of $x$,
and let $\mu^{+}_{x}$ (resp. $\mu^{-}_{x}$) be the components of $\mu_x - \{x\}$ in the positive (resp. negative) side of $x$.
Note that $\lambda^{\pm}_{x}, \mu^{\pm}_{x}$ are disconnected if and only if
$x$ is a singularity.

We first review the concept of chain of perfect fits in \cite{Fen16}.

\begin{defn}[Chain of perfect fits]\rm
    A set of leaves
    \[\mathcal{C} = \{l_\alpha \mid l_\alpha \text{ is a leaf of } \Fa \text{ or } \Fb\}\]
is called a \emph{chain of perfect fits} if the following conditions hold:

\begin{enumerate}[(1)]
    \item For any two leaves $a,b \in \mathcal{C}$,
there are finitely many leaves 
\[a = l_0, l_1,\ldots,l_{n-1}, l_n = b \in \mathcal{C}\]
such that $l_i, l_{i+1}$ makes a perfect fit for all $0 \leqslant i \leqslant n-1$.
\item Let $l \in \mathcal{C}$.
For any leaf $s$ of $\Fa$ or $\Fb$ with $s \notin \mathcal{C}$,
$l, s$ doesn't make perfect fit.
\end{enumerate}
\end{defn}

To make $\PP$ have a well-defined $2$-sphere abstract flow ideal boundary,
we made the following assumptions on $(\PP, \Fa, \Fb)$:

\begin{assume}[Basic properties for 
the abstract flow ideal boundary]\rm\label{assume: no infinite chain}  
Any two distinct leaves in the same chain of perfect fits are disjoint, and the leaves in the same chain of perfect fits do not form 
a (finite-sided) ideal polygon.
Moreover,
there is no infinite chain of perfect fits,
    and $\Fa, \Fb$ have no product region (see an explanation below).
\end{assume}

Here, by a \emph{product region} of $\Fa$ or $\Fb$,
we mean a region described as follows.
Let $x,y \in \PP$ with $\lambda_x = \lambda_y$ such that,
for any $t \in [x,y]_{\lambda_x}$,
each leaf of $\Fb$ intersecting $\mu^{+}_{t}$ must intersect both of
$\mu^{+}_{x}, \mu^{+}_{y}$,
where the infinite region bounded by $\mu^{+}_{x}, \mu^{+}_{y}$ and $[x,y]_{\lambda_x}$ is
then a product region of $\Fb$ on the positive side.
Product regions of $\Fb$ on the negative side, 
as well as product regions of $\Fa$, 
can be defined similarly.

Assumption \ref{assume: no infinite chain} is necessary to make $\PP$ have 
an abstract flow ideal boundary homeomorphic to $S^{2}$ from 
the approach of \cite[Theorem 6.5]{Fen16}.

\begin{fact}\rm
    If $\phi$ is a bounded pseudo-Anosov flow, then Assumption \ref{assume: no infinite chain} holds for $\Or(\phi)$.
\end{fact}
\begin{proof}
    It is proved in \cite[Proposition 6.3]{Fen16} that 
all chains of perfect fits consist of disjoint leaves,
and that leaves in the same chain of perfect fits cannot form a finite-sided ideal polygon.
In addition,
bounded pseudo-Anosov flows cannot have infinite chains of perfect fits \cite[Theorem C]{Fen16}.
As bounded pseudo-Anosov flows 
are not topologically conjugate to suspension Anosov flows,
they cannot have product regions \cite[Theorem 5.1]{Fen98}, \cite[Theorem 4.10]{Fen99}.
Thus, Assumption \ref{assume: no infinite chain} holds.
\end{proof}

As illustrated above,
Assumption \ref{assume: no infinite chain} holds for the orbit space of
any pseudo-Anosov flow without product regions and infinite chain of perfect fits.

Let $S^{1}_{\infty}$ denote the ideal boundary of $\PP$.
The following definition is Definition \ref{flow ideal boundary} 
adapted to the setting of bifoliated planes:

\begin{defn}[Abstract flow ideal boundary]\rm\label{abstract flow ideal boundary}
    For $a,b \in S^{1}_{\infty}$,
    we define $a \sim b$ if $a,b$ are contained in two leaves $l_1, l_2$ of $\Fa$ or $\Fb$ such that $l_1, l_2$ are in the same chain of perfect fits
(it is possible that $l_1 = l_2$).
    Let $\RR = S^{1}_{\infty} / \sim$.
\end{defn}

It is guaranteed in \cite[Theorem 6.5]{Fen16} that

\begin{prop}[Fenley]\label{prop:2-sphere}
    $\RR$ is homeomorphic to a $2$-sphere,
    and $S^{1}_{\infty}$ can be canonically identified with a sphere-filling Peano curve.
\end{prop}

Next, we will show that

\begin{prop}\label{prop: convergence group action}
    Suppose that the induced action of $G$ on $\RR$ is a convergence group action.
\begin{enumerate}[(a)]
    \item The action of $G$ on $\PP$ is flowable.
    Thus, there is a $3$-manifold $M$ with a reduced pseudo-Anosov flow $\phi$
    such that $(M,\phi)$ realizes the action of $G$ on $\PP$.

    \item Suppose further that the action of $G$ on $\RR$ is uniform.
    Then either the resulting $3$-manifold $M$ is closed hyperbolic, or $M$ is a non-compact manifold with Gromov hyperbolic group and with no $T^2$ end.

    \item
    Under the assumption of (b), in the case $M$ is closed, 
    $\phi$ is a quasigeodesic topological pseudo-Anosov of $M$.
\end{enumerate}
\end{prop}

We conjecture that the second case in part (b) cannot happen, i.e. $M$ is always a closed hyperbolic 3-manifold. 

\begin{proof}[The proof of (a)]
As the action of $G$ on $\PP$ already satisfies Properties (A1) and (A4),
it suffices to prove that Property (A7) (Definition \ref{A.7}) holds.

    Let $L_+, L_-$ denote the leaf space of $\Fa, \Fb$ respectively.
    Let $q_+: L_+ \to \RR$, $q_-: L_- \to \RR$ be 
    the canonical quotient maps from Definition \ref{abstract flow ideal boundary}.
    Let $\Theta(\RR)$ be the set of distinct triples on $\RR$,
    where the action of $G$ on $\RR$ induces 
    a free and discrete action on $\Theta(\RR)$
    \cite{Tuk94, Bow99}.

    Let $\lambda$ be a leaf of $\Fa$ and
    let $x, y$ be distinct points on $\lambda$ with $x \in \lambda^{+}_{y}$.
    By Assumption \ref{assume: no infinite chain},
    $q_+(\lambda) \ne q_-(\mu_x), q_+(\lambda) \ne q_-(\mu_y)$.

    We first assume that $q_-(\mu_x) \ne q_-(\mu_y)$.
    Because the action of $G$ on $\Theta(\RR)$ is free and discrete,
    there is a neighborhood $U$ of the distinct triple
    $(q_+(\lambda), q_-(\mu_x), q_-(\mu_y))$ in 
    $\Theta(\RR)$ such that $g(U) \cap U = \emptyset$.
    Hence, 
    there exists a neighborhood $U_1$ of $\lambda$ in $L_+$ and
    neighborhoods $U_2, U_3$ of $\mu_x, \mu_y$ in $L_-$ such that,
    for any leaf $l_1 \in U_1$ and two leaves 
    $l_2 \in U_2, l_3 \in U_3$ with $q_-(l_2) \ne q_-(l_3)$,
    we have $(q_+(l_1), q_-(l_2), q_-(l_3)) \in U$.
    This implies that any $g \in G - \{1\}$ can not send 
    $(q_+(\lambda), q_-(\mu_x), q_-(\mu_y))$ into $\{(a, b, c): a\in U_1, b\in U_2, c\in U_3\}$.
    We can choose neighborhoods $V_x, V_y$ of $x, y$ in $\PP$ such that,
    for any $t \in V_x$, $s \in V_y$,
    we have $\lambda_t, \lambda_s \in U_1$, $\mu_t \in U_2$, $\mu_s \in U_3$.
    It is clear that all $g \in G - \{1\}$ can not send $\{x,y\}$ to two points in $V_x$ and $V_y$ respectively. 

    Now assume that $q_-(\mu_x) = q_-(\mu_y)$.
    We show that the following dose not hold: any leaf of $\Fa$ 
    intersects $\mu_x$ if and only if it intersects $\mu_y$. We will show that this implies that $\PP$ either has a product region of $\Fb$ or
    has a finite-sided ideal polygon formed by
    some leaves in the chain of perfect fits containing $\mu_x$, which contradicts Assumption \ref{assume: no infinite chain}.
    We shall explain this below:

    Suppose that each leaf of $\Fa$ intersects $\mu_x$ if and only if 
    it intersects $\mu_y$.
    Let $p_x: (0,1) \to L_+$ be the embedding whose image is
    the union of those leaves of $\Fa$ that intersect $\mu_x, \mu_y$.
    It's clear that, for any $s \in [x,y]_{\lambda}$,
    $\mu_s \cap p_x(t) \ne \emptyset$ for all $t \in (0,1)$.
    Thus, both of $\lim_{t \to 0} p_x(t), \lim_{t \to 1} p_x(t)$ exist; otherwise there must exist a product region of $\Fb$.
    Let $P_- = \lim_{t \to 0} p_x(t)$,
    $P_+ = \lim_{t \to 1} p_x(t)$,
    where each of $P_-, P_+$ is
    either a unique point of $L_+$ or a set of non-separated points of $L_+$.
    One can see that
    there must be a product region of $\Fb$ if
    either of $\mu_x, \mu_y$ does not make perfect fit with some leaf of $P_-$ or $P_+$.
    Thus,
    both of $\mu_x, \mu_y$ make perfect fit with some leaves of $P_-$ and
    with some leaves of $P_+$,
    it follows that there exists a finite-sided ideal polygon formed by
    $\mu_x, \mu_y$ and some leaves in $P_-, P_+$.
    This contradicts Assumption \ref{assume: no infinite chain}.

    Hence, there is a leaf $l$ of $\Fa$ that intersects exactly one of $\mu_x, \mu_y$.
    Without loss of generality,
    we may assume that $l \cap \mu_x \ne \emptyset$ and
    $l \cap \mu_y = \emptyset$,
    and that $l \cap \mu_x \in \mu^{+}_{x}$.
    Let $J$ denote the component of $\mu^{+}_{x}$ with $l \cap \mu_x \in J$
    (where $J = \mu^{+}_{x}$ if $x$ is not a singularity).

    There is a point $z \in J$ such that (1) for each $t \in [x,z]_{\mu_x} - \{z\}$,
    $\lambda_t$ intersects $\mu_y$, and (2) for each $t \in J - [x,z]_{\mu_x}$,
$\lambda_t$ does not intersect $\mu_y$.
    Then one of the following three possibilities holds:

\begin{case1}\rm\label{case1 1}
    $\lambda_z \cap \mu_y \ne \emptyset$.
    Then $\lambda_z$ is a singular leaf of $\Fa$.
\end{case1}

Let $t_1, t_2, \ldots$ be a collection of points in $[x,z]_{\mu_x} - \{z\}$ with
$\lim_{n \to +\infty} t_i = z$,
and let $s_i = \lambda_{t_i} \cap \mu_y$ for each $i \in \N$.

\begin{case1}\rm\label{case1 2}
$\lambda_z \cap \mu_y = \emptyset$ and
    $\lim_{n \to +\infty} s_i = s$ for some $s \in \mu_y$.
    Then $\lambda_s, \lambda_t$ are a pair of distinct non-separated leaves.
\end{case1}

\begin{case1}\rm\label{case1 3}
$\lambda_z \cap \mu_y = \emptyset$ and
    $\lim_{n \to +\infty} s_i$ is an end of $\mu_y$.
    Then $\mu_y$ makes perfect fit with some leaf of $\Fa$ non-separated with $\lambda_z$.
\end{case1}

\begin{figure}
	\centering
	\subfigure[]{
		\includegraphics[width=0.3\textwidth]{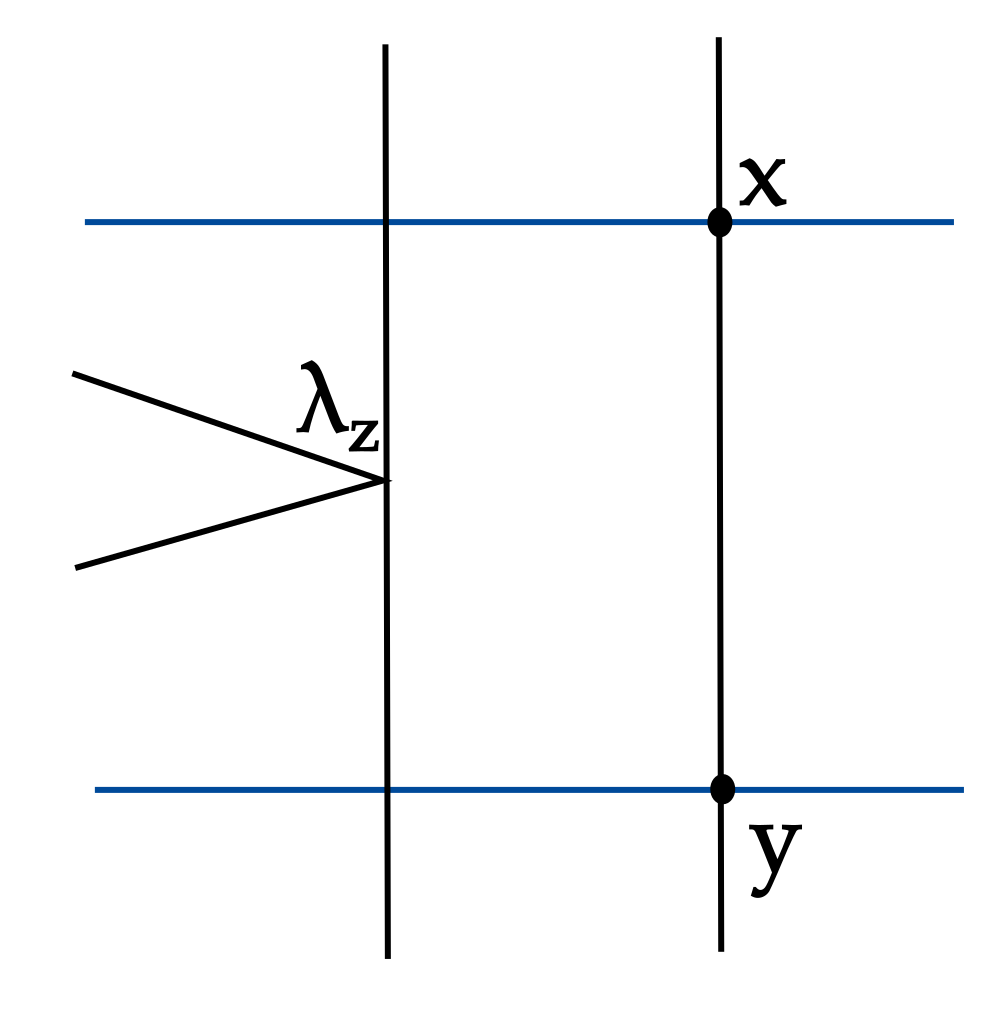}}
	\subfigure[]{
		\includegraphics[width=0.3\textwidth]{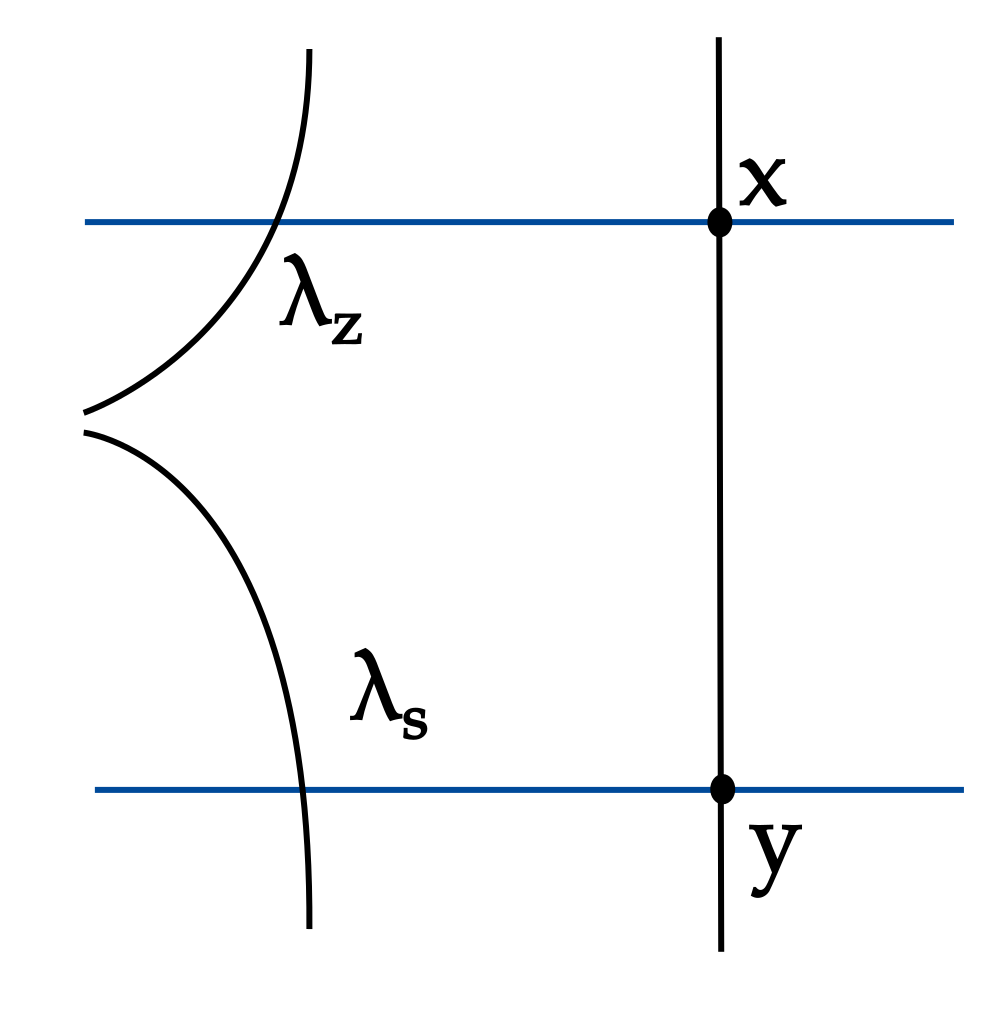}}
  \subfigure[]{
		\includegraphics[width=0.3\textwidth]{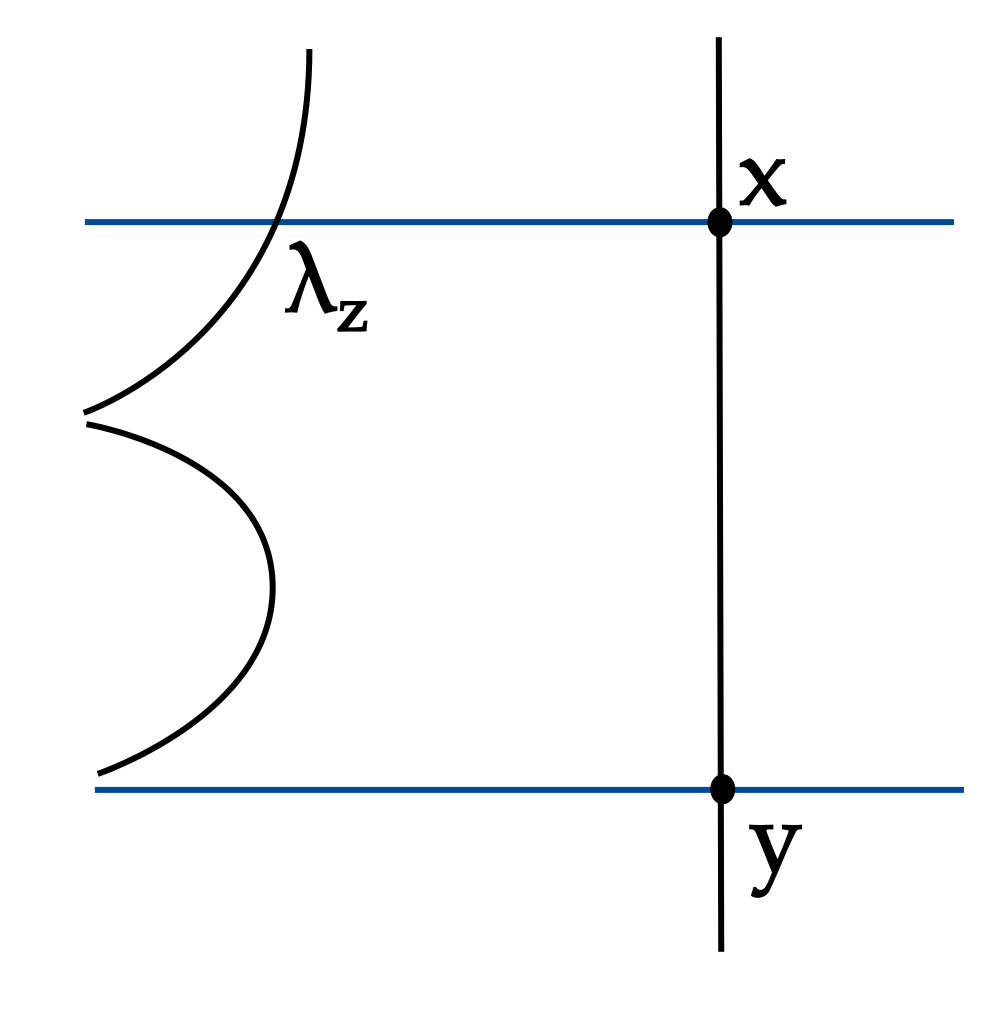}}
  \caption{This figure describes the leaf $\lambda_z$ in Cases \ref{case1 1}$\sim$\ref{case1 3}. (a) illustrates Case \ref{case1 1}, where $\lambda_z$ is a singular leaf. (b) illustrates Case \ref{case1 2}: $\lambda_z$ is non-separated $\lambda_s$ for some $s \in \mu_y$. (c) illustrates Case \ref{case1 3}; in this case, $\lambda_z, \mu_y$ are contained in the same chain of perfect fits.}\label{fig: three cases}
  \end{figure}

We first consider Case \ref{case1 1}.
Because the union of singularities of $\Fa$ is a dense subset of $\PP$,
there exists a neighborhood $U_x$ of $x$ and a neighborhood $U_y$ of $y$ such that,
for any $t_1 \in U_x$, $t_2 \in U_y$,
$\lambda_{t_1} \cap \mu_z, \lambda_{t_2} \cap \mu_z \ne \emptyset$.
If $g \in G - \{1\}$ satisfies that $g(x) \in U_x$ and $g(y) \in U_y$,
then $g(\mu_z) = \mu_z$.
It suffices to show that the elements in the stabilizer of $\mu_z$ cannot take $x, y$ sufficiently close to themselves.
Let $h$ be a generator of the stabilizer of $\mu_z$.
We can choose a neighborhood $V_x \subseteq U_x$ of $x$ and
a neighborhood $V_y \subseteq U_y$ of $y$ such that
$\lambda_{h(x)}, \lambda_{h^{-1}(x)} \cap V_x = \emptyset$,
$\lambda_{h(y)}, \lambda_{h^{-1}(y)} \cap V_y = \emptyset$.
Recall that each element of the stabilizer of $\mu_z$ 
either topologically expands all half-leaves of $\mu_z$ or
topologically contracts them,
we can ensure that $\lambda_{h^{k}(x)} \cap V_x = \emptyset$ for all $k \in \N$.
Therefore,
for any $g \in G - \{1\}$,
at least one of $g(x) \notin V_x$, $g(y) \notin V_y$ holds.

Now we consider Case \ref{case1 2}.
In this case, there exists $s \in \mu_y$ such that
$(\lambda_s, \lambda_z)$ is a pair of distinct non-separated leaves.
We still can choose a neighborhood $U_x$ of $x$ and a neighborhood $U_y$ of $y$ such that,
for any $t_1 \in U_x$, $t_2 \in U_y$,
$\lambda_{t_1} \cap \mu_z, \lambda_{t_2} \cap \mu_s \ne \emptyset$.
For $g \in G - \{1\}$ with $g(x) \in U_x$ and $g(y) \in U_y$,
we must have $g(\mu_z) = \mu_z$ and $g(\mu_s) = \mu_s$.
It follows that
$g$ is contained in the stabilizers of $\mu_z$ and $\mu_s$.
Let $\mu_0$ denote the leaf of $\Fb$ that makes perfect fit with $\mu_z$
(at the negative end of $\mu_z$).
Then $g$ is also contained in the stabilizer of $\mu_0$
(and thus $\mu_0$ is a periodic leaf).
Let $h$ be a generator of the stabilizer of $\mu_z$.
We can also choose a neighborhood $V_x \subseteq U_x$ of $x$ and 
a neighborhood $V_y \subseteq U_y$ of $y$ such that,
$\lambda_{h(x)}, \lambda_{h^{-1}(x)} \cap V_x = \emptyset$,
$\lambda_{h(y)}, \lambda_{h^{-1}(y)} \cap V_y = \emptyset$.
Similar to the above case,
we can ensure that $g(x) \notin V_x$, $g(y) \notin V_y$ for all $g \in G - \{1\}$.

Finally we consider Case \ref{case1 3}.
The conditions of Case \ref{case1 3} implies that
$\lambda_z, \mu_y$ are contained in the same chain of perfect fits.
However,
$\mu_x, \mu_y$ are also contained in the same chain of perfect fits,
and $\lambda_z \cap \mu_x \ne \emptyset$.
This contradicts Assumption \ref{assume: no infinite chain}.

Therefore,
the action of $G$ on $\PP$ is flowable.
Recall that $N = \{(x,y) \mid x \in \PP, y \in r_x\}$ as in Construction \ref{construction of N} (a),
and let $\w{\phi}$ be the flow constructed in Construction \ref{construction of N} (b).
As shown above,
the action of $G$ on $\PP$ is flowable,
and thus $G$ acts on $N$ freely and discretely.
Let $M = N / G$ be the quotient $3$-manifold and
let $\phi = \w{\phi} / G$ be the induced flow in $M$.
Then $(M, \phi)$ realizes the action of $G$ on $\PP$.
This completes the proof of (a).
\end{proof}

\begin{proof}[The proof of (b)]


If the action of $G$ on $\RR$ is a uniform convergence action, by Bowditch \cite{bowditch1998topological}, $G$ is Gromov hyperbolic and $\mathcal{R}$ is $G$-equivariantly homeomorphic to the Gromov boundary of $G$. Since $G$ is Gromov hyperbolic, $G$ does not contain any $\mathbb{Z}^2$ subgroup. So $M$ is either closed or non-compact with no $T^2$ end. 

By Thurston-Perelman geometrization theorem, a closed 3-manifold with a Gromov hyperbolic fundamental group is a hyperbolic 3-manifold. Now, Part (b) follows. 


\end{proof}

\begin{proof}[The proof of (c)]
    Recall that $S^{1}_{\infty}$ denotes the ideal boundary of $\PP$.
    Let $D = S^{1}_{\infty} \cup \PP$ and $C = D \times I$.
    Then $\partial C$ is exactly $(D \times \{0,1\}) \cup (S^{1}_{\infty} \times I)$.
    Following \cite{Fen12}, \cite{Fen16},
    let $\sim$ be the closed equivalence relation on $C$ generated by 
    the following relations,
    (1)
    $(x,0) \sim (y,0)$ when there exists a leaf $\mu$ of $\Fb$ such that
    each of $x,y$ is either contained in $\mu$ or is the ideal endpoint of $\mu$
    in $S^{1}_{\infty}$,
    (2)
    $(x,1) \sim (y,1)$ when there exists a leaf $\lambda$ of $\Fa$ such that
    each of $x,y$ is either contained in $\lambda$ or is the ideal endpoint of $\lambda$
    in $S^{1}_{\infty}$,
    (3)
    For $(x,t_1), (y,t_2) \in S^{1}_{\infty} \times I$,
    $(x,t_1) \sim (y,t_2)$ if and only if $x = y$.
    Let $S^{2}_{\infty}$ be the quotient of $\partial C$ under the above equivalence relation,
    and let $q: \partial C \to S^{2}_{\infty}$ be the quotient map.
    Because $N = \w{M}$ can be identified with $Int(C) \cong \PP \times (0,1)$,
    $N \cup S^{2}_{\infty}$ is a compactification of $N$.

    Note that we can canonically identify the abstract flow ideal boundary $\RR$ with $S^{2}_{\infty}$.
    For any $\lambda \in L_+$,
    $q_+(\lambda) \in \RR$ is canonically identified with $q(\lambda,1)$.
    For any $\mu \in L_-$,
    $q_-(\mu) \in \RR$ is canonically identified with $q(\mu,0)$.
    Then $S^{2}_{\infty}$ is homeomorphic to $\RR$ and thus is a $2$-sphere.

    Because $M$ is a closed hyperbolic $3$-manifold, $N$ can be identified with $\mathbb{H}^{3}$ and the action of $G$ on $N$ would be a co-compact Klein group acting on 
    $\mathbb{H}^{3}$, and by the same argument as in Part (b) above, $S^{2}_{\infty}$, which is homeomorphic to $\RR$, can be identified with the ideal $2$-sphere of $N \cong \mathbb{H}^{3}$.

    For any $x \in \PP$,
    we denote by $\rho_x$ the flowline of $\w{\phi}$ projected to $x$.
    Then $\rho_x$ is a quasigeodesic with ideal endpoints 
    $q_+(\lambda_x), q_-(\mu_x) \in S^{2}_{\infty}$.
    We show that $\rho$ is an expansive flow of $M$ as follows.
    
    Let $d$ be a metric on $M$,
    which induces a metric on $N$.
    We choose a collection of sufficiently small product charts
    $\{U_\alpha\}_{\alpha \in \Psi}$ (where $\Psi$ is an index set) such that,
    (1) their interiors cover $M$,
    (2) the lift of each $U_\alpha$ ($\alpha \in \Psi$) to $\w{M}$ projects onto
    a product rectangle or a product polygon of $\PP$.
    There exists $\epsilon > 0$ 
    the $\epsilon$-neighborhood any point in $M$ is contained in some $U_\alpha$
    ($\alpha \in \Psi$).
    And there exists $R > 0$ sufficiently large so that, 
    for any $t_1, t_2 \in M$,
    if $t_1 \in U_\alpha$ for some $\alpha \in \Psi$ and $d(t_1,t_2) > R$,
    then $t_2 \notin U_\alpha$.

    Let $x,y \in \PP$ distinct.
    We claim that,
    for any orientation-preserving homeomorphism $h: \rho_x \to \rho_y$,
    there is $t \in \rho_x$ with $d(t,h(t)) > \epsilon$.
    It's clear that 
    the claim holds if $\rho_x, \rho_y$ have distinct positive or negative ideal endpoints.
    Hence we can confirm this claim when
    $q_+(\lambda_x) \ne q_+(\lambda_y)$ or $q_-(\mu_x) \ne q_-(\mu_y)$.

    Now suppose that $q_+(\lambda_x) = q_+(\lambda_y)$ and
    $q_-(\mu_x) = q_-(\mu_y)$.
    We can choose $z \in \PP$ such that
    (1) $\lambda_z$ separates $\lambda_x, \lambda_y$ (when $\lambda_x = \lambda_y$, we allow $\lambda_z = \lambda_x, \lambda_y$),
    (2) $\mu_z$ separates $\mu_x, \mu_y$ (when $\mu_x = \mu_y$, we allow $\mu_z = \mu_x, \mu_y$),
    (3) $\{q_+(\lambda_x), q_-(\mu_x)\} \ne \{q_+(\lambda_z), q_-(\mu_z)\}$.
    Then there exists $t \in \rho_x$ with $d(t,\rho_z) > R$.
    Therefore,
    for each lift $\w{U_\alpha}$ of some $U_\alpha$ ($\alpha \in \Psi$) containing $t$,
    we have $\rho_z \cap \w{U_\alpha} = \emptyset$.
    Since $\w{U_\alpha}$ projects to a product rectangle or a product polygon in $\PP$,
    we can ensure that $\rho_y \cap \w{U_\alpha} = \emptyset$.
    Therefore, $d(t,\rho_y) > \epsilon$.
    
    Now the claim is confirmed,
    and hence $\phi$ is an expansive flow.
    It's proved in \cite[Theorem 1.5]{inaba1990nonsingular},
    \cite[Lemma 7]{paternain1993expansive} that
    $\phi$ preserves a pair of singular foliations $\mathcal{F}_1, \mathcal{F}_2$;
    in addition,
    for any $x, y \in \PP$,
    $\rho_x, \rho_y$ are contained in the same leaf of $\mathcal{F}_1$ 
    (resp. $\mathcal{F}_2$) if and only if
    there are orientation-preserving homeomorphisms 
    $h_x: \R \to \rho_x$, $h_y: \R \to \rho_y$ such that
    $\lim_{t \to +\infty}d(h_x(t), h_y(t)) = 0$
    (resp. $lim_{t \to -\infty}d(h_x(t), h_y(t)) = 0$),
    see \cite[Page 8]{Bru92} for an explanation.
    
    We now explain that $\F_1 = \Fs, \F_2 = \Fu$ is the only possible case.
    Let $x \in \PP$ and let $l_x$ be the leaf of $\w{\F_1}$ containing $\rho_x$,
    where $\w{\F_1}$ is the pull-back of $\F_1$ in $\w{M}$.
    Let $U$ be a sufficiently small neighborhood of $x$ in $\PP$ such that
    $q_+(\lambda) \ne q_+(\lambda_y), q_-(x) \ne q_-(\lambda_y)$ for
    all $y \in U - \{x\}$,
    and let $E_U = \bigcup_{t \in U} \rho_t$.
    We first assume that $x$ is not a singularity of $\Fa$.
    For any $t \in U$,
    $\rho_x, \rho_t$ share a positive ideal endpoint if and only if $t \in \lambda_x$,
    hence $l_x \cap E_U$ can only be $\bigcup_{t \in \lambda_x \cap U} \rho_t$.
    Now assume that $x$ is a singularity of $\Fa$.
    By the above discussion,
    for each component $L$ of $\lambda_x - \{x\}$,
    $\bigcup_{t \in L} \rho_t$ is contained in a leaf of $\w{\F_1}$,
    and we can ensure that $\bigcup_{t \in L} \rho_t \subseteq l_x$ by
    the local structure of $\w{\F_1}$.
    Therefore,
    $\w{F_1} = \w{\Fs}$ and hence $\F_1 = \Fs$.
    We can confirm that $\F_2 = \Fu$ similarly.

    For any distinct $x,y \in \PP$ in the same leaf $\lambda$ of $\Fa$,
    there is $z \in \lambda$ separating $x,y$ such that
    $q_-(\mu_z) \ne q_-(\mu_x), q_-(\mu_y)$.
    Let $l$ denote the leaf of $\w{\Fs}$ containing $\rho_x, \rho_y, \rho_z$,
    and let $d_l$ denote the path metric on $l$ induced from the metric $d$.
    For any $K > 0$,
    there is $t \in \rho_x$ such that 
    $d_l(t,\rho_z) > K$,
    which implies $d_l(t,\rho_y) > K$.
    Therefore,
    for any orientation-preserving homeomorphism $h_x: \R \to \rho_x$,
    we have $\lim_{t \to -\infty}d(h_x(t), \rho_y) = +\infty$.
    Similarly,
    we can ensure that flowlines in the same leaf of $\Fu$ have this property.
    Hence $\phi$ is a topological pseudo-Anosov flow.
\end{proof}

\begin{rmk}\label{rmk:adding compactness}
    A $G$-action on $\PP$ is said to have the \emph{compactness property} if
    there are finitely many flow boxes or singular flow boxes of $N$ (see Construction \ref{construction of N}) whose images under the $G$-action cover $N$.
    This is a generalization of Definition \ref{cptprop} to the context of Section \ref{sec:frombifoliatedplane}, which characterizes when $M$ is compact. If the $G$-action on $\PP$ is exactly the $\pi_1$-action of a smooth pseudo-Anosov flows in a closed hyperbolic $3$-manifold, then it satisfies the compactness property, as the flow is transitive and thus this $3$-manifold is homeomorphic to $M = N/G$ by Proposition \ref{equivalent}.

    Combined with Proposition \ref{prop: convergence group action} (b), (c), if $G$ acts on $\RR$ as a uniform convergence group action and $G$ satisfies the compactness property, then the action of $G$ on $\PP$ is realized by a closed hyperbolic $3$-manifold with a quasigeodesic pseudo-Anosov flow. 
\end{rmk}




Our reconstruction of $M$ provides a description of $M$ in terms of 
the Gromov boundary $\RR$,
which can be identified with the ideal boundary of $\mathbb{H}^{3}$ in the first case of Proposition \ref{prop: convergence group action} (b);
in fact, $M$ can be described from $\RR \times \RR \times \RR$.

\begin{rmk}\label{rem: unordered triple}
Assume that $\rho$ induces a uniform convergence group action on $\RR$.
    Recall from Subsection \ref{subsec: the second construction},
    the universal cover $N$ of $M$ is a quotient space of 
    $N_2 = \{(x,y) \mid x \in \PP, y \in R_x\}$,
    and each point $(x,y) \in N_2$ can be expressed as
    an ordered triple $(\lambda_x, \mu_x, \mu_y)$,
    which can be identified to a point 
    \[(q_+(\lambda_x), q_-(\mu_x), q_-(\mu_y)) \in \RR \times \RR \times \RR.\]
    Thus,
    $N_2$ can be canonically identified with a subspace of 
    $\RR \times \RR \times \RR$.

    We assume further that no leaves of $\Fa, \Fb$ make perfect fit.
    Then there exists an injective map
    $q: L_+ \cup L_- \to \RR$ with $q \mid_{L_+} = q_+$, $q \mid_{L_-} = q_-$.
    Hence any ordered triple $(\lambda,\mu_1,\mu_2)$ with
    $\lambda \in L_+$, $\mu_1, \mu_2 \in L_-$ can be uniquely identified with
    a point of the distinct ordered triple $\Theta(\RR)$ of $\RR$.
    Recall that $G$ acts on $\Theta(\RR)$ co-compactly,
    as $G$ is a uniform convergence group action.
    In this case, $N_2$ is a subspace of $\Theta(\RR)$.
    Since $N$ is a quotient space of $N_2$ under 
    a $G$-equivariant equivalence relation and $M = N / G$,
    $M$ is a quotient space of $N_2 / G$ and $N_2 / G$ is a subspace of $\Theta(\RR) / G$.
    This gives a description of $M$ from $\Theta(\RR) / G$.
    \end{rmk}

This leaves some interested questions to be investigated in the future.
For example,
can we have a straightforward description for 
the relationship between $M, \Theta(\RR) / G$,
independent of their universal covers $N, \Theta(\RR)$?
See Subsection \ref{geometric reconstruct} for more information.

\section{Further questions} \label{sec:questions}

\subsection{Dictionary between laminar group theory and bifoliated planes}
\cite{baik2022groups}, \cite{BBM}, and our paper have some common aspect which is to develop a dictionary between the theory of (pre or almost) circle laminations and bifoliated planes. In general, it would be an interesting project to expand or even complete this dictionary. 

For an example of possible direction along this line, first note that Section \ref{sec:circle-lamination} describes what a single almost lamination should look like to be the endpoint data of the stable or unstable foliation on the orbit space for a pseudo-Anosov flow. On the other hand, it is not completely clear how a pair of almost laminations exactly looks like to be the endpoint data of the pair of stable/unstable foliations on the orbit space. We know the ideal polygon gaps come in interleaving pairs, but how about cataclysms (see \cite[Theorem 4.9]{Fen99} for the complete picture of cataclysms)? What are exact configurations of the other types of gaps for the pair of almost laminations? 

\begin{que}
    Give a precise description of bifoliar pairs of almost laminations induced by the orbit spaces of (pseudo-)Anosov flows by analyzing precedent work on (pseudo-)Anosov flows such as \cite{Bar95}, \cite{Fen99}. 
\end{que}

\subsection{More investigation on geometric reconstruction}\label{geometric reconstruct}

As mentioned in the introduction, \cite{tukia1988homeomorphic}, \cite{gabai1992convergence}, \cite{casson1994convergence}, \cite{baik2015fuchsian}, \cite{baik2021characterization} are all geometric reconstruction results in the sense that they reconstruct the hyperbolic surface not just a topological surface from the group action at infinity. 

In $3$-dimension, \cite{baik2022groups} is partially a geometric reconstruction in the sense that if the resulting 3-manifold is compact, then it is shown to be hyperbolic. 

Our approach offers a reconstruction from the infinity. 
However, it remains unclear how this reconstruction of a flow in a hyperbolic $3$-manifold can be promoted to a reconstruction in terms of the hyperbolic metric.

Here we briefly review what is known about the geometric structure associated with the ambient $3$-manifold of a pseudo-Anosov flow. Let $\phi$ be a pseudo-Anosov flow in a closed $3$-manifold $M$. $\Or(\phi)$ completely characterizes when $M$ is a toroidal manifold \cite{BF13}. It is known that $\Or(\phi)$ can only be (trivial or skewed) $\R$-covered when $M$ is Seifert fibered \cite{Bar95, BF13}; in this case, $\phi$ is orbitally equivalent to a geodesic flow on $T^{1}S$ for some closed hyperbolic surface $S$, up to finite covers, where $T^{1}S$ is the tangent bundle of $S$. Therefore, by Perelman-Thurston geometrization theorem, $\Or(\phi)$ characterizes when $M$ is hyperbolic if $\Or(\phi)$ is non-$\R$-covered. However, there also exist infinitely many hyperbolic $3$-manifolds with $\R$-covered Anosov flows \cite{Fen94}.

Fenley's work \cite{Fen12, Fen16} provides a description on the hyperbolicity of $3$-manifolds with quasigeodesic pseudo-Anosov flows, in terms of the large scale geometry \cite{Bow99}. Cannon's conjecture provides a framework for relating the hyperbolicity of closed $3$-manifolds from the group actions on the infinity. Theorem \ref{thm: convergence} offers a reconstruction of the $3$-manifolds from the $S^{2}$ ideal boundary, which we hope will contribute to a deeper understanding of the connections between group actions on $\partial \mathbb{H}^3$, hyperbolic $3$-manifolds and their quasi-geodesic pseudo-Anosov flows. This motivates the following questions:

$\bullet$
Can we explicitly construct a complete hyperbolic structure of the $3$-manifold from the constructed flow without appealing to the geometrization? 

$\bullet$
As in Remark \ref{rem: unordered triple},
orientable quasigeodesic pseudo-Anosov flows in hyperbolic $3$-manifolds can be described from $S^{2}_{\infty} \times S^{2}_{\infty} \times S^{2}_{\infty}$ (or $\Theta(S^{2}_{\infty})$ when the flow have no perfect fit), where $S^{2}_{\infty}$ is the ideal $2$-sphere of $\mathbb{H}^{3}$.
How to relate this structure with the hyperbolic metrics of $3$-manifolds?

$\bullet$
For a closed hyperbolic $3$-manifolds with orientable quasigeodesic pseudo-Anosov flow, Theorem \ref{thm: convergence} provides a reconstruction from a sphere-filling Peano curve on $\partial \mathbb{H}^{3}$. Can we generalize this approach to reconstruct the 3-manifold without relying on the given flow and sphere-filling curve?

$\bullet$
More generally, to what extent does the reconstruction in Theorem \ref{thm: convergence} generalize to the problem of reconstructing the (closed) hyperbolic $3$-manifold from the group action on $\partial \mathbb{H}^{3}$?

\subsection{Geometry models for (pseudo-)Anosv flows}

The Perelman-Thurston geometrization theorem says that every closed oriented prime 3-manifold can be cut along tori, so that the interior of each piece has a geometric structure with finite volume modeled on one of the eight geometries. Among those eight model geometries, the hyperbolic geometry is the most central one. As in \cite{Fen12}, the relation between pseudo-Anosov flows and hyperbolic geometry/large scale geometry has been extensively studied. 

In the current work, we consider a new type of geometry models for (pseudo-)Anosov flows.
They are covered by the quotient of subspaces of $\DD \times \DD$ (or $\PP \times \PP$),
where $\DD$ is the Poincar\'e disk and $\PP$ is certain bifoliated plane.

\begin{defn}\rm
Let $\rho: G \to \text{Homeo}(S^{1})$ be a faithful action of $G$ on $S^{1}$ that
preserves a bi-foliar pair of circle laminations $\La,\Lb$.
Then $\rho$ induces an action of $G$ on a bifoliated plane $\PP$,
as well as an action of $G$ on $\PP \times \PP$.

(a)
Call a manifold $M$ \emph{fundamental} if $\PP \times \PP$ has
a $G$-equivariant subspace $N^{'}$ that
has a quotient space $N$ (via a $G$-equivariant equivalence relation on $N^{'}$) and
$M = N / G$.
Let $e: N^{'} \to N$ denote the quotient map.

(b)
Under the assumption of (a),
let $\rho_x = e(\{(x,y) \mid (x,y) \in N^{'}\})$ for each $x \in \PP$.
Suppose that $\{\rho_x \mid x \in \PP\}$ is a $G$-equivariant flow of $N$,
denoted $\w{\phi}$.
Define $\phi = \w{\phi} / G$,
then we call $\phi$ a \emph{fundamental} flow of $M$,

(c)
Under the assumption of (b),
for each leaf $\mu$ of $\Fb$,
define $l_\mu = e(\{(x,y) \mid (x,y) \in N^{'}, y \in \mu\})$.
Suppose that $l_\mu$ is a collection of $2$-planes in $N^{'}$,
and the components of 
$\{l_\mu \mid \mu \text{ is a leaf of } \Fb\}$ is a $G$-equivariant foliation of $N$
(denoted $\w{\F}$).
Let $\F = \w{\F} / G$,
and we call $\F$ a \emph{fundamental foliation} of $M$.
\end{defn}



The motivation for this definition is that our geometry model provides more structural results such as in Subsection \ref{subsec 1.2}. We pose the following questions. 

$\bullet$
We showed the orientable pseudo-Anosov flows are fundamental in the above sense. How to generalize this structure to non-orientable pseudo-Anosov flows?

$\bullet$
Do quasi-geodesic flows in hyperbolic $3$-manifolds have similar properties?

Assume that $\Fa, \Fb$ don't have orientations preserved by the action of $G$.
It's known that our reconstruction can be generalized to this setting when 
$\Fa, \Fb$ have no singularity,
as noted in a personal communication between the last named author and Jonathan Zung.
Below, we provide a sketch of this.

\begin{rmk}[{\cite{ZZ}}]\label{personal communication}
Recall from Definition \ref{A.7},
the action of $G$ on $\PP$ satisfies Property (A7) if
any two distinct points of $\PP$, contained in the same leaf of $\Fa$ (or $\Fb$), 
are not moved by a sufficiently small amount simultaneously under 
the action of any nontrivial element of $G$.
We now introduce a generalization of this property,
called Property (A7$_*$),
which holds if, for any two distinct points $x, y \in \PP$ 
contained in the same leaf of $\Fa$ (or $\Fb$),
there are neighborhoods $U,V$ of $x,y$ respectively, such that
for any $g \in G - \{1\}$,
if $g(x)$ is contained in one of $U,V$,
then $g(y)$ is not contained in the other.
It's not hard to see that Property (A7$_*$) is equivalent to Property (A7) when
$\Fa, \Fb$ have orientations preserved by the action of $G$.
In addition,
both of our proofs of Lemma \ref{property 1} implicitly imply that smooth pseudo-Anosov flows satisfy Property (A7$_*$).

Now assume that the action of $G$ on $\PP$ has Property (A7$_*$).
Because $\Fa, \Fb$ have no singularity,
we can assign them leafwise continuously varying orientations.
As the $G$-action on $\PP$ doesn't preserve these orientations,
$G$ has an index $2$ subgroup $H$ that acts on $\PP$ while preserving
    the orientations on $\Fa, \Fb$.
    For each $x \in \PP$,
    we denote by $\lambda_x, \mu_x$ the leaves of $\Fa, \Fb$ containing $x$ and
    denote by $\lambda^{+}_{x}, \lambda^{-}_{x}$ 
    the two components of $\lambda_x - \{x\}$ in the positive and negative sides of $\lambda_x - \{x\}$.

    Let $N = \{(x,y) \mid x \in \PP, y \in \lambda^{+}_{x}\}$ be
    the space constructed in Theorem \ref{topological flow}.
    We define $g(x,y) = (g(x),g(y))$ if $g \in H$ and
    $g(x,y) = (g(y), g(x))$ if $g \in G - H$.
    By Property (A7$_*$),
    the action of $G$ on $N$ is free and discrete,
    which implies that $N / G$ is a $3$-manifold.
    Denoted this $3$-manifold by $M$.
    We note that $N$ can also be considered as the set of unordered distinct pairs $(x,y)$
    with $y \in \lambda_x$.
    In this setting, $G$ acts on $N$ via the natural induced action.

    Clearly, the flow $\w{\phi} = \{(x,\lambda^{+}_{x}) \mid x \in \PP\}$ 
    is equivariant under the action of $H$.
    However,
    $\w{\phi}$ is not equivariant under elements of $G$;
    for any $g \in G - H$, $g(\w{\phi})$ is the flow 
    $\w{\phi_*} = \{(\lambda^{-}_{x},x) \mid x \in \PP\}$.
    To construct a $G$-equivariant flow of $N$,
    we define a $G$-equivariant projection $e_*: N \to \PP$ such that
    $\{e^{-1}_{*}(x) \mid x \in \PP\}$ is a $G$-equivariant flow of $N$,
    which descends to a flow of $M$.
    In addition,
    $\{e^{-1}_{*}(\lambda) \mid \lambda \text{ is a leaf of } \Fa\}$,
    $\{e^{-1}_{*}(\mu) \mid \mu \text{ is a leaf of } \Fb\}$ are
    a pair of foliations equivariant under the $G$-action,
    which descends to a pair of transverse foliations of $M$.
\end{rmk}

\subsection{Connection to other objects related to pseudo-Anosov flows}

First of all, we can ask if it is possible to construct the foliation as given in subsection \ref{subsec: foliation} without using our geometry model (see the definition of the fundamental foliation in the previous section) but from the based manifold $M$ directly?

Also, what other objects can be translated into our geometry model? For example, by Eliashberg-Thurston \cite{ET98},
all co-orientable taut foliations have a pair of associated contact structures (see \cite{KR17}, \cite{Bow16} for $C^{0}$-foliations). We pose the following questions. 

\begin{que}
    (a) In the case of Anosov flows, how can we describe the associated contact structures of the stable foliations in a fundamental way?

    (b) How can we describe the associated contact structure of the foliation constructed in
    Subsection \ref{subsec: foliation} in a fundamental way?
\end{que}

\subsection{Reconstructions from group equivariant sphere-filling curves}

Theorem \ref{thm: convergence} provides a reconstruction of the manifold and flow from
uniform convergence group actions on $S^{2}$ with an additional assumption that
certain group equivariant sphere-filling Peano curve exists.
Such sphere-filling Peano curves are the images of embedded curves under
the quotient map $S^{2} \to S^{2} / \mathcal{U}$ induced from certain cellular decomposition $\mathcal{U}$ of $S^{2}$.

An important question arises, first introduced to the authors by Danny Calegari:

\begin{que}
    For hyperbolic $3$-manifolds with circular orderable group,
    do they have essential laminations?
\end{que}

Any (reduced) pseudo-Anosov flow in a $3$-manifold gives rise to
a pair of transverse (possibly singular) foliations,
which can be split open to a pair of essential laminations with solid torus guts.
In light of this,
we pose a series of questions regarding this question,
as related to the Cannon's conjecture for circular orderable groups:

\begin{que}
    (a) Let $G$ be a circular orderable group acting on $S^{2}$ 
    as a uniform convergence group action.
    Are there a $G$-equivariant sphere-filling Peano curve on $S^{2}$
    as described above?

    (b) Can we reconstruct a bifoliated plane from this sphere-filling Peano curve?
    If so, can our approach be generalized to reconstructing
    a $3$-manifold with a (quasigeodesic) pseudo-Anosov flow?
\end{que}

We note that not all closed hyperbolic $3$-manifolds have circular orderable fundamental group \cite[Theorem 9.2]{CD03} or admits pseudo-Anosov flows \cite{Fen07}

\bibliographystyle{spmpsci}
\bibliography{refs}

\end{document}